\documentclass[11pt,a4paper]{article}
\usepackage{amsmath, amssymb, enumitem, dsfont}
\usepackage[a4paper, left=2.5cm, right=2.5cm, top=2.5cm, bottom=2.5cm]{geometry}

\usepackage{caption} 
\usepackage{xcolor} 
\usepackage{amsthm,thmtools} 
\usepackage{graphicx}

\usepackage[colorlinks=true, linkcolor=black, citecolor=black, urlcolor=black]{hyperref}
\usepackage{cleveref}
\crefformat{theorem}{{Theorem~#2#1#3}}
\crefformat{corollary}{{Corollary~#2#1#3}}
\crefformat{proposition}{{Proposition~#2#1#3}}
\crefformat{lemma}{{Lemma~#2#1#3}}
\crefformat{definition}{{Definition~#2#1#3}}
\crefformat{remark}{{Remark~#2#1#3}}
\crefformat{assumption}{{Assumption~#2#1#3}}
\crefformat{section}{{Section~#2#1#3}}
\crefformat{subsection}{{Subsection~#2#1#3}}

\definecolor{bluepoli}{cmyk}{0.4,0.1,0,0.4}

\declaretheoremstyle[
  headfont=\color{bluepoli}\normalfont\bfseries,
  bodyfont=\color{black}\normalfont\itshape,
]{colored}

\newtheorem{theorem}{Theorem}[section]

\newtheorem{corollary}[theorem]{Corollary}
\newtheorem{proposition}[theorem]{Proposition}
\newtheorem{lemma}[theorem]{Lemma}
\newtheorem{definition}[theorem]{Definition}
\newtheorem{remark}[theorem]{Remark}

\numberwithin{equation}{section}

\title{\textbf{On a semilinear heat equation on infinite graphs II: blow-up for arbitrary initial data and global existence}}
\author{
Fabio Punzo\thanks{
Dipartimento di Matematica, Politecnico di Milano,
Piazza Leonardo da Vinci 32, 20133 Milano, Italia.
Email: \texttt{fabio.punzo@polimi.it}.
Membro INdAM e parzialmente supportato dal progetto GNAMPA 2026.
}
\and
Federico Zucchero\thanks{
Dipartimento di Matematica, Politecnico di Milano,
Piazza Leonardo da Vinci 32, 20133 Milano, Italia.
Email: \texttt{federico.zucchero@mail.polimi.it}.
}
}

\date{}

\begin{document}

\maketitle


\abstract{This paper is the second part of the study initiated in \cite{PZ1}
and is devoted to finite-time blow-up and global existence for a
semilinear heat equation on infinite weighted graphs. We first establish basic results on mild and classical solutions (which, to the best of our knowledge, were not previously available in the setting of graphs) proving their equivalence under suitable assumptions and showing the existence of a solution between a given sub- and supersolution.
We then analyze blow-up and global existence on $\mathbb Z^N$,
providing proofs based on methods different from those used on $\mathbb Z^N$ in
the existing literature. Moreover, for graphs with positive
spectral gap, we prove global existence for small initial data.
In contrast with previous functional analytic approaches yielding
mild solutions, our method relies on the construction of
global-in-time supersolutions and leads to the existence of
classical solutions.
}

\medskip
{\it Mathematics Subject Classification}: 35A01, 35A02, 35B44, 35K05, 35K58, 35R02

{\it Keywords}: Semilinear parabolic equations, infinite graphs, blow-up, global existence, sub-- supersolutions.

\bigskip
\medskip


\section{Introduction}\label{section_introduction}
In this paper, we continue the study initiated in \cite{PZ1}
concerning finite-time blow-up and global-in-time existence
of solutions to the Cauchy problem
\begin{equation}\label{problem_HR_G}
\begin{cases}
u_t - \Delta u = u^p & \text{in } X \times (0,T), \\
u = u_0 & \text{in } X \times \{0\},
\end{cases}
\tag{1.1}
\end{equation}
where $(X,\omega,\mu)$ is an infinite weighted graph with edge
weight $\omega$ and node measure $\mu$, $T>0$, $p>1$, and
$\Delta$ denotes the weighted Laplacian on the graph.
We refer to \cite{PZ1} for a comprehensive overview of the
literature related to problem \eqref{problem_HR_G} and, more
generally, to the analysis of elliptic and parabolic equations
on combinatorial graphs. We simply refer the reader to the monographs
\cite{Grig1, KLWb, Mu} for a general introduction to the theory
of graphs and to elliptic and parabolic differential equations
on graphs. Furthermore, we only recall here that blow-up and
global existence for semilinear parabolic equations have been
investigated in several works, including
\cite{vCr, GMP2, LW1, LW2, MPS2, PSa, PZ1, Wu}.

In order to properly frame our results, we briefly review some
contributions from the literature. In \cite{LW2}, nonexistence
results for problem \eqref{problem_HR_G} were established under
suitable geometric assumptions on the graph and appropriate
conditions on $p$. In particular, when $X=\mathbb Z^N$, mild
solutions blow up for $p<1+\frac{2}{N}$ for any nontrivial
initial datum, whereas mild solutions exist for
$p>1+\frac{2}{N}$. The nonexistence of solutions is established by following the proof strategy adopted in \cite{BPT, Weiss1}, which relies on two-sided estimates of the heat kernel. Global existence is obtained via the Picard
iteration method in a suitable complete metric space.
In \cite{MPS2}, the blow-up result was extended to the critical
case $p=1+\frac{2}{N}$; by contrast, the proof is based on the test function method combined with suitable a priori estimates. More generally, \cite{MPS2} deals with graphs satisfying suitable assumptions on volume growth and
considers very weak solutions.

Another class of graphs studied in \cite{GMP2, PSa} consists of
those for which the bottom of the $\ell^2$-spectrum of the
Laplacian, denoted by $\lambda_1$, is positive. In this setting, for any $p>1$,
global existence of mild solutions for sufficiently small
initial data was proved by means of the contraction mapping
principle in an appropriate complete metric space.

\smallskip

We would like to emphasize that our approach is entirely independent of the methods developed in the cited papers. 

As a first step, we establish some basic results for the
semilinear problem \eqref{problem_HR_G} which, to the best of
our knowledge, were not previously available in the literature. In particular, under suitable assumptions, we prove
the equivalence between mild and classical solutions. Moreover,
given a supersolution $\bar u$ and a subsolution
$\underline u$ to \eqref{problem_HR_G} satisfying
$0\leq \underline u \leq \bar u$, we prove the existence of a
solution $u$ such that $\underline u \leq u \leq \bar u$.
While these results play an auxiliary role in the proofs of our
main theorems, we believe that they are of independent interest
and may provide useful tools for researchers working on
semilinear equations on graphs.

Next, by exploiting a result from \cite{PZ1}, based on
Kaplan’s method and concerning nonexistence of global
solutions for large initial data, combined with suitable
estimates for the heat semigroup on $\mathbb Z^N$, we prove
finite-time blow-up of solutions to \eqref{problem_HR_G} on
$\mathbb Z^N$ for any $p \leq 1+\frac{2}{N}$, for every nontrivial initial datum $u_0$. Although this
blow-up threshold was already identified in
\cite{LW2, MPS2}, our proof follows a completely different
strategy and does not make use of the arguments developed
therein. In the setting of $\mathbb{R}^N$, the use of Kaplan’s method to study blow-up phenomena, for arbitrary initial data, can be found for instance in \cite{BLev, deP, Weiss2}. To the best of our knowledge, this approach has not previously been applied in the context of graphs, where we adopt it here for the first time.

We then prove global existence of solutions to
\eqref{problem_HR_G} with $X=\mathbb Z^N$ for any $p>1+\frac 2N$, for sufficiently small initial data.
The proof is based on the construction of a global
supersolution and on the comparison framework developed in
the first part of the paper. This provides an alternative approach to global existence, distinct from the fixed-point and successive approximation methods employed in the literature on graphs. Moreover, whereas
previous works established the existence of mild solutions
via functional analytic methods, here we obtain the existence
of classical solutions.

Finally, we consider graphs with $\lambda_1>0$. Also in this
case, we prove global existence of classical solutions for
sufficiently small initial data, for any $p>1$. Our approach differs from
that in \cite{GMP2, PSa}: instead of using a fixed point
argument, we deduce global existence from the construction of
a suitable global supersolution. As above, while the existing
literature deals with mild solutions, we establish the global
existence of classical solutions.

\medskip

The paper is organized as follows. In
Section \ref{section_mathematical_framework}, we introduce
notation and basic definitions concerning graphs. In
Section \ref{section_main_results}, we state our main results,
including both the auxiliary results described above and the
principal theorems on blow-up and global existence.
Section \ref{subsection_H_E_graph_setting} is devoted to the
heat equation on graphs; in particular, we state some
estimates for the heat kernel on $\mathbb Z^N$, essentially
contained in \cite{book_Barlow}. In Section \ref{ardim}, we
prove the auxiliary results concerning mild and classical
solutions. Section \ref{section_proofs} is devoted to the
proofs of blow-up and global existence on $\mathbb Z^N$,
while in Section \ref{subsection_proofs_2} we establish
global existence results for graphs with $\lambda_1>0$.


\section{Mathematical framework}\label{section_mathematical_framework}


\subsection{The graph setting}\label{subsection_graph_setting}

\medskip
\begin{definition}\label{definition_weighted_graph}
Let \( X \) be an infinitely countable set and consider a function \( \mu : X \to (0,+\infty) \). Moreover, let \( \omega : X \times X \to [0,+\infty) \) be a map satisfying the following properties:
\begin{enumerate}
\renewcommand{\labelenumi}{\alph{enumi})}
    \item \( \omega(x,x) = 0 \) for all \( x \in X \);
    \item \( \omega \) is symmetric, that is, \( \omega(x,y) = \omega(y,x) \) for all \( x,y \in X \);
    \item \(\displaystyle \sum_{y \in X} \omega(x,y) < +\infty\) for all \( x \in X \).
\end{enumerate}
Then the triplet \( (X, \omega, \mu) \) is called weighted graph, and the functions \( \mu \) and \( \omega \) are referred to as vertex (or node) measure and edge weight, respectively. Two vertices \( x,y \in X \) are said to be adjacent (or equivalently connected, joint or neighbors) whenever \( \omega(x, y) > 0 \); in this case we write \( x \sim y \), and the pair \( (x,y) \) defines an edge of the graph with endpoints \( x,y \).
\end{definition}

\medskip
\noindent Observe that, since the weight function \( \omega \) determines the edges of a weighted graph, condition a) in \cref{definition_weighted_graph} implies the absence of \emph{self-loops}, that is, there are no edges of type \( (x,x) \). Moreover, condition b) in \cref{definition_weighted_graph} guarantees that the graph is \emph{undirected}, meaning that its edges do not have an orientation.

\medskip
\noindent The definition of the node measure \( \mu \) can be extended to the power set \( \mathcal{P}(X) \) as follows:
\begin{equation}\label{mu_extended}
\begin{split}
\mu(A) &:= \sum_{x \in A} \mu(x) \hspace{1.6em} \text{for every } A \in \mathcal{P}(X) \setminus \left\{ \emptyset \right\}, \\
\mu(\emptyset) &:= 0.
\end{split}
\end{equation}
\noindent Such extension defines a measure; in particular, \( (X, \mathcal{P}(X), \mu) \) is a measure space.

\medskip
\begin{definition}\label{definition_locally_finite}
The weighted graph \( (X, \omega, \mu) \) is said to be
\begin{enumerate}
\renewcommand{\labelenumi}{\alph{enumi})}
    \item locally finite if each vertex has only finitely many neighbors, namely
    \[
    \left| \{ y \in X : y \sim x \} \right| < +\infty, \qquad \quad \text{for all } x \in X;
    \]
    \item connected if for all \( x, y \in X \) there exists a path linking \( x \) to \( y \), namely a list of vertices \( \gamma = \{z_1, z_2, \ldots, z_n\} \subseteq X \) such that:
\vspace{-0.35em}
\[
z_1 = x, \hspace{1.4em} z_n = y \hspace{1.4em} \text{and} \hspace{1.4em} z_i \sim z_{i+1} \hspace{1.2em} \text{for} \hspace{0.5em} i = 1, \ldots, n-1.
\]
\end{enumerate}
\end{definition}

\noindent Throughout the following, we shall always suppose that the weighted graph is both connected and locally finite.

\begin{definition}\label{definition_weighted_degree}
Let \( (X, \omega, \mu) \) be a weighted graph. Then we define, for any vertex \( x \in X \), the degree of \( x \) as
\[
\deg(x) := \sum_{y \in X} \omega(x,y)
\]
\noindent and the weighted degree of \( x \) as
\[
\operatorname{Deg}(x) := \frac{\deg(x)}{\mu(x)} = \frac{1}{\mu(x)} \sum_{y \in X} \omega(x,y).
\]
\end{definition}

\noindent Notice that condition c) in \cref{definition_weighted_graph} ensures that every vertex in a weighted graph has finite degree.


\subsection{Difference and Laplace operators}\label{subsection_difference_Laplace_operators}

\noindent We first introduce the space of all real-valued functions on the vertex set:
\[
C(X) := \{ f : X \to \mathbb{R} \}.
\]

\begin{definition}\label{definition_difference_operators}
Let \( (X, \omega, \mu) \) be a weighted graph. Then:
\begin{enumerate}
\renewcommand{\labelenumi}{\alph{enumi})}
    \item given two vertices \( x,y \in X \), we define the corresponding difference operator as \( \nabla_{xy} \), acting in the following way for any \( f \in C(X) \):
    \[
    \nabla_{xy} f := f(y) - f(x);
    \]

    \item the (weighted) Laplacian acts as \( \Delta: \mathcal{D}_{\Delta}(X) \to C(X) \), where
    \[
    \mathcal{D}_{\Delta}(X) := \left\{ f \in C(X) : \sum_{y \in X} \omega(x,y) |f(y)| < +\infty, \text{ } \forall x \in X \right\},
    \]
    \noindent and is defined, for all \( x \in X \), as:
    \[
    \begin{aligned}
    \Delta f(x) :=& \frac{1}{\mu(x)} \sum_{y \in X} \omega(x,y) \left[ f(y) - f(x) \right] \\
    =& \frac{1}{\mu(x)} \sum_{y \in X} \omega(x,y) \left[ \nabla_{xy} f \right].
    \end{aligned}
    \]
\end{enumerate}
\end{definition}

\noindent Notice that, since by assumption \( \mu > 0 \) on every vertex, then the Laplacian of a function belonging to \( \mathcal{D}_{\Delta}(X) \) is defined over the whole vertex set \( X \). Moreover, it is trivial to verify that \( \Delta \) is a linear operator.

\begin{remark}\label{remark_locally_finite_laplacian}
It is straightforward to verify that, for any locally finite weighted graph \( (X, \omega, \mu) \), the following identity holds:
\[
C(X) = \mathcal{D}_{\Delta}(X),
\]
\noindent so that the Laplacian is well defined for any function \( f \in C(X) \).
\end{remark}


\subsection{The combinatorial distance}\label{subsection_combinatorial_distance}

\noindent Throughout the following, we shall adopt the notation \( \mathbb{N}_0 \) to denote the set \( \mathbb{N} \cup \{ 0 \} \).

\begin{definition}\label{definition_combinatorial_distance}
Let \( (X,\omega,\mu) \) be a connected weighted graph. We define the combinatorial graph distance \( \rho : X \times X \to \mathbb{N}_0 \) as follows: for any pair of vertices \( x, y \in X \), \( \rho(x,y) \) denotes the minimal number of edges in a path connecting \( x \) to \( y \). In other words, we set:
\[
\rho(x,y) := \inf \left\{ n : \gamma = \{x_k\}_{k=0}^n \text{ is a path between } x \text{ and } y \right\}, \hspace{1.8em} \text{for all } x,y \in X.
\]
\noindent Furthermore, given a nonempty finite subset \( \Omega \subset X \), we define the distance from any vertex of the graph to \( \Omega \) as
\[
\rho(x,\Omega) := \min_{y \in \Omega} \hspace{0.1em} \rho(x,y) \hspace{1.8em} \text{for all } x \in X.
\]
\noindent Within this framework, for any \( r \in \mathbb{N}_0 \), we refer to the shell of radius \( r \) centered at \( \Omega \) as
\[
S_r(\Omega) := \left\{ x \in X : \rho(x, \Omega) = r \right\}.
\]
\end{definition}

\noindent We point out a slight abuse of notation in the preceding definition: for a fixed vertex \( x \in X \), we write \( \rho(x,\cdot) \) both for the distance from \( x \) to another vertex of the graph and for the distance from \( x \) to the finite set \( \Omega \subset X \). \noindent The meaning will always be clear from the context.

\medskip
\noindent We also observe that the combinatorial distance is well defined on \( X \times X \), since the weighted graph \( (X,\omega,\mu) \) is assumed to be connected. Moreover, it is straightforward to verify that \( \rho \) actually defines a metric on \( X \).

\medskip
\noindent For notational convenience, we shall set
\begin{equation}\label{definition_r_rho}
r(x) := \rho(x,\Omega) \hspace{1.8em} \text{for all } x \in X.
\end{equation}


\subsection{Functional spaces}\label{subsection_functional_spaces}

\noindent For any \( p \in [1,+\infty] \), we define the weighted sequence spaces as
\[
\begin{split}
\ell^p(X, \mu) := & \left\{ f \in C(X) : \sum_{x \in X} |f(x)|^p \, \mu(x) < +\infty \right\}, \quad \text{whenever } p \in [1,+\infty), \\
\ell^{\infty}(X, \mu) \equiv & \, \, \ell^{\infty}(X) := \left\{ f \in C(X) : \sup_{x \in X} |f(x)| < +\infty \right\},
\end{split}
\]
\noindent and equip them with the following norms:
\[
\begin{split}
&\|f\|_p := \left[ \sum_{x \in X} |f(x)|^p \, \mu(x) \right]^{1/p} \qquad \text{for all } p \in [1,+\infty), \\
&\|f\|_{\infty} := \sup_{x \in X} |f(x)|.
\end{split}
\]


\subsection{Homogeneous model trees}\label{subsection_model_trees}

\begin{definition}\label{definition_model_tree}
Let \( b \in \mathbb{N} \) and consider a connected weighted graph \( (\mathbb{T}_b,\omega_0,\mu_1) \), together with a reference vertex \( x_0 \in \mathbb{T}_b \). Let us define the shells of radius \( r \in \mathbb{N}_0 \) as
\[
S_r(x_0) := \left\{ x \in \mathbb{T}_b : \rho(x,x_0) = r \right\},
\]
\noindent where \( \rho \) denotes the combinatorial distance, introduced in \emph{\cref{definition_combinatorial_distance}}. Assume that the following holds:
\begin{enumerate}
\renewcommand{\labelenumi}{\alph{enumi})}
    \item the edge weight \( \omega_0 : \mathbb{T}_b \times \mathbb{T}_b \to \{0,1\} \) is given by
    \begin{equation}\label{definition_omega_0}
    \omega_0(x,y) :=
    \begin{cases}
    1 & \text{ if } x \sim y \\
    0 & \text{ otherwise;}
    \end{cases}
    \end{equation}
    \item there are no edges connecting vertices within the same shell, i.e., for any \( r \in \mathbb{N}_0 \) we have:
    \[
    \omega_0(x,y) = 0 \hspace{1.8em} \text{for all } (x,y) \in S_r(x_0) \times S_r(x_0);
    \]
    \item \( \mu_1(x) \equiv 1 \) for every \( x \in \mathbb{T}_b \);
    \item it holds \( \deg(x_0) = b \);
    \item for all \( x \in \mathbb{T}_b \setminus \left\{ x_0 \right\} \) we have:
    \[
    \begin{aligned}
    & \left| \{ y \in \mathbb{T}_b : y \sim x, \rho(y,x_0) = \rho(x,x_0) - 1 \} \right| = 1, \\
    & \left| \{ y \in \mathbb{T}_b : y \sim x, \rho(y,x_0) = \rho(x,x_0) + 1 \} \right| = b.
    \end{aligned}
    \]
\end{enumerate}
\noindent Then \( (\mathbb{T}, \omega_0, \mu_1) \) is said to be a homogeneous model tree with root \( x_0 \) and branching \( b \).
\end{definition}

\noindent With respect to \cref{definition_combinatorial_distance}, we remark that a slight abuse of notation has been adopted in the definition of the shells: for all \( r \in \mathbb{N}_0 \), we write \( S_r(x_0) \) instead of \( S_r(\{x_0\}) \).

\medskip
\noindent For notational convenience, we shall adapt, in the present case where \( \Omega = \{x_0\} \), the convention introduced in \eqref{definition_r_rho}, namely:
\[
r(x) := \rho(x,x_0), \hspace{1.8em} \text{for all } x \in \mathbb{T}_b.
\]

\begin{remark}\label{remark_model_trees_locally_finite}
\noindent We observe that, since in a homogeneous model tree the edge weight is given by \eqref{definition_omega_0}, then for every \(x \in \mathbb{T}_b\) we can write:
\[
\deg(x) = \sum_{y \in \mathbb{T}_b} \omega_0(x,y) = \sum_{\substack{y \in \mathbb{T}_b \\ y \sim x}} \omega_0(x,y)
= \sum_{\substack{y \in \mathbb{T}_b \\ y \sim x}} 1 = \left| \{ y \in \mathbb{T}_b : y \sim x \} \right|.
\]
\noindent In other words, the degree of a vertex coincides with the number of its adjacent vertices. In addition, by combining properties b), d), and e) in \emph{\cref{definition_model_tree}}, it is straightforward to verify that for all \( x \in \mathbb{T}_b \) it holds:
\[
\deg(x) =
\begin{cases}
b & \text{ if } \hspace{0.05em} x = x_0 \\
b + 1 & \text{ if } \hspace{0.05em} x \in \mathbb{T}_b \setminus \{ x_0 \}.
\end{cases}
\]
\noindent In conclusion, every node has finitely many neighbors, meaning that any homogeneous model tree is locally finite.
\end{remark}


\subsection{The integer lattice}\label{subsection_integer_lattice}

\begin{definition}\label{definition_lattice}
Let \( N \in \mathbb{N} \), and consider a connected weighted graph fulfilling the following properties:
\begin{enumerate}
\renewcommand{\labelenumi}{\alph{enumi})}
    \item the vertex set coincides with \( \mathbb{Z}^N \), that is, each vertex \( x \) is composed of \( N \) integer components:
    \[
    x = (x_1, x_2, \dots, x_N), \hspace{1.8em} \text{with } x_i \in \mathbb{Z} \text{ for each } i \in \left\{ 1, 2, \dots, N \right\};
    \]
    \item for each couple of vertices \( x,y \in \mathbb{Z}^N \), we have \( x \sim y \) if and only if there exists \( k \in \left\{ 1, 2, \dots, N \right\} \) such that
    \[
    y_k = x_k \pm 1 \hspace{1.2em} \text{and} \hspace{1.2em} y_i = x_i  \hspace{0.7em} \text{for each } i \in \left\{ 1, 2, \dots, N \right\} \setminus \{ k \};
    \]
    \item the edge weight \( \omega_0 \) satisfies the definition given in \eqref{definition_omega_0}, namely \( \omega_0 : \mathbb{Z}^N \times \mathbb{Z}^N \to \{0,1\} \) and
    \[
    \omega_0(x,y) :=
    \begin{cases}
    1 & \text{ if } x \sim y \\
    0 & \text{ otherwise;}
    \end{cases}
    \]
    \item the node measure is defined as
    \[
    \mu(x) \equiv 2N \hspace{1.8em} \text{for every } x \in \mathbb{Z}^N.
    \]
\end{enumerate}
\noindent Then the corresponding weighted graph \( (\mathbb{Z}^N,\omega_0,\mu) \) is referred to as the  \( N \)--dimensional integer lattice.
\end{definition}

\medskip
\noindent We observe that, since in the lattice the edge weight is given by \eqref{definition_omega_0}, then we can argue as in \cref{remark_model_trees_locally_finite}, obtaining the following identity, valid for every \( x \in \mathbb{Z}^N \):
\[
\deg(x) = \sum_{y \in \mathbb{Z}^N} \omega_0(x,y) = \left| \{ y \in \mathbb{Z}^N : y \sim x \} \right|.
\]
\noindent Therefore, the degree of a vertex coincides with the number of its adjacent vertices. In addition, it can be easily verified that every node counts exactly \( 2N \) neighbors; we then infer that the integer lattice is locally finite and that the following identities hold:
\[
\mu(x) = \deg(x) = \sum_{y \in \mathbb{Z}^N} \omega_0(x,y) = \left| \{ y \in \mathbb{Z}^N : y \sim x \} \right| = 2N, \hspace{2.0em} \text{for all } x \in \mathbb{Z}^N.
\]
\noindent Finally, in the context of the integer lattice, it is not difficult to verify the validity of the following expression for the combinatorial distance \( \rho \):
\begin{equation}\label{definition_rho_lattice}
\rho(x,y) = \sum_{i=1}^N |x_i - y_i| \hspace{1.8em} \text{for all } x,y \in \mathbb{Z}^N.
\end{equation}
\noindent Despite this convenient form, for analytical purposes it turns out to be preferable to endow the lattice with the Euclidean distance.

\begin{definition}\label{definition_euclidean_distance_lattice}
\noindent We equip the integer lattice \( (\mathbb{Z}^N,\omega_0,\mu) \) with the Euclidean distance \( d : \mathbb{Z}^N \times \mathbb{Z}^N \to [0,+\infty) \), defined as follows:
\[
d(x,y) := \left[ \sum_{i=1}^N |x_i - y_i|^2 \right]^{\frac{1}{2}} \hspace{1.8em} \text{for all } x,y \in \mathbb{Z}^N,
\]
\noindent for which we will also use the notation \( \left| x-y \right| \).

\noindent Moreover, given a reference vertex \( x_0 \in \mathbb{Z}^N \) and \( r > 0 \), we define the ball of radius \( r \) centred at \( x_0 \) as
\[
B_r(x_0) := \{ x \in \mathbb{Z}^N : d(x,x_0) < r \},
\]
\noindent having boundary
\[
\partial B_r(x_0) = \{ x \in \mathbb{Z}^N : d(x,x_0) = r \}.
\]
\end{definition}

\medskip
\noindent We shall also recall the equivalence, in the context of the integer lattice, between the combinatorial distance \( \rho \) and the Euclidean distance \( d \). Indeed, it holds:
\begin{equation}\label{equivalence_rho_d__lattice}
d(x,y) \leq \rho(x,y) \leq \sqrt{N} \, d(x,y) \hspace{2.5em} \text{for all } x,y \in \mathbb{Z}^N.
\end{equation}








\section{Statement of the results}\label{section_main_results}
\subsection{Main results}
\noindent As already specified in \cref{subsection_graph_setting}, we shall always assume the weighted graph \( (X,\omega,\mu) \) to be both connected and locally finite.

\medskip
\noindent In addition, we suppose that the initial datum \( u_0 \) of problem \eqref{problem_HR_G} is nonnegative and bounded, namely:
\begin{equation}\label{eq:HP_ID_G}
u_0 : X \to [0,+\infty) \hspace{1.6em} \text{and} \hspace{1.6em} u_0 \in \ell^\infty(X).
\end{equation}

We establish the following blow-up and global existence results on the lattice $\mathbb{Z}^N$.

\begin{theorem}\label{theorem_blow_up_lattice_sub_critical_p}
Let \( 1 < p \leq 1 + \frac{2}{N} \). Let $u$ be a solution of problem \eqref{problem_HR_G} with $X=\mathbb Z^N$ and $u_0\not\equiv 0$. Then $u$ blows up in finite time.
\end{theorem}

\begin{theorem}\label{theorem_global_existence_lattice}
If\hspace{0.1em} \( p > 1 + \frac{2}{N} \), then problem \eqref{problem_HR_G}, posed on the integer lattice, admits a global classical solution, provided that the initial datum \( u_0 \in C(\mathbb{Z}^N) \) satisfies the following estimates:
\begin{equation}\label{bound_ID_global_existence_lattice}
0 \leq u_0(x) \leq \overline{u}(x,0) \hspace{2.5em} \text{for all } \hspace{0.1em} x \in \mathbb{Z}^N,
\end{equation}
\noindent where \( \overline{u} \) is a suitably constructed global classical supersolution of the equation
\begin{equation}\label{e30}
u_t - \Delta u = u^p \quad \text{ in } \text{ } \mathbb{Z}^N\times (0,\infty).
\end{equation}

\noindent In particular, the solution is nontrivial if \( u_0 \not \equiv 0 \).
\end{theorem}

Let $\sigma(-\Delta)$ denote the $\ell^2$-spectrum of the operator $-\Delta$.
It is well known that $\sigma(-\Delta) \subseteq [0,\infty)$.
We further define
\[
\lambda_1(X) := \inf \sigma(-\Delta),
\]
that is, the bottom of the spectrum.

On graphs for which $\lambda_1>0$, we obtain the following global existence result. In this setting, global existence holds for every $p>1.$

\begin{theorem}\label{theorem_global_existence_G}
Let \( (X,\omega,\mu) \) be a connected and locally finite weighted graph, and assume that the weighted degree is bounded, namely
\begin{equation}\label{bounded_weighted_degree_2}
D := \sup_{x\in X} \, \operatorname{Deg}(x) = \sup_{x\in X} \, \frac{1}{\mu(x)} \, \sum_{y\in X} \, \omega(x,y) < +\infty.
\end{equation}
Suppose also that \( \lambda_1(X) > 0 \), and
\[
\mu_{min} := \inf_{x \in X} \mu(x) > 0.
\]
\noindent Then, for all \( p > 1 \), problem \eqref{problem_HR_G} admits a global classical solution, provided that the initial datum \( u_0 \in C(X) \) satisfies the following estimate:
\begin{equation}\label{bound_ID_global_existence_G}
0 \leq u_0(x) \leq \overline{u}(x,0) \hspace{2.5em} \text{for all } \hspace{0.1em} x \in X,
\end{equation}
where \( \overline{u} \) is a suitably constructed global classical supersolution of the equation
\begin{equation}\label{e31}
u_t - \Delta u = u^p \quad \text{ in } \text{ } X\times(0,\infty).
\end{equation}

\end{theorem}

As a special case of graphs with $\lambda_1>0$, we now consider homogeneous trees $(\mathbb{T}_b,\omega_0,\mu_1)$ with branching \( b \geq 2 \). Indeed, in this scenario we have (see, e.g., \cite{FR}):
\begin{equation}\label{e20}
\lambda_1(\mathbb{T}_b) = \left( \sqrt{b} - 1 \right)^2 > 0.
\end{equation}

From Theorem \ref{theorem_global_existence_G} we deduce the next result.

\begin{corollary}\label{theorem_homogeneous_model_trees_existence}
Let \( (\mathbb{T}_b,\omega_0,\mu_1) \) be a homogenous model tree with branching \( b \geq 2 \). 
Then, for all \( p > 1 \), problem \eqref{problem_HR_G} with $X=\mathbb{T}_b$ admits a global classical solution, provided that the initial datum \( u_0 \in C(\mathbb{T}_b) \) satisfies the following estimate:
\[
0 \leq u_0(x) \leq \overline{u}(x,0) \hspace{2.5em} \text{for all } \hspace{0.1em} x \in \mathbb{T}_b,
\]
where \( \overline{u} \) is a suitably constructed global classical supersolution of equation \eqref{e31} with $X=\mathbb{T}_b.$

\end{corollary}

\subsection{Auxiliary results}
\subsubsection{Definition of classical and mild solutions}
\noindent We are now ready to introduce the notion of solution for problem \eqref{problem_HR_G}. In particular, we restrict our attention to nonnegative \emph{classical} solutions that remain bounded for all times prior to \( T \).

\medskip
Let $S_T:=X\times (0, T).$

\begin{definition}\label{definition_sub_super_sols_G}
Let \( u_0 \) satisfy \eqref{eq:HP_ID_G} and consider a function \( u : X \times [0,T) \to \mathbb{R} \) such that
\begin{equation}\label{reg_sol_1_G}
  u(x,\cdot) \in C^1((0,T)) \cap C^0([0,T)) \hspace{1.6em} \text{for all } x \in X,
\end{equation}

\vspace{-1.5em}

\begin{equation}\label{reg_sol_2_G}
  u \in L^\infty(X \times [0,T']) \quad \text{ for all } \text{ } T' \in (0,T).
\end{equation}
\noindent Then \( u \) is said to be a solution of problem \eqref{problem_HR_G}, with initial condition \( u_0 \), if \( u \) satisfies \eqref{problem_HR_G} pointwise, and \( u \geq 0 \) in \( S_T \). \\
A function \( \overline{u}: X \times [0,T) \to \mathbb{R} \) is said to be a supersolution of problem \eqref{problem_HR_G}, associated with the initial datum \( u_0 \), if it satisfies \eqref{reg_sol_1_G} and \eqref{reg_sol_2_G}, and if it solves pointwise the following inequalities:
\[
  \begin{cases}
  \overline{u}_t - \Delta \overline{u} \geq \overline{u}^p & \text{ in } \hspace{0.05em} S_T \\
  \overline{u} \geq u_0 & \text{ in } \hspace{0.05em} X \times \{0\} \\
  \overline{u} \geq 0 & \text{ in } \hspace{0.05em} S_T.
  \end{cases}
\]

\noindent Analogously, a function \( \underline{u}: X \times [0,T) \to \mathbb{R} \) is said to be a subsolution of problem \eqref{problem_HR_G}, with initial condition \( u_0 \), if it satisfies \eqref{reg_sol_1_G} and \eqref{reg_sol_2_G}, and if it solves pointwise the following inequalities:
\[
  \begin{cases}
  \underline{u}_t - \Delta \underline{u} \leq \underline{u}^p & \text{ in } \hspace{0.05em} S_T \\
  \underline{u} \leq u_0 & \text{ in } \hspace{0.05em} X \times \{0\} \\
  \underline{u} \geq 0 & \text{ in } \hspace{0.05em} S_T.
  \end{cases}
\]
\end{definition}

\noindent From the previous definition, it is immediate to observe that a function is a solution to problem \eqref{problem_HR_G} if and only if it is both a subsolution and a supersolution.

\medskip
\noindent In addition, for future reference, we clarify the notion of supersolution to the semilinear equation under consideration. In particular, given \( T \in (0,+\infty] \), we say that a function \( \overline{u}: X \times [0,T) \to \mathbb{R} \) is a \emph{supersolution of the equation} \( u_t - \Delta u = u^p \) if it satisfies all the requirements of a supersolution to problem \eqref{problem_HR_G}, according to \cref{definition_sub_super_sols_G}, except for the inequality involving the initial datum \( u_0 \). Such supersolution is said to be \emph{global} if the corresponding maximal existence time is \( T = +\infty \).

\begin{remark}\label{remark_time_translation_G}
\noindent Consider a solution \( u \) to problem \eqref{problem_HR_G}, defined in \( X \times [0,T) \). Then, after fixing an arbitrary time \( t^* \in (0,T) \) and setting \( u_0^* := u(\cdot, t^*) \), it is immediate to verify that, thanks to the regularity condition \eqref{reg_sol_2_G}, \( u_0^* \) satisfies \eqref{eq:HP_ID_G}, hence being a valid initial datum for the following problem, solved by \( u \):
\[
\begin{cases}
u_t - \Delta u = u^p & \emph{\text{ in }} X \times (t^*,T) \\
u = u_0^* & \emph{\text{ in }} X \times \{ t^* \} \\
u \geq 0 & \emph{\text{ in }} X \times (t^*,T).
\end{cases}
\]
\end{remark}

\noindent We also highlight that, by arguing as in the previous remark, \eqref{reg_sol_2_G} allows us to deduce that any solution \( u \) of problem \eqref{problem_HR_G} satisfies
\[
u(\cdot,t) \in \ell^\infty(X) \hspace{1.6em} \text{for each } t \in (0,T),
\]

\noindent and two mutually exclusive possibilities arise: either the maximal existence time is \( T = +\infty \), in which case the function \( u(\cdot,t) \) remains bounded in \( X \) for all times \( t > 0 \), or \( T \in (0,+\infty) \), and \( \|u(\cdot,t)\|_\infty \) exhibits a vertical asymptote at time \( T \). This dichotomy motivates the introduction of the following definition.

\begin{definition}\label{blow_up_global_G}
Let \( u \) be a solution of problem \eqref{problem_HR_G}. Then:
\begin{enumerate}
\renewcommand{\labelenumi}{\alph{enumi})}

  \item we say that \( u \) blows up in finite time, or equivalently that \( u \) is a nonglobal solution of problem \eqref{problem_HR_G}, if the corresponding maximal time of existence is finite, i.e., \( 0 < T < +\infty \), with
\[
u(\cdot,t) \in \ell^\infty(X) \hspace{1.2em} \text{for all } \text{ } t \in (0,T) \hspace{1.6em} \text{and} \hspace{1.6em} \lim_{t \to T^-} \|u(\cdot, t)\|_\infty = +\infty;
\]

\vspace{-1.18em}
  \item we say that \( u \) is a global solution of problem \eqref{problem_HR_G} if the corresponding maximal time of existence is \( T = +\infty \) and
\[
u(\cdot, t) \in \ell^\infty(X) \hspace{1.0em} \text{for all } \hspace{0.15em} t > 0.
\]
\end{enumerate}
\end{definition}

\noindent Throughout the following, we shall also deal with \emph{mild} solutions, subsolutions, and supersolutions to our problem. Nevertheless, in what follows, unless explicitly stated otherwise, the term \emph{solution} always refers to a \emph{classical} solution, in the sense of \cref{definition_sub_super_sols_G}. Whenever a result is formulated for a \emph{mild} solution, this will be clearly indicated. The same convention applies to subsolutions and supersolutions.

Let \( \{ P_t \}_{t \geq 0} \) be the {\it heat semigroup} on $X$ (see Section \ref{subsection_H_E_graph_setting} below).

\begin{definition}\label{definition_mild_sol_sub_super_sol}
Let \( \tau \in (0,+\infty) \), and consider an initial datum \( u_0 \), for problem \eqref{problem_HR_G}, satisfying \eqref{eq:HP_ID_G}.

\medskip
\noindent A function \( u : X \times [0,\tau] \to [0,\infty) \) is called a mild solution of problem \eqref{problem_HR_G} on \( [0,\tau] \) if \( u \in L^\infty([0,\tau]; \ell^\infty(X)) \) and if it holds
\begin{equation}\label{eq:nonlinear-mild}
u(\cdot,t) = P_t \, u_0 + \int_0^t \, P_{t-s}  \left[ u(\cdot,s)^p \right] \, ds \qquad \quad \text{for each } \hspace{0.1em} t \in (0,\tau],
\end{equation}
together with \( u(\cdot,0) = u_0 \), pointwise in \( X \). Such solution is said to be global if the previous equality holds for all \( \tau > 0 \), namely for all \( t \in (0,+\infty) \).

\medskip
\noindent A function \( \underline{u} : X \times [0,\tau] \to [0,\infty) \) is called a mild subsolution of problem \eqref{problem_HR_G} on \( [0,\tau] \) if \( \underline{u} \in L^\infty([0,\tau]; \ell^\infty(X)) \) and if it holds
\[
\underline{u}(\cdot,t) \leq P_t \, u_0 + \int_0^t \, P_{t-s} \left[ \underline{u}(\cdot,s)^p \right] \, ds \qquad \quad \text{for each } \hspace{0.1em} t \in (0,\tau],
\]
together with \( \underline{u}(\cdot,0) \leq u_0 \), pointwise in \( X \). Such function is said to be global mild subsolution if the previous inequality holds for all \( \tau > 0 \), namely for all \( t \in (0,+\infty) \).

\medskip
\noindent A function \( \overline{u} : X \times [0,\tau] \to [0,\infty) \) is called a mild supersolution of problem \eqref{problem_HR_G} on \( [0,\tau] \) if \( \overline{u} \in L^\infty([0,\tau]; \ell^\infty(X)) \) and if it holds
\[
\overline{u}(\cdot,t) \geq P_t \, u_0 + \int_0^t \, P_{t-s} \left[ \overline{u}(\cdot,s)^p \right] \, ds \qquad \quad \text{for each } \hspace{0.1em} t \in (0,\tau],
\]
together with \( \overline{u}(\cdot,0) \geq u_0 \), pointwise in \( X \). Such function is said to be global mild supersolution if the previous inequality holds for all \( \tau > 0 \), namely for all \( t \in (0,+\infty) \).
\end{definition}


\begin{remark}\label{remark_mild_well_posed}
(i)
Given \( \tau \in (0,+\infty) \), assume that $u \in L^\infty([0,\tau];\ell^\infty(X))$ and that \eqref{eq:nonlinear-mild}
holds $\mu$-a.e.\ in $X$ and for a.e.\ $t\in(0,\tau)$.
Since $\mu(x)>0$ for every $x\in X$, any $\mu$-null subset of $X$ is necessarily empty; indeed, the \( \mu \)-extension given by \eqref{mu_extended} ensures that, for all \( A \subseteq X \), it holds \( \mu(A) = 0 \) if and only if \( A = \emptyset \). This implies that any equality holding $\mu$-a.e.\ in $X$ actually holds pointwise in $X$.
In particular, in this scenario, the validity of \eqref{eq:nonlinear-mild} ensures that $u(x,t)$ is defined for every $x\in X$ and for a.e.\ $t\in(0,\tau)$.

Moreover, the right-hand side of \eqref{eq:nonlinear-mild} defines, for each fixed $x\in X$, a continuous function of $t\in[0,\tau]$. This follows from standard semigroup arguments and dominated convergence, using the main properties of the heat kernel \( p \) and of the heat semigroup $\{P_t\}_{t\ge 0}$, collected, respectively, in \emph{\cref{proposition_properties_HK_G}} and in \emph{\cref{subsection_heat_semigroup}} below.

Therefore $u$ admits a representative (coinciding with $u$ a.e.\ in $X\times(0,\tau)$)
which is defined for every $(x,t)\in X\times[0,\tau]$, satisfies the mild identity \eqref{eq:nonlinear-mild} pointwise in \( X\times[0,\tau] \), and is continuous in time for each fixed $x$.

In particular, these considerations show that, even if one does not a priori require a mild solution to be defined pointwise in space and time, this property actually follows naturally within our setting.

\smallskip
(ii) The integral appearing in \eqref{eq:nonlinear-mild} is therefore intended pointwise in $X$.
More precisely, for each fixed $x \in X$ and $t \in (0,\tau]$, the identity
\begin{equation}\label{eq:Duhamel_G}
u(x,t) = (P_t \, u_0)(x) + \int_0^t \left( P_{t-s} \, \left[u(\cdot,s) \right]^p \right)(x) \, ds,
\end{equation}
corresponding to \eqref{eq:nonlinear-mild} when evaluated at \( x \), is understood as an equality between real numbers, where the time integral is taken
in the usual Lebesgue sense. In particular, the equality in \eqref{eq:nonlinear-mild}
is required to hold for every $t \in (0,\tau]$, not merely for almost every time.

We stress that no continuity in time is assumed in the definition of mild solution. However (see \emph{\cref{lemma2}} below), if the weighted degree is uniformly bounded, namely \eqref{bounded_weighted_degree_2} holds, then the Laplacian $\Delta$ is a bounded operator on $\ell^\infty(X)$ and the associated heat semigroup is uniformly continuous on $\ell^\infty(X)$. As a consequence, every mild solution in the above sense belongs in fact to $C([0,\tau];\ell^\infty(X))$.
\end{remark}

\subsubsection{Basic properties of classical and mild solutions}

Let \( p \) denote the {\it heat kernel} on $X$ (see Section \ref{subsection_H_E_graph_setting} below).

\begin{definition}\label{definition_stochastically_complete_graph}
\noindent Let \( (X,\omega,\mu) \) be a connected and locally finite weighted graph, and consider the heat kernel \( p \) on \( (X,\omega,\mu) \). Then \( (X,\omega,\mu) \) is said to be stochastically complete if
\[
\sum_{y \in X} p(x,y,t) \, \mu(y) = 1 \hspace{2.5em} \text{for all } \hspace{0.1em} (x,t) \in X \times (0,+\infty).
\]
\noindent Otherwise, the weighted graph \( (X,\omega,\mu) \) is said to be stochastically incomplete.
\end{definition}

\noindent Observe that, at time \( t = 0 \), the equality above is automatically satisfied, since (see \cref{proposition_properties_HK_G} below) the heat kernel \( p \) satisfies:
\[
\sum_{y \in X} p(x,y,0) \, \mu(y) = \sum_{y \in X} \delta_{xy} = \delta_{xx} \equiv 1 \hspace{2.5em} \text{for all } x \in X.
\]

\noindent For future reference, we recall that, as it is well-known, the integer lattice \( (\mathbb{Z}^N,\omega_0,\mu) \) is stochastically complete, for each \( N \in \mathbb{N} \).

\medskip
The following two results address the relationship between mild solutions and classical solutions.

\begin{proposition}\label{proposition_Duhamel_formula_HR_G}
\noindent Suppose that \( (X,\omega,\mu) \) is connected, locally finite and stochastically complete and that \( u_0 \) satisfies \eqref{eq:HP_ID_G}. Let \( u : X \times [0,+\infty) \to \mathbb{R} \) be a global classical solution to problem \eqref{problem_HR_G}, with initial condition \( u_0 \). Then \( u \) is a global mild solution of problem \eqref{problem_HR_G}.
\end{proposition}

\begin{proposition}\label{proposition_mild_implies_classical_global_general_graph}
Let \( (X,\omega,\mu) \) be a connected and locally finite weighted graph, and assume that the weighted degree is bounded, namely
\eqref{bounded_weighted_degree_2} holds. Then every global mild solution of problem \eqref{problem_HR_G}, according to \emph{\cref{definition_mild_sol_sub_super_sol}}, is in particular a global classical solution, in the sense of \emph{\cref{definition_sub_super_sols_G}}.
\end{proposition}

\bigskip

The following two propositions establish the existence of a solution to problem {\eqref{problem_HR_G} lying between a given ordered subsolution and supersolution. Again, regarding the properties of the heat semigroup \( \{ P_t \}_{t \geq 0} \) appearing in the following statement, we refer to Section \ref{subsection_H_E_graph_setting}.

\begin{proposition}\label{proposition_existence_between_barriers_mild}
Suppose that the heat semigroup \( \{ P_t \}_{t \geq 0} \) is positivity preserving, and that it is an \( \ell^\infty \)-contraction (see \emph{\cref{definition_S_1_S_2_semigroup}} below). Let \( \tau \in  (0,+\infty) \), and consider an initial datum \( u_0 \), for problem \eqref{problem_HR_G}, satisfying \eqref{eq:HP_ID_G}. Assume that \( \underline{u} \) and \( \overline{u} \) are, respectively, a mild subsolution and a mild supersolution of problem \eqref{problem_HR_G} on \( [0,\tau] \). In addition, suppose that the following pointwise inequality holds:
\[
\underline{u}(\cdot,t) \leq \overline{u}(\cdot,t) \qquad \quad \text{for all } \hspace{0.1em} t \in [0,\tau].
\]
Then there exists a mild solution \( u \) to problem \eqref{problem_HR_G} on \( [0,\tau] \), satisfying the following pointwise bounds:
\[
\underline{u}(\cdot,t) \leq u(\cdot,t) \leq \overline{u}(\cdot,t) \qquad \quad \text{for all } \hspace{0.1em} t \in [0,\tau].
\]
\end{proposition}

\begin{proposition}\label{proposition_summarize_global_existence}
Let \( (X,\omega,\mu) \) be a connected and locally finite weighted graph, and assume that the weighted degree is bounded, namely \eqref{bounded_weighted_degree_2} holds. Consider an initial datum \( u_0 \), for problem \eqref{problem_HR_G}, satisfying \eqref{eq:HP_ID_G}. Assume that \( \underline{u} \) and \( \overline{u} \) are, respectively, a global mild subsolution and a global mild supersolution of problem \eqref{problem_HR_G}, such that the following inequality holds pointwise in \( X \):
\[
\underline{u}(\cdot,t) \leq \overline{u}(\cdot,t) \qquad \quad \text{for all } \hspace{0.1em} t \in [0,+\infty).
\]
Then there exists a global classical solution \( u \) to problem \eqref{problem_HR_G}, satisfying the following pointwise bounds:
\[
\underline{u}(x,t) \leq u(x,t) \leq \overline{u}(x,t) \qquad \text{for all } \hspace{0.1em} (x,t) \in X \times [0,+\infty).
\]
\end{proposition}

\section{The heat equation on infinite graphs}\label{subsection_H_E_graph_setting}

\noindent Consider the following global Cauchy problem for the unknown \( v = v(x,t) \):
\begin{equation}\label{problem_H_G}
\begin{cases}
v_t - \Delta v = 0 & \text{ in } X \times (0, +\infty) \\
v = v_0 & \text{ in } X \times \{0\}. \\
\end{cases}
\end{equation}
\noindent Let us assume the initial datum \( v_0 \) to be bounded, that is,
\[
v_0 \in \ell^\infty(X).
\]
\noindent Problem \eqref{problem_H_G} admits a \emph{heat kernel} \( p : X \times X \times [0,+\infty) \to \mathbb{R} \). In particular, \( p = p(x,y,t) \) is defined as the smallest nonnegative function which solves the following problem:
\begin{equation}\label{problem_solved_by_p}
\begin{cases}
p_t - \Delta p = 0 & \text{ in } \hspace{0.05em} X \times X \times (0, +\infty) \\
p = \frac{\delta_{xy}}{\mu(x)} & \text{ in } \hspace{0.05em} X \times X \times \{0\},
\end{cases}
\end{equation}
\noindent where the discrete Laplacian operator \( \Delta \) can be applied either in \( x \) or in \( y \).

\medskip
\noindent The next proposition states the main properties of the heat kernel \( p \); it follows from both {Lemma 3.5} in \cite{W} and {Theorem 2.5} in \cite{Woj}.

\begin{proposition}\label{proposition_properties_HK_G}
\noindent Let \( (X,\omega,\mu) \) be a connected and locally finite weighted graph, and consider the heat kernel \( p \) on \( (X,\omega,\mu) \). Then, for all \( t \geq 0, s \geq 0 \) and for any \( x \in X, y \in X \), the following properties hold:
\begin{enumerate}
\renewcommand{\labelenumi}{\alph{enumi})}
    \item \( p(x,y,\cdot) \in C^{\infty}((0,+\infty)) \cap C^0([0,+\infty)) \);
    \item \( p(x,y,t) > 0 \) whenever \( t > 0 \), while \( p(x,y,0) = \frac{\delta_{xy}}{\mu(x)} = \frac{\delta_{xy}}{\mu(y)} \);
    \item \( p(x,y,t) = p(y,x,t) \);
    \item \( \| p(x,\cdot,t) \|_1 = \sum \limits_{z \in X} p(x,z,t) \hspace{0.1em} \mu(z) \leq 1 \);
    \item \( \lim \limits_{t \to 0^+} \sum \limits_{z \in X} p(x,z,t) \, \varphi(z) \, \mu(z) = \varphi(x) \), for any \( \varphi \in \ell^\infty(X) \);
    \item \( p(x,y,t+s) = \sum \limits_{z \in X} p(x,z,t) \hspace{0.1em} p(z,y,s) \hspace{0.1em} \mu(z) \).
    \end{enumerate}
\end{proposition}

\noindent We shall now introduce the \emph{convolution operator} in the graph setting. More specifically, for every \( \varphi \in \ell^1(X,\mu) \cup \ell^\infty(X) \), we define its convolution with the heat kernel \( p \) as
\begin{equation}\label{definition_convolution_HK_general_graphs}
(P_t \, \varphi)(x) := \sum_{y \in X} p(x,y,t) \, \varphi(y) \, \mu(y) \hspace{2.0em} \text{for all } (x,t) \in X \times [0,+\infty).
\end{equation}

\noindent We highlight that, at time \( t = 0 \), the convolution operator defined in \eqref{definition_convolution_HK_general_graphs} simply coincides with the identity. Indeed, thanks to property b) in \cref{proposition_properties_HK_G}, for any \( \varphi \in \ell^1(X,\mu) \cup \ell^\infty(X) \) and for all \( x \in X \) we obtain:
\[
(P_0 \, \varphi)(x) = \sum_{y \in X} p(x,y,0) \, \varphi(y) \, \mu(y) = \sum_{y \in X} \delta_{xy} \, \varphi(y) = \varphi(x).
\]

\begin{lemma}\label{lemma_property_HK_G}
\noindent Let \( (X,\omega,\mu) \) be a connected and locally finite weighted graph, and consider the heat kernel \( p \) on \( (X,\omega,\mu) \). Then, for all \( t \geq 0 \), the operator \( P_t \) defined in \eqref{definition_convolution_HK_general_graphs} maps both \( \ell^1(X,\mu) \) and \( \ell^\infty(X) \) into themselves; in particular, it holds that
\[
\| P_t \, \varphi \|_{\infty} \leq \| \varphi \|_{\infty} \hspace{2.5em} \text{ for all } \hspace{0.1em} \varphi \in \ell^\infty(X)
\]
\noindent and that
\[
\| P_t \, \varphi \|_1 \leq \| \varphi \|_1 \hspace{2.5em} \text{ for all } \hspace{0.1em} \varphi \in \ell^1(X,\mu).
\]
\noindent In addition, for any \( \varphi \in \ell^1(X,\mu) \) we have:
\[
\sum_{x \in X} (P_t \, \varphi)(x) \, \mu(x) \leq \sum_{x \in X} \varphi(x) \, \mu(x) \hspace{2.5em} \text{for all } \hspace{0.1em} t \geq 0.
\]
\noindent If \( (X,\omega,\mu) \) is also stochastically complete, then the previous estimate becomes an identity, namely, for any \( \varphi \in \ell^1(X,\mu) \):
\[
\sum_{x \in X} (P_t \, \varphi)(x) \, \mu(x) = \sum_{x \in X} \varphi(x) \, \mu(x) \hspace{2.5em} \text{for all } \hspace{0.1em} t \geq 0.
\]
\end{lemma}

\medskip
\noindent In the present analysis we are only interested in \emph{classical} solutions to problem \eqref{problem_H_G}.

\begin{definition}\label{definition_sol__subsol_supersol_H_G}
\noindent Let \( (X,\omega,\mu) \) be a connected and locally finite weighted graph, and assume that \( v_0 \in \ell^\infty(X) \). Consider a function \( v : X \times [0,+\infty) \to \mathbb{R} \) such that
\begin{equation}\label{reg_sol__H_G}
v(x,\cdot) \in C^1((0,+\infty)) \cap C^0([0,+\infty)) \hspace{1.6em} \text{for all } \hspace{0.1em} x \in X.
\end{equation}
\noindent Then \( v \) is said to be a solution of problem \eqref{problem_H_G}, with initial condition \( v_0 \), if \( v \) satisfies \eqref{problem_H_G} pointwise. Moreover, if \( v \in L^\infty(X \times (0,+\infty)) \), then \( v \) is said to be a bounded solution of problem \eqref{problem_H_G}.
\end{definition}

\noindent The following result concerns the well-posedness of problem \eqref{problem_H_G}; it is obtained by combining {Theorem 7.2} and {Theorem 7.3} in \cite{KLWb}.

\begin{theorem}\label{theorem_well_posedness_H_G}
\noindent Let \( (X,\omega,\mu) \) be a connected and locally finite weighted graph, and assume that \( v_0 \in \ell^\infty(X) \). Then the function \( v : X \times [0,+\infty) \to \mathbb{R} \), defined as
\[
v(x,t) := (P_t \, v_0)(x) = \sum_{y \in X} p(x,y,t) \, v_0(y) \, \mu(y) \hspace{2.0em} \text{for all } \hspace{0.1em} (x,t) \in X \times [0,+\infty),
\]
\noindent is a bounded solution to problem \eqref{problem_H_G}. In particular, if \( (X,\omega,\mu) \) is stochastically complete, then this function is the only bounded solution to problem \eqref{problem_H_G}.
\end{theorem}


\subsection{The heat semigroup}\label{subsection_heat_semigroup}

\noindent In the literature, the family \( \{ P_t \}_{t \geq 0} \), introduced through \eqref{definition_convolution_HK_general_graphs}, is referred to as the \emph{heat semigroup} generated by the operator \( - \Delta \). We shall now present some basic definitions and properties regarding the heat semigroup on generic graphs.

\begin{definition}\label{definition_S_1_S_2_semigroup}
\noindent The heat semigroup \( \{ P_t \}_{t \geq 0} \) is said to be
\begin{enumerate}
\renewcommand{\labelenumi}{\alph{enumi})}
    \item positivity preserving if, for all \( t \geq 0 \), it holds:
    \begin{equation}\label{definition_positivity_preserving}
    P_t \, \varphi \geq 0 \hspace{1.0em} \text{pointwise in } X,
    \hspace{2.5em} \text{for all } \hspace{0.1em} \varphi \in \ell^\infty(X) \hspace{0.1em} \text{ such that } \hspace{0.1em} \varphi \geq 0 \hspace{0.1em} \text{ in } X;
    \end{equation}

    \item \( \ell^\infty \)-contractive if, for all \( t \geq 0 \), we have:
    \begin{equation}\label{definition_ell_infty_contraction}
    \| P_t \, \varphi \|_\infty \leq \| \varphi \|_\infty \hspace{2.5em} \text{for all } \hspace{0.1em} \varphi \in \ell^\infty(X);
    \end{equation}

    \item uniformly continuous if it holds:
    \[
    \lim_{t \to 0^+} \, \|P_t - I\|_{\mathcal{L}} = 0,
    \]
    where the norm \( \| \cdot \|_{\mathcal{L}} \) is clearly the one in \( \mathcal{L}(\ell^\infty(X)) \) and \( I : \mathcal{L}(\ell^\infty(X)) \to \mathcal{L}(\ell^\infty(X)) \) represents the identity operator;

    \item strongly continuous, or equivalently of class \( C_0 \), if
    \[
    \lim_{t \to 0^+} \, \|P_t \, \varphi - \varphi \|_\infty = 0 \qquad \quad \text{for every } \hspace{0.1em} \varphi \in \ell^\infty(X),
    \]
    namely if \( P_t \, \varphi \) converges to \( \varphi \) in \( \ell^\infty(X) \), as \( t \to 0^+ \).
\end{enumerate}
\end{definition}

\begin{remark}\label{remark_properties_SG_connected_locally_finite}
\noindent As already specified, we shall always assume that the weighted graph \( (X,\omega,\mu) \) is connected and locally finite. Then, by combining property b) in \emph{\cref{proposition_properties_HK_G}} with the positivity of the node measure \( \mu \), the inequality \eqref{definition_positivity_preserving} is trivially verified. On the other hand, in this scenario \emph{\cref{lemma_property_HK_G}} can be applied, yielding the validity of \eqref{definition_ell_infty_contraction}. Therefore, under our assumptions on the graph, the heat semigroup \( \{ P_t \}_{t \geq 0} \) is both positivity preserving and \( \ell^\infty \)-contractive.

\noindent In addition, under condition \eqref{definition_ell_infty_contraction}, the map \( t \mapsto P_t \) acts from \( [0,+\infty) \) to \( \mathcal{L}(\ell^\infty(X)) \). Indeed, after fixing an arbitrary time instant \( t \geq 0 \), \eqref{definition_ell_infty_contraction} allows us to infer that \( P_t \) maps \( \ell^\infty(X) \) into itself, and also that this operator is bounded. Moreover, it is easy to verify that \( P_t \) is a linear operator. In conclusion, within our framework it holds \( P_t \in \mathcal{L}(\ell^\infty(X)) \).
\end{remark}

\begin{remark}\label{remark_uniformly_strongly_continuous_HK}
\noindent Notice that, under our hypotheses on the graph, the uniform continuity of the heat semigroup implies the strong continuity. Indeed, for all \( \varphi \in \ell^\infty(X) \), \( t \geq 0 \), we can write:
\[
\|P_t \, \varphi - \varphi \|_\infty = \| (P_t - I) \, \varphi \|_\infty \leq \|P_t - I\| _{\mathcal{L}} \, \| \varphi \|_\infty.
\]
In particular, if the semigroup is uniformly continuous, then the right-hand side tends to zero as \( t \to 0^+ \), so that also \( \|P_t \, \varphi - \varphi \|_\infty \to 0 \), yielding the strong continuity.
\end{remark}

\noindent For future reference, we also state the validity, under condition \eqref{definition_ell_infty_contraction}, of the semigroup identity
\begin{equation}\label{identity_P_t_P_s}
P_{t+s} = P_t \, P_s,
\end{equation}
\noindent holding for all \( s,t \geq 0 \) as an equality between two elements of \( \mathcal{L}(\ell^\infty(X)) \). This fact is due to the combination between the definition \eqref{definition_convolution_HK_general_graphs} of the operator \( P_{(\cdot)} \) and property f) in \cref{proposition_properties_HK_G}.

\medskip
\noindent Another fact which will be expedient in the following analysis is the estimate
\begin{equation}\label{L_norm_P_t}
\| P_t \|_{\mathcal{L}(\ell^\infty(X))} \leq 1 \quad \qquad \text{for all } t \geq 0,
\end{equation}
obtained by combining the definition of the operator norm \( \| \cdot \|_{\mathcal{L}(\ell^\infty(X))} \) with the \( \ell^\infty \)-contractivity of the heat semigroup, which holds in our framework, as already specified in \cref{remark_properties_SG_connected_locally_finite}.

\subsection{The heat equation on the integer lattice}\label{subsection_H_E_lattice}

\noindent Let us consider problem \eqref{problem_H_G} posed on the lattice, namely:
\[
\begin{cases}
v_t - \Delta v = 0 & \text{ in } \mathbb{Z}^N \times (0, +\infty) \\
v = v_0 & \text{ in } \mathbb{Z}^N \times \{0\}, \\
\end{cases}
\]
\noindent with the same notation and assumptions as those stated at the beginning of \cref{subsection_H_E_graph_setting}.

\medskip
\noindent Although the analytical expression of the heat kernel associated to problem \eqref{problem_H_G} is often not available, in the specific context of the integer lattice \( p : \mathbb{Z}^N \times \mathbb{Z}^N \times [0,+\infty) \to \mathbb{R} \) is given by
\[
p(x,y,t) = \frac{1}{2N} \, e^{-t} \prod_{i=1}^{N} I_{|x_i-y_i|} \left( \frac{t}{N} \right).
\]
\noindent One possible reference for this fact is \cite{HSS}.

\medskip
\noindent The function \( I \) appearing in the explicit formula for the heat kernel is known as the \emph{modified I-Bessel function}. For any \( m \in \mathbb{Z} \), the function \( I_m \) is defined as a solution to the differential equation
\[
t^2 \frac{d^2 w}{dt^2} + t \frac{dw}{dt} - (t^2 + m^2) = 0,
\]
\noindent and it holds \( I_m : (0,+\infty) \to (0,+\infty) \), together with \( I_m = I_{-m} \). Therefore, whenever \( m \in \mathbb{Z} \), it is common to write this function as \( I_{|m|} \), in order to maintain the positivity of the argument.

\medskip
\noindent In order to ease the notation, we now set
\begin{equation}\label{definition_K_t_HK_lattice}
K_t(x) := p(x,0,t) = \frac{1}{2N} \, e^{-t} \prod_{i=1}^{N} I_{|x_i|} \left( \frac{t}{N} \right) \hspace{2.5em} \text{for all } (x,t) \in \mathbb{Z}^N \times [0,+\infty),
\end{equation}
\noindent which will also be referred to as the \emph{heat kernel on the lattice}, with a slight abuse of notation. Now, it is trivial to notice that the following relation holds:
\[
p(x,y,t) = K_t(x-y) \hspace{2.5em} \text{for all } (x,y,t) \in \mathbb{Z}^N \times \mathbb{Z}^N \times [0,+\infty).
\]
\noindent In this way, thanks to the definition of the operator \( P_t \) given in \eqref{definition_convolution_HK_general_graphs}, we can extend, for any function \( \varphi \in \ell^1(\mathbb{Z}^N) \cup \ell^\infty(\mathbb{Z}^N) \) and for all \( (x,t) \in \mathbb{Z}^N \times [0,+\infty) \), the definition of the convolution operator * as follows:
\begin{equation}\label{definition_convolution_lattice}
(K_t * \varphi)(x) := \hspace{0.3em} 2N \sum_{y \in \mathbb{Z}^N} K_t(x-y) \, \varphi(y)
= \hspace{0.3em} \sum_{y \in \mathbb{Z}^N} p(x,y,t) \, \varphi(y) \, \mu(y)
= \hspace{0.3em} (P_t \, \varphi)(x).
\end{equation}
\noindent This is the discrete analogue of the standard convolution operator defined in the Euclidean setting.

\medskip
\noindent The following properties of the heat kernel on the lattice can be easily established (see, e.g., \cite{book_Barlow}).

\begin{proposition}\label{proposition_properties_HK_lattice}
\noindent Consider the integer lattice \( (\mathbb{Z}^N,\omega_0,\mu) \), together with the heat kernel \( K_t \) defined in \eqref{definition_K_t_HK_lattice}. Then, for all \( t \geq 0, s \geq 0 \) and for any \( x \in \mathbb{Z}^N \), the following properties hold:
\begin{enumerate}
\renewcommand{\labelenumi}{\alph{enumi})}
    \item \( K_{(\cdot)}(x) \in C^{\infty}((0,+\infty)) \cap C^0([0,+\infty)) \);
    \item \( K_t(x) > 0 \) whenever \( t > 0 \);
    \item \( K_0(0) = \frac{1}{2N} \), while \( K_0(x) = 0 \) whenever \( x \neq 0 \);
    \item \( \| K_t \|_1 = 2N \sum \limits_{y \in \mathbb{Z}^N} K_t(y) = 1 \);
    \item \( \sum \limits_{y \in \mathbb{Z}^N} (K_t * \varphi)(y) = \sum \limits_{y \in \mathbb{Z}^N} \varphi(y) \), for every \( \varphi \in \ell^1(\mathbb{Z}^N) \);
    \item \( \lim \limits_{\hspace{0.15em} t \to 0^+} (K_t * \varphi)(x) = \varphi(x) \), for every \( \varphi \in \ell^\infty(\mathbb{Z}^N) \);
    \item \( (K_t * K_s)(x) = K_{t+s}(x) \).
\end{enumerate}
\end{proposition}

\noindent We now establish two-sided estimates for the heat kernel $K$. They are essentially contained in \cite{book_Barlow}; more precisely, they are obtained from those results through straightforward arguments.

\begin{lemma}\label{lemma_inequalities_HK_lattice}
\noindent The heat kernel \( K_t \) defined in \eqref{definition_K_t_HK_lattice} satisfies both
\[
K_t(x) \geq \frac{c_1}{t^{N/2}} \, e^{- c_2 \, \frac{|x|^2}{t}}
\hspace{2.0em} \text{for all } \hspace{0.1em} (x,t) \in \mathbb{Z}^N \times (0,+\infty) \hspace{0.3em} \text{ such that } \hspace{0.3em} t \geq \max \left\{ 1, \rho(x,0) \right\}
\]
\noindent and
\[
K_t(x) \leq \frac{c_3}{t^{N/2}} \hspace{2.5em} \text{for all } \hspace{0.1em} (x,t) \in \mathbb{Z}^N \times (0,+\infty),
\]
\noindent where \( \rho \) denotes the combinatorial distance, while \( c_1, c_2, c_3 \) are three positive constants depending only on the dimension \( N \).
\end{lemma}

\noindent \emph{Proof.} We first recall that the lattice admits the expression \eqref{definition_rho_lattice} for the combinatorial distance:
\[
\rho(x,y) = \sum_{i=1}^{N} |x_i - y_i| \hspace{2.5em} \text{for all } x,y \in \mathbb{Z}^N.
\]
\noindent We now exploit result b) of {Theorem 6.28} in \cite{book_Barlow}, which reads, within our notation, as follows:
\[
p(x,y,t) \geq c \, t^{-N/2} \, \exp \left( - C \, \frac{\left[ \rho(x,y) \right]^2}{t} \right),
\]
\noindent holding for all \( (x,y,t) \in \mathbb{Z}^N \times \mathbb{Z}^N \times (0,+\infty) \) such that \( t \geq \max \left\{ 1, \rho(x,y) \right\} \), where \( c \) and \( C \) are two positive constants only depending on \( N \). In particular, setting \( y = 0 \) and exploiting the relation \eqref{definition_K_t_HK_lattice} between \( K_t \) and \( p(\cdot,\cdot,t) \) yields:
\begin{equation}\label{lower_bound_HK_lattice_1}
K_t(x) = p(x,0,t) \geq c \, t^{-N/2} \, \exp \left( - C \, \frac{\left[ \rho(x,0) \right]^2}{t} \right),
\end{equation}
\noindent holding for all \( (x,t) \in \mathbb{Z}^N \times (0,+\infty) \) such that \( t \geq \max \left\{ 1, \rho(x,0) \right\} \).

\medskip
\noindent Furthermore, we can make use of the bounds \eqref{equivalence_rho_d__lattice} between the Euclidean distance and the combinatorial one, namely:
\[
d(x,y) \leq \rho(x,y) \leq \sqrt{N} \, d(x,y) \hspace{2.5em} \text{for all } x,y \in \mathbb{Z}^N.
\]
\noindent More specifically, we exploit the second inequality with \( y = 0 \), which thanks to \eqref{lower_bound_HK_lattice_1} yields, for all \( (x,t) \in \mathbb{Z}^N \times (0,+\infty) \) such that \( t \geq \max \left\{ 1, \rho(x,0) \right\} \), the following bound:
\[
K_t(x) \geq c \, t^{-N/2} \, \exp \left( - C \, \frac{\left[ \rho(x,0) \right]^2}{t} \right)
\geq c \, t^{-N/2} \, \exp \left( - C N \, \frac{\left[ d(x,0) \right]^2}{t} \right)
\equiv \frac{c}{t^{N/2}} \, e^{- C N \, \frac{|x|^2}{t}},
\]
\noindent where the last inequality exploits the fact that the exponential function involved here is decreasing, while the last equality is due to the fact that \( d(x,0) = |x-0| \equiv |x| \). Therefore, after setting \( c_1 := c \) and \( c_2 := C N \), the first bound in the statement is proved.

\medskip
\noindent In order to conclude the proof, we refer again to \cite{book_Barlow}. In particular, combining {Lemma 3.13} with {Theorem 5.15} ensures that
\[
p(x,y,t) \leq \frac{c_3}{t^{N/2}} \hspace{2.5em} \text{for all } \hspace{0.1em} (x,y,t) \in \mathbb{Z}^N \times \mathbb{Z}^N \times (0,+\infty),
\]
\noindent where \( c_3 \) is a positive constant only depending on \( N \). The second bound in the statement immediately follows by applying the last inequality with \( y = 0 \), and by exploiting the relation \eqref{definition_K_t_HK_lattice}. \hfill \( \square \)

\medskip
\noindent The next result, corresponding to {Theorem 5.1} in \cite{HSS}, addresses the well-posedness of problem \eqref{problem_H_G} posed on the lattice, giving also an explicit representation formula for the unique bounded solution.

\begin{theorem}\label{theorem_well_posedness_H_G_lattice}
If \( v_0 \in \ell^\infty(\mathbb{Z}^N) \), then the only bounded solution of problem \eqref{problem_H_G} posed on the lattice is
\[
v(x,t) = (K_t * v_0)(x) = e^{-t} \sum_{y \in \mathbb{Z}^N} \, \prod_{i=1}^{N} I_{|x_i - y_i|} \left( \frac{t}{N} \right) \, v_0(y) \hspace{2.5em} \text{for all } (x,t) \in \mathbb{Z}^N \times [0,+\infty),
\]
\noindent where the convolution is intended as in \eqref{definition_convolution_lattice}.

\end{theorem}


\section{Auxiliary results: proofs}\label{ardim}

\noindent \emph{Proof of \emph{\cref{proposition_Duhamel_formula_HR_G}}.} Let us fix an arbitrary time instant \( t_* \in (0,+\infty) \), and consider the function \( v : X \times [0,t_*) \to [0,+\infty) \), defined as
\begin{equation}\label{definition_function_v_instrumental}
v(x,t) := (P_t \, u_0)(x) + \int_0^t \, (P_{t-s} \, \left[ u(\cdot,s) \right]^p)(x) \, ds \qquad \quad \text{for all } (x,t) \in X \times [0, t_*).
\end{equation}
By exploiting \cref{lemma_property_HK_G} and the fact that \( u_0 \in \ell^\infty(X) \), we know that it holds
\[
| P_t \, u_0 (x) | \leq \| P_t \, u_0 \|_{\infty} \leq \| u_0 \|_{\infty} < +\infty \hspace{2.5em} \text{ for all } x \in X.
\]
In addition, after fixing \( x \in X \) and \( t \in [0,t_*) \), for every \( s \in [0,t] \) we have:
\begin{equation}\label{P_t_s_bound_1}
\begin{aligned}
\left| (P_{t-s} \, \left[ u(\cdot,s) \right]^p)(x) \right| & = \left| \sum_{y \in X} \, p(x,y,t-s) \, \left[ u(y,s) \right]^p \, \mu(y) \right| \\
& \leq \left[ \|u\|_{L^\infty(X \times [0,t_*))} \right]^p \, \sum_{y \in X} \, p(x,y,t-s) \, \mu(y) \\
& = \left[ \|u\|_{L^\infty(X \times [0,t_*))} \right]^p \\
& < +\infty.
\end{aligned}
\end{equation}
Here, we have exploited the nonnegativity of \( p \), the strict positivity of \( \mu \), the assumption of stochastic completeness, and the boundedness condition \eqref{reg_sol_2_G}. Moreover, we made use of the identity \( \| (\cdot)^p \|_\infty = \left[ \| (\cdot) \|_\infty \right]^p \). By integrating the previous estimate over \( s \in [0,t] \), we get:
\[
\begin{aligned}
\left| \int_{0}^{t} \, (P_{t-s} \, \left[ u(\cdot,s) \right]^p)(x) \, ds \right|
& \leq \int_{0}^{t} \, \left| (P_{t-s} \, \left[ u(\cdot,s) \right]^p)(x) \right| \, ds \\
& \leq \int_{0}^{t} \, \left[ \|u\|_{L^\infty(X \times [0,t_*))} \right]^p \, ds \\
& = t \left[ \|u\|_{L^\infty(X \times [0,t_*))} \right]^p \\
& < +\infty.
\end{aligned}
\]
Hence all the series implicitly involved in \eqref{definition_function_v_instrumental} are absolutely convergent, and the function \( v \) is well-defined. The fact that \( v \geq 0 \) is a trivial consequence of the action of \( P_{(\cdot)} \), together with the nonnegativity of \( u_0, u, \mu, \) and \( p \).

\medskip
\noindent We now separately study the time differentiability of the two terms at the right-hand side of \eqref{definition_function_v_instrumental}.

\medskip
\noindent \textbf{(I)} Let us fix an arbitrary vertex \( x \in X \). We aim to show that the map \( t \mapsto (P_t \, u_0)(x) \) belongs to \( C^0([0,t_*)) \cap C^1((0,t_*)) \) and satisfies \( \frac{\partial}{\partial t} (P_t \, u_0)(x) = (\Delta P_t \, u_0)(x)\).

\medskip
\noindent In order to do so, we consider, for \( h > 0 \), the incremental quotient
\begin{equation}\label{definition_incremental_quotient}
\frac{(P_{t+h} \, u_0)(x) - (P_t \, u_0)(x)}{h} = \sum_{y \in X} \, \left[ \frac{p(x,y,t+h) - p(x,y,t)}{h} \right] u_0(y) \, \mu(y),
\end{equation}
where the equality is a simple consequence of the action of the operator \( P_{(\cdot)} \). Now, for every fixed vertex \( y \in X \), we have:
\[
\lim_{h \to 0^+} \, \frac{p(x,y,t+h) - p(x,y,t)}{h} = \frac{\partial}{\partial t} \, p(x,y,t) = \Delta_x \, p(x,y,t).
\]
Here, the first equality is due to property a) of \cref{proposition_properties_HK_G}, stating that \( p(x,y,\cdot) \in C^{\infty}((0,+\infty)) \), while the second one holds since the heat kernel solves problem \eqref{problem_solved_by_p}, where the discrete Laplacian operator \( \Delta \) can be applied either in \( x \) or in \( y \).

\medskip
\noindent In particular, the time regularity of the heat kernel allows us to exploit the mean value formula, yielding the following, for every \( h > 0 \):
\[
\frac{p(x,y,t+h) - p(x,y,t)}{h} = \frac{1}{h} \, \int_t^{t+h} \, \frac{\partial}{\partial \tau} \, p(x,y,\tau) \, d\tau = \frac{1}{h} \, \int_t^{t+h} \, \Delta_x \, p(x,y,\tau) \, d\tau.
\]
By combining the obtained identity with \eqref{definition_incremental_quotient}, we obtain:
\begin{equation}\label{formula_before_time_space_swap}
\frac{(P_{t+h} \, u_0)(x) - (P_t \, u_0)(x)}{h} = \frac{1}{h} \, \sum_{y \in X} \left[ \int_t^{t+h} \, \Delta_x \, p(x,y,\tau) \, d\tau \right] u_0(y) \, \mu(y).
\end{equation}
For the moment, let us fix \( h > 0 \) and \( \tau \in (t,t+h) \). The meaning of the term involving the Laplacian with respect to \( x \) of the heat kernel is clearly the following:
\[
\Delta_x \, p(x,y,\tau) = \frac{1}{\mu(x)} \, \sum_{z \in X} \, \left[p(z,y,\tau) - p(x,y,\tau) \right] \omega(x,z),
\]
\noindent which corresponds to a sum over the neighbors of \( x \); this sum is then finite, since the graph is assumed to be locally finite. Therefore, this sum can commute with the one over \( y \in X \). In particular, the following can be obtained:
\[
\sum_{y \in X} \, \Delta_x \, p(x,y,\tau) \, u_0(y) \, \mu(y)
= \frac{1}{\mu(x)} \, \sum_{z \in X} \, \left\{ \sum_{y \in X} \, \left[ p(z,y,\tau) - p(x,y,\tau) \right] u_0(y) \, \mu(y) \right\} \omega(x,z).
\]
\noindent By the definition of the operator \( P_{(\cdot)} \), for each \( y \in X \) we have:
\[
\begin{aligned}
\sum_{y \in X} \, \left[ p(z,y,\tau) - p(x,y,\tau) \right] u_0(y) \, \mu(y) & = \sum_{y \in X} \, p(z,y,\tau) \, u_0(y) \, \mu(y) - \sum_{y \in X} \, p(x,y,\tau) \, u_0(y) \, \mu(y) \\
& = P_{\tau} \, u_0 (z) -  P_{\tau} \, u_0 (x),
\end{aligned}
\]
therefore the previous identity becomes
\begin{equation}\label{internal_formula_after_swap}
\begin{aligned}
\sum_{y \in X} \, \Delta_x \, p(x,y,\tau) \, u_0(y) \, \mu(y) & = \frac{1}{\mu(x)} \, \sum_{z \in X} \, \left[ P_{\tau} \, u_0 (z) -  P_{\tau} \, u_0 (x) \right] \omega(x,z) \\
& = \left[ \Delta \left( P_{\tau} \, u_0 \right) \right] (x).
\end{aligned}
\end{equation}
We now want to swap the sum and the time integral at the right-hand side of \eqref{formula_before_time_space_swap}. In order to do so, by reasoning as above, we obtain the following estimate:
\[
\begin{aligned}
\sum_{y \in X} \, \left| \Delta_x \, p(x,y,\tau) \right| \, u_0(y) \, \mu(y) &
\leq \frac{\| u_0 \|_\infty}{\mu(x)} \, \sum_{z \in X} \left\{ \sum_{y \in X} \, \left[ p(z,y,\tau) + p(x,y,\tau) \right] \mu(y) \right\} \omega(x,z) \\
& = \frac{2 \, \| u_0 \|_\infty}{\mu(x)} \, \sum_{z \in X} \, \omega(x,z) \\
& = 2 \, \| u_0 \|_\infty \, \operatorname{Deg}(x).
\end{aligned}
\]
Here, after swapping the finite sum over \( z \sim x \) with the one over \( y \), we have exploited the assumption \eqref{eq:HP_ID_G} satisfied by \( u_0 \), together with the nonnegativity of the heat kernel and the stochastic completeness of the graph. Finally, the last equality is due to \cref{definition_weighted_degree}.

\medskip
\noindent We notice that the right-hand side in the last estimate is independent of \( \tau \). Moreover, for the fixed vertex \( x \in X \) it clearly holds \( \operatorname{Deg}(x) < +\infty \). In conclusion, the series
\[
\sum_{y \in X} \, \Delta_x \, p(x,y,\tau) \, u_0(y) \, \mu(y)
\]
is absolutely convergent and uniformly dominated in \( \tau \). Therefore, the Fubini-Tonelli Theorem can be applied, yielding the exchange between the sum and the integral at the right-hand side of \eqref{formula_before_time_space_swap}. We then obtain:
\begin{equation}\label{formula_after_time_space_swap}
\begin{aligned}
\frac{(P_{t+h} \, u_0)(x) - (P_t \, u_0)(x)}{h} & = \frac{1}{h} \, \int_t^{t+h} \left[ \sum_{y \in X} \, \Delta_x \, p(x,y,\tau) \, u_0(y) \, \mu(y) \right] d\tau \\
& = \frac{1}{h} \, \int_t^{t+h} \left[ \Delta \left( P_{\tau} \, u_0 \right) \right] (x) \, d\tau,
\end{aligned}
\end{equation}
where the last equality is due to \eqref{internal_formula_after_swap}.

\medskip
\noindent In particular, thanks to the first equality in \eqref{formula_after_time_space_swap} and to the previous estimate, we infer:
\[
\begin{aligned}
\left| (P_{t+h} \, u_0)(x) - (P_t \, u_0)(x) \right| & = \left| \int_t^{t+h} \left[ \sum_{y \in X} \, \Delta_x \, p(x,y,\tau) \, u_0(y) \, \mu(y) \right] d\tau \right| \\
& \leq \int_t^{t+h} \left[ \sum_{y \in X} \, \left| \Delta_x \, p(x,y,\tau) \right| u_0(y) \, \mu(y) \right] d\tau \\
& \leq 2 \, h \, \| u_0 \|_\infty \, \operatorname{Deg}(x).
\end{aligned}
\]
Now, we recall that the vertex \( x \) is fixed, and the right-hand side of the obtained estimate tends to zero as \( h \to 0^+ \). Therefore, it holds:
\[
\lim_{h \to 0^+} \, (P_{t+h} \, u_0)(x) = (P_t \, u_0)(x),
\]
hence the map \( t \mapsto (P_t \, u_0)(x) \) is continuous on \( [0,t_*) \).

\medskip
\noindent Now, since the graph is assumed to be locally finite, then for each \( t \in [0,t_*) \) the quantity \( \left[ \Delta \left( P_t \, u_0 \right) \right] (x) \) is given by a finite sum, depending on the time variable only through terms of the form \( P_t \, u_0 (w) \), for some \( w \in X \). All these contributes are continuous, as we have just shown; we then conclude that also the map \( t \mapsto \left[ \Delta \left( P_t \, u_0 \right) \right] (x) \) is continuous on \( [0,t_*) \). This fact allows us to apply the mean–value theorem for integrals, yielding the following, for all \( t \in (0,t_*) \):
\[
\lim_{h \to 0^+} \, \frac{1}{h} \, \int_t^{t+h} \left[ \Delta \left( P_{\tau} \, u_0 \right) \right] (x) \, d\tau = \left[ \Delta \left( P_t \, u_0 \right) \right] (x).
\]
By combining this result with \eqref{formula_after_time_space_swap}, we obtain that the map \( t \mapsto \left( P_t \, u_0 \right) (x) \) is continuously differentiable on \( (0,t_*) \), with
\[
\frac{d}{dt} \left( P_t \, u_0 \right) (x) = \lim_{h \to 0^+} \, \frac{(P_{t+h} \, u_0)(x) - (P_t \, u_0)(x)}{h} = \left[ \Delta \left( P_t \, u_0 \right) \right] (x).
\]
\noindent Since the vertex \( x \in X \) is arbitrary, we have the following identity, satisfied pointwise in \( X \):
\[
\frac{\partial}{\partial t} \left( P_t \, u_0 \right) = \Delta \left( P_t \, u_0 \right) \qquad \qquad \text{for all } t \in (0,t_*).
\]

\medskip
\noindent \textbf{(II)} We now focus on the time differentiability of the integral term at the right-hand side of \eqref{definition_function_v_instrumental}. For the ease of notation, after fixing an arbitrary node \( x \in X \), we set
\[
F(t,s) := (P_{t-s} \, \left[ u(\cdot,s) \right]^p)(x) \qquad \qquad \text{for all } \hspace{0.1em} t,s \hspace{0.3em} \text{ such that } \hspace{0.3em} 0 \leq s \leq t < t_*.
\]
Thanks to the boundedness condition \eqref{reg_sol_2_G}, we  know that \( u \) satisfies, for all \( t \in [0,t_*) \):
\begin{equation}\label{bound_u_p_Duhamel}
\left[ u(x,t) \right]^p \leq \left[ \| u \|_{L^\infty(X \times [0,t_*))} \right]^p < +\infty,
\end{equation}
where the obtained bound is independent of the time variable. This fact allows us to repeat the same passages of step \textbf{(I)}, inferring that, for fixed \( s \in [0,t] \), the map \( \tau \mapsto (P_{\tau} \left[ u(\cdot,s) \right]^p)(x) \) belongs to \( C^0([0,t_*)) \hspace{0.1em} \cap \hspace{0.1em} C^1((0,t_*)) \) and satisfies the following identity pointwise in \( X \):
\[
\frac{\partial}{\partial \tau} \left( P_{\tau} \left[ u(\cdot,s) \right]^p \right) = \Delta \left( P_{\tau} \left[ u(\cdot,s) \right]^p \right) \qquad \qquad \text{for all } \tau \in (0,t_*).
\]
This implies that, for fixed \( s \in [0,t] \), the map \( t \mapsto F(t,s) = (P_{t-s} \left[ u(\cdot,s) \right]^p)(x) \) belongs to \( C^0([0,t_*)) \hspace{0.1em} \cap \hspace{0.1em} C^1((0,t_*)) \), and it holds, pointwise in \( X \):
\begin{equation}\label{partial_time_derivative_F}
\frac{\partial}{\partial t} F(t,s) = \frac{\partial}{\partial t} (P_{t-s} \left[ u(\cdot,s) \right]^p) = \Delta (P_{t-s} \left[ u(\cdot,s) \right]^p)
\qquad \qquad \text{for all } t \in (0,t_*).
\end{equation}
Moreover, for all \( t,s \) such that \( 0 \leq s \leq t < t_* \), the following bound follows from \cref{lemma_property_HK_G} and \eqref{bound_u_p_Duhamel}:
\[
\begin{aligned}
|F(t,s)| & = \left| (P_{t-s} \, \left[ u(\cdot,s) \right]^p)(x) \right| \\
& \leq \left\| P_{t-s} \, \left[ u(\cdot,s) \right]^p \right\|_\infty \\
& \leq \left\| \left[ u(\cdot,s) \right]^p \right\|_\infty \\
& = \left[ \left\| u(\cdot,s) \right\|_\infty \right]^p \\
& \leq \left[ \| u \|_{L^\infty(X \times [0,t_*))} \right]^p \\
& < +\infty.
\end{aligned}
\]
\noindent We now claim that it holds:
\begin{equation}\label{Leibniz_rule_derivation}
\frac{d}{dt} \, \int_0^t \, F(t,s) \, ds = F(t,t) + \int_0^t \, \frac{\partial}{\partial t} \, F(t,s) \, ds, \qquad \text{for all } t \in (0,t_*).
\end{equation}
First, we notice that the integral appearing at the left-hand side of \eqref{Leibniz_rule_derivation} is well-defined for every fixed \( t \in (0,t_*) \), in view of the last uniform bound on the function \( F \). Now, in order to justify rigorously the differentiation under the integral sign,
we provide a uniform bound on the partial derivative with respect to $t$ of the
function $F$.

\medskip
\noindent Again, let us fix a vertex \( x \in X \) and recall that, for $0 \le s \le t < t^*$, \eqref{partial_time_derivative_F} holds, namely:
\[
\frac{\partial}{\partial t} F(t,s)
= \Delta\big(P_{t-s}[u(\cdot,s)]^p\big)(x).
\]
In addition, it is easily seen that the function \( P_{t-s}[u(\cdot,s)]^p \) is bounded; indeed, from \eqref{P_t_s_bound_1} it immediately follows that
\begin{equation}\label{P_t_s_bound_2}
\| P_{t-s} \, \left[ u(\cdot,s) \right]^p \|_\infty = \sup_{x \in X} \, \left| (P_{t-s} \, \left[ u(\cdot,s) \right]^p)(x) \right| \leq \left[ \|u\|_{L^\infty(X \times [0,t_*))} \right]^p < +\infty.
\end{equation}
\noindent Since the graph is locally finite, the Laplacian at the vertex $x$ is well-defined as
a finite sum over the neighbors of $x$. Therefore, for every bounded function $g \in \ell^\infty(X)$, we have
\[
|\Delta g(x)| = \left| \frac{1}{\mu(x)} \sum_{\substack{z \in X \\ z \sim x}} \omega(x,z) \left[ g(z) - g(x) \right] \right|
\le \frac{1}{\mu(x)} \sum_{\substack{z \in X \\ z \sim x}} \omega(x,z) \left[ |g(z)| + |g(x)| \right]
\le 2 \, \mathrm{Deg}(x) \, \|g\|_\infty.
\]
Applying this estimate to the bounded function \( g = P_{t-s}[u(\cdot,s)]^p \) and exploiting both \eqref{partial_time_derivative_F} and \eqref{P_t_s_bound_2},
we obtain, for $0 \le s \le t < t^*$:
\[
\left| \frac{\partial}{\partial t} F(t,s) \right| = \left| \Delta\big(P_{t-s}[u(\cdot,s)]^p\big)(x) \right|
\le 2\,\mathrm{Deg}(x) \, \| P_{t-s}[u(\cdot,s)]^p \|_\infty
\le 2 \, \mathrm{Deg}(x) \left[ \|u\|_{L^\infty(X \times [0,t_*))} \right]^p.
\]
\noindent The right-hand side is independent of $s$ and belongs to $L^1(0,t)$ for every fixed $t \in (0,t^*)$.
Hence, by a standard differentiation theorem under the integral sign (Leibniz rule with dominated derivative),
we conclude that \eqref{Leibniz_rule_derivation} holds.

\medskip
\noindent Now, note that for all \( t \in (0,t_*) \) it holds:
\[
F(t,t) = (P_0 \, \left[ u(\cdot,t) \right]^p)(x) = \left[ u(x,t) \right]^p,
\]
since \( P_0 \) coincides with the identity operator. On the other hand, from \eqref{partial_time_derivative_F} it follows, for all \( t \in (0,t_*) \):
\[
\int_0^t \, \frac{\partial}{\partial t} \, F(t,s) \, ds = \int_0^t \, \Delta (P_{t-s} \left[ u(\cdot,s) \right]^p)(x) \, ds
= \Delta \left( \int_0^t \, P_{t-s} \left[ u(\cdot,s) \right]^p \, ds \right) (x),
\]
where the last equality is due to an interchange between the time integral and the sum defining the involved Laplacian term, which is finite thanks to the hypothesis of local finiteness of the graph. We can now insert the last two identities inside \eqref{Leibniz_rule_derivation} and exploit the definition of \( F(\cdot,\cdot) \), in order to obtain, for all \( x \in X \) and for each \( t \in (0,t_*) \):
\[
\frac{d}{d t} \, \int_0^t \, (P_{t-s} \, \left[ u(\cdot,s) \right]^p)(x) \, ds = \left[ u(x,t) \right]^p + \Delta \left( \int_0^t \, P_{t-s} \left[ u(\cdot,s) \right]^p \, ds \right) (x).
\]

\medskip
\noindent By combining the results obtained in \textbf{(I)} and \textbf{(II)}, and by exploiting the definition of the function \( v \), namely \eqref{definition_function_v_instrumental}, we get:
\[
\begin{aligned}
\frac{\partial}{\partial t} \, v(\cdot,t) & = \frac{\partial}{\partial t} \, (P_t \, u_0) + \frac{\partial}{\partial t} \, \int_0^t \, P_{t-s} \, \left[ u(\cdot,s) \right]^p \, ds \\
& = \Delta \left( P_t \, u_0 \right) + \left[ u(\cdot,t) \right]^p + \Delta \left( \int_0^t \, P_{t-s} \left[ u(\cdot,s) \right]^p \, ds \right) \\
& = \left[ u(\cdot,t) \right]^p + \Delta \left( P_t \, u_0 + \int_0^t \, P_{t-s} \left[ u(\cdot,s) \right]^p \, ds \right) \\
& = \left[ u(\cdot,t) \right]^p + \Delta \left[ v(\cdot,t) \right],
\end{aligned}
\]
holding pointwise in \( X \), for all \( t \in (0,t_*) \). Now, by the arbitrariness of \( t_* \in (0,+\infty) \), we conclude:
\[
v_t = u^p + \Delta \, v \qquad \quad \text{pointwise in } X \times (0,+\infty).
\]
Moreover, since \( P_0 \) corresponds to the identity operator, from \eqref{definition_function_v_instrumental} we also infer:
\[
v(\cdot,0) = P_0 \, u_0 = u_0 \qquad \quad \text{pointwise in } X.
\]
Since \( u \) is, by assumption, a global solution of problem \eqref{problem_HR_G}, we can define
\[
w := u - v \qquad \text{in } X \times [0,+\infty),
\]
so that we get
\[
w_t = u_t - v_t = \Delta \, u + u^p - \left[ u^p + \Delta \, v \right] = \Delta \, u - \Delta \, v = \Delta \left( u - v \right) = \Delta \, w,
\]
pointwise in \( X \times (0,+\infty) \), and also
\[
w(\cdot,0) = u(\cdot,0) - v(\cdot,0) = u(\cdot,0) - u(\cdot,0) \equiv 0,
\]
pointwise in \( X \).

\medskip
\noindent In conlusion, \( w \) solves \eqref{problem_H_G}, namely the Cauchy problem associated with the heat equation in the graph setting, with the trivial initial datum. Thanks to the stochastic completeness of the graph, we can then invoke \cref{theorem_well_posedness_H_G} in order to conclude that
\[
w = P_t \, w(\cdot,0) \equiv 0 \qquad \quad \text{pointwise in } X \times (0,+\infty),
\]
meaning that \( u \equiv v \) on \( X \times (0,+\infty) \). In particular, \eqref{definition_function_v_instrumental} implies that
\[
u(x,t) = v(x,t) = (P_t \, u_0)(x) + \int_0^t \left( P_{t-s} \, \left[u(\cdot,s) \right]^p \right)(x) \, ds \hspace{2.5em} \text{for all } \hspace{0.1em} (x,t) \in \hspace{0.05em} X \times (0,+\infty),
\]
so that \eqref{eq:Duhamel_G} has been shown to hold pointwise in \( X \times (0,+\infty) \). In addition, since \( u \) is assumed to be a global classical solution to problem \eqref{problem_HR_G}, then it is defined pointwise on \( X \times [0,+\infty) \) and it is nonnegative; moreover, from \eqref{reg_sol_2_G} we know that \( u \in L^\infty(X \times [0,\tau]) \) for all \( \tau \in (0,+\infty) \). Hence it is straightforward to verify that \( u \in L^\infty([0,\tau]; \ell^\infty(X)) \) for all \( \tau \in (0,+\infty) \). In conclusion, \( u \) satisfies all the requirements of \cref{definition_mild_sol_sub_super_sol} for a global mild solution to problem \eqref{problem_HR_G}. This ends the proof. \hfill \( \square \)

\bigskip
Before proving Proposition \ref{proposition_mild_implies_classical_global_general_graph}, we need the following lemma, whose proof requires the use of a result contained in \cite{book_Pazy}.

\begin{lemma}\label{lemma2}
Let $(X,\omega,\mu)$ be a connected and locally finite weighted graph satisfying \eqref{bounded_weighted_degree_2}. Let $T \in (0,+\infty]$ and let
\[
u : X \times [0,T) \to [0,+\infty)
\]
be a mild solution of problem \eqref{problem_HR_G} on $[0,\tau]$, in the sense of \emph{\cref{definition_mild_sol_sub_super_sol}}, for every \( \tau \in (0,T) \).
Then
\[
u \in C([0,T); \ell^\infty(X)).
\]
\end{lemma}

\medskip
\noindent \emph{Proof.} In order to ease the notation, throughout this proof we shall replace \( u(\cdot,t) \) simply with \( u(t) \), for every involved time instant \( t \); moreover, we will denote the norm on \( \mathcal{L}(\ell^\infty(X)) \) with \( \| \cdot \|_{\mathcal{L}} \).

\medskip
\noindent First, let us fix $\tau \in (0,T)$. By assumption,
\[
M_\tau := \|u\|_{L^\infty([0,\tau];\ell^\infty(X))} < +\infty,
\]
\noindent where this quantity only depends on \( \tau \). Now, let us consider the operator \( F: \ell^\infty(X) \to \ell^\infty(X) \) defined as
\begin{equation}\label{operator_F}
F(v) := |v|^p \qquad \quad \text{for all } v \in \ell^\infty(X).
\end{equation}
First, it is trivial to notice that \( F \) actually maps \( \ell^\infty(X) \) into itself; indeed, for any \( v \in \ell^\infty(X) \) we have:
\[
\| F(v) \|_\infty = \| |v|^p \|_\infty \equiv \| v^p \|_\infty = \left[ \| v \|_\infty \right]^p < +\infty.
\]
Moreover, since \( u \) is nonnegative and belongs to the space \( L^\infty([0,\tau];\ell^\infty(X)) \), then it follows that \( u(s) \in \ell^\infty(X) \) for a.e. \( s \in (0,\tau) \), and we have:
\begin{equation}\label{F_u_s_1}
F(u(s)) = [u(s)]^p \quad \text{ and } \quad \| F(u(s)) \|_\infty = \left[ \| u(s) \|_\infty \right]^p \le M_\tau^p,
\quad \text{ for a.e. } s\in(0,\tau).
\end{equation}
\noindent Because of \eqref{bounded_weighted_degree_2}, the Laplacian $\Delta$ is a bounded linear operator on $\ell^\infty(X)$. In fact, we first notice that, since the graph is assumed to be locally finite, then \( \Delta \) is well-defined in the whole space \( C(X) \), thanks to \cref{remark_locally_finite_laplacian}. Now, for all \( v \in \ell^\infty(X) \) and for each \( x \in X \), we can write:
\[
\begin{aligned}
|(\Delta v)(x)| & = \left| \frac{1}{\mu(x)} \, \sum_{y\in X} \left[ v(y)-v(x) \right] \omega(x,y) \right| \\
& \le \frac{1}{\mu(x)} \, \sum_{y\in X} \, \left[ |v(y)| + |v(x)| \right] \omega(x,y) \\
& \le \frac{2 \, \|v\|_\infty}{\mu(x)} \, \sum_{y\in X} \, \omega(x,y) \\
& = 2 \, \operatorname{Deg}(x) \, \|v\|_\infty \\
& \le 2 D \, \|v\|_\infty,
\end{aligned}
\]
where we have exploited the positivity of \( \mu \) and the nonnegativity of \( \omega \), together with the assumption \eqref{bounded_weighted_degree_2} regarding the boundedness of the weighted degree. Now, taking the supremum over \( x \in X \) in the previous estimates yields the following:
\[
\| \Delta \, v \|_\infty \le 2 D \, \|v\|_\infty \qquad \quad \text{for all } v \in \ell^\infty(X),
\]
hence the operator \( \Delta \), which is clearly linear, maps \( \ell^\infty(X) \) into itself. In particular, it is a bounded linear operator, namely \( \Delta \in \mathcal{L}(\ell^\infty(X)) \).

\medskip
\noindent We can then invoke {Theorem 1.2} from {Section 1.1} of \cite{book_Pazy}, which yields that the semigroup \( \{P_t\}_{t\ge 0} \) on \( \ell^\infty(X) \) generated by the bounded operator \( \Delta \) on the Banach space \( \ell^\infty(X) \) is \emph{uniformly continuous}. Moreover, since the graph is assumed to be connected and locally finite, \cref{remark_properties_SG_connected_locally_finite} ensures that the heat semigroup is both positivity preserving and \( \ell^\infty \)-contractive. In particular, thanks to \cref{remark_uniformly_strongly_continuous_HK}, we infer that \( \{P_t\}_{t\ge 0} \) is of class \( C_0 \), namely \emph{strongly continuous} on \( \ell^\infty(X) \).

\medskip
\noindent Now, let us fix an arbitrary time instant \( t \geq 0 \) and a generic function \( f \in \ell^\infty(X) \). Then, by combining the property \eqref{identity_P_t_P_s} with the \( \ell^\infty \)-contractivity of the heat semigroup, for every \( h \geq 0 \) we obtain:
\[
\| P_{t+h} \, f - P_t \, f \|_\infty = \| P_t \left[ \left( P_h - I \right) f \right] \|_\infty \leq \| \left( P_h - I \right) f \|_\infty \leq \| P_h - I \|_{\mathcal{L}} \| f \|_\infty,
\]
where the last inequality easily follows from the definition of the norm \( \| \cdot \|_{\mathcal{L}} \). In particular, thanks to the uniform continuity of the heat semigroup, it follows that \( P_{t+h} \, f \to P_t \, f \) in \( \ell^\infty(X) \), as \( h \to 0^+ \), so that the map \( r \mapsto P_r f \) is continuous at \( t \) from the right. Analogously, after fixing \( t > 0 \) and choosing an arbitrary \( h \in (0,t] \), by arguing as before we get:
\[
\| P_t \, f - P_{t-h} \, f \|_\infty
= \| P_{t-h} \left[ \left( P_h - I \right) f \right] \|_\infty
\leq \| \left( P_h - I \right) f \|_\infty
\leq \| P_h - I \|_{\mathcal{L}} \| f \|_\infty,
\]
so that \( P_{t-h} \, f \to P_t \, f \) in \( \ell^\infty(X) \), as \( h \to 0^+ \), meaning that the map \( r \mapsto P_r f \) is continuous at \( t \) from the left. Therefore, thanks to the arbitrariness of \( t \), we conclude that
\[
t \mapsto P_t \, f
\quad \text{is continuous } \text{on } [0,+\infty) \text{ in } \ell^\infty(X),
\qquad \text{for every } f \in \ell^\infty(X).
\]
In particular, the map \( t \mapsto P_t \, u_0 \) is continuous on \( [0,+\infty) \) in \( \ell^\infty(X) \).

\medskip
\noindent Now, by the mild formulation, and in particular by \eqref{eq:nonlinear-mild}, we have, for each $t\in[0,\tau]$,
\begin{equation}\label{mild_u_F}
u(t) = P_t u_0 + \int_0^t P_{t-s}F(u(s))\,ds,
\end{equation}
where the equality exploits \eqref{F_u_s_1}. Here, as specified in \cref{remark_mild_well_posed}, the integral is understood pointwise in $X$.

\medskip
\noindent We now define the map \( G : [0,\tau] \to \ell^\infty(X) \) as
\[
G(t) := \int_0^t P_{t-s} F(u(s)) \, ds \quad \qquad \text{for all } t \in [0,\tau],
\]
and we first observe that it actually defines an element of $\ell^\infty(X)$ for every $t\in[0,\tau]$, since
\[
\left\| G(t) \right\|_\infty
\leq \int_0^t \left\| P_{t-s} F(u(s)) \right\|_\infty \, ds
\leq \int_0^t \left\| F(u(s)) \right\|_\infty \, ds
\leq t \, M_\tau^p < +\infty,
\]
where we exploited the definition of \( G \), together with the \( \ell^\infty \)-contractivity of the heat semigroup and the bound given by \eqref{F_u_s_1}.

\medskip
\noindent Now, let \( t \in [0,\tau] \) and \( \{ t_n \}_{n \in \mathbb{N}} \subseteq [0,\tau] \) be a nonincreasing sequence such that $t_n \geq t$ for all \( n \in \mathbb{N} \) and \( t_n \downarrow t \) as \( n \to +\infty \). By exploiting the definition of the function \( G \), for all \( n \in \mathbb{N} \) we can write:
\begin{equation}\label{G_t_n_G_t_1}
G(t_n)-G(t) = \int_0^t \big( P_{t_n-s} - P_{t-s} \big) F(u(s)) \, ds + \int_t^{t_n} P_{t_n-s} F(u(s)) \, ds,
\end{equation}
where the operator \( P_{t_n-s} \) appearing inside the first integral is well-defined, since \( t_n \geq t \geq s \). We now estimate the \( \ell^\infty \)-norm of the two integral terms at the right-hand side of \eqref{G_t_n_G_t_1} separately, studying their asymptotic behavior as \( n \to +\infty \).

\medskip
\noindent The second term is bounded by
\[
\left\| \int_t^{t_n} P_{t_n-s} F(u(s)) \,ds \right\|_\infty \leq \int_t^{t_n} \left\| P_{t_n-s} F(u(s)) \right\|_\infty ds
\leq (t_n-t) M_\tau^p \xrightarrow[n \to +\infty]{} 0.
\]
Here, the last inequality is obtained by combining the \( \ell^\infty \)-contractivity of the heat semigroup with the estimate in \eqref{F_u_s_1}.

\medskip
\noindent For the first term we estimate
\begin{equation}\label{bound_DCT_1}
\left\| \int_0^t (P_{t_n-s} - P_{t-s}) F(u(s)) \, ds \right\|_\infty
\le \int_0^t \| P_{t_n-s} - P_{t-s} \|_{\mathcal{L}} \, \| F(u(s)) \|_\infty \, ds.
\end{equation}
Now, for each fixed $s\in(0,t)$ we have:
\[
\| P_{t_n-s} - P_{t-s} \|_{\mathcal{L}} =
\| P_{t-s} \left[ P_{t_n-t} - I \right] \|_{\mathcal{L}}
\leq \| P_{t-s} \|_{\mathcal{L}} \, \| P_{t_n-t} - I \|_{\mathcal{L}}
\leq \| P_{t_n-t} - I \|_{\mathcal{L}}
\xrightarrow[n \to +\infty]{} 0,
\]
where we have exploited the definition of the operator norm \( \| \cdot \|_{\mathcal{L}} \), the two properties \eqref{identity_P_t_P_s} and \eqref{L_norm_P_t} and the uniform continuity of the heat semigroup. Hence, for a.e. \( s \in (0,t) \), the integrand function at the right-hand side of \eqref{bound_DCT_1} tends to zero as \( n \to +\infty \). Moreover, by combining the triangular inequality for the norm \( \| \cdot \|_{\mathcal{L}} \) with \eqref{L_norm_P_t} and \eqref{F_u_s_1}, we obtain, for a.e. $s\in(0,t)$ and for all \(  n \in \mathbb{N}\):
\[
\| P_{t_n-s} - P_{t-s} \|_{\mathcal{L}} \, \| F(u(s)) \|_\infty
\leq \left[ \| P_{t_n-s} \|_{\mathcal{L}} + \| P_{t-s} \|_{\mathcal{L}} \right] \| F(u(s)) \|_\infty
\leq 2 M_\tau^p,
\]
where the last term is independent of \( s \), hence integrable on \( (0,t) \). Therefore, by dominated convergence, the right-hand side of \eqref{bound_DCT_1} tends to $0$ as $n\to\infty$. In particular, from \eqref{G_t_n_G_t_1} we easily infer that $G(t_n) \to G(t)$ in $\ell^\infty(X)$, as \( n \to +\infty \), so that the function \( G \) is continuous at \( t \) from the right.

\medskip
\noindent Now, taking an arbitrary nondecreasing sequence \( \{ t_n \}_{n \in \mathbb{N}} \subseteq [0,\tau] \) such that $t_n \leq t$ for all \( n \in \mathbb{N} \) and \( t_n \uparrow t \) as \( n \to +\infty \), for all \( n \in \mathbb{N} \) we can rewrite \eqref{G_t_n_G_t_1} as:
\begin{equation}\label{G_t_n_G_t_2}
G(t_n) - G(t) = \int_0^{t_n} \big( P_{t_n-s} - P_{t-s} \big) F(u(s)) \, ds - \int_{t_n}^t P_{t-s} F(u(s)) \, ds,
\end{equation}
where the operator \( P_{t-s} \) appearing inside the first integral is well-defined, since \( t \geq t_n \geq s \).

\medskip
\noindent By arguing as before, the \( \ell^\infty \)-norm of the second integral term is bounded by
\[
\left\| \int_{t_n}^t P_{t-s} F(u(s)) \,ds \right\|_\infty
\leq \int_{t_n}^t \left\| P_{t-s} F(u(s)) \right\|_\infty ds
\leq (t - t_n) M_\tau^p
\xrightarrow[n \to +\infty]{} 0.
\]
\noindent For the first integral term we proceed similarly with respect to \eqref{bound_DCT_1}, obtaining:
\begin{equation}\label{bound_DCT_2}
\begin{aligned}
\left\| \int_0^{t_n} (P_{t_n-s} - P_{t-s}) F(u(s)) \, ds \right\|_\infty
& \leq \int_0^{t_n} \| P_{t_n-s} - P_{t-s} \|_{\mathcal{L}} \, \| F(u(s)) \|_\infty \, ds \\
& = \int_0^t \mathds{1}_{[0,t_n]}(s) \, \| P_{t-s} - P_{t_n-s} \|_{\mathcal{L}} \, \| F(u(s)) \|_\infty \, ds,
\end{aligned}
\end{equation}
where \( \mathds{1}_{(\cdot)} \) denotes the indicator function. Now, let us fix $s\in(0,t)$. Since \( t_n \uparrow t \) as \( n \to +\infty \), there exists \( N(s) \in \mathbb{N} \) such that \( t_n > s \) for all \( n \geq N(s) \). Thus, for every \( n \geq N(s) \), it holds \( \mathds{1}_{[0,t_n]}(s) = 1 \), and moreover the operator \( P_{t_n-s} \) appearing at the right-hand side of \eqref{bound_DCT_2} is well-defined, so that, arguing as before, we have:
\[
\| P_{t-s} - P_{t_n-s} \|_{\mathcal{L}} =
\| P_{t_n-s} \left[ P_{t-t_n} - I \right] \|_{\mathcal{L}}
\leq \| P_{t_n-s} \|_{\mathcal{L}} \, \| P_{t-t_n} - I \|_{\mathcal{L}}
\leq \| P_{t-t_n} - I \|_{\mathcal{L}}
\xrightarrow[n \to +\infty]{} 0.
\]
Hence, for all \( s \in (0,t) \) and \( n \geq N(s) \) it holds:
\[
\mathds{1}_{[0,t_n]}(s) \, \| P_{t-s} - P_{t_n-s} \|_{\mathcal{L}} \equiv \| P_{t-s} - P_{t_n-s} \|_{\mathcal{L}} \leq \| P_{t-t_n} - I \|_{\mathcal{L}},
\]
which implies that, for a.e. \( s \in (0,t) \), the integrand function at the right-hand side of \eqref{bound_DCT_2} tends to zero as \( n \to +\infty \). Moreover, by proceeding as before, for a.e. $s \in (0,t)$ and for each \(  n \in \mathbb{N}\) we have that the integrand function at the right-hand side of \eqref{bound_DCT_2} is either identically zero or bounded by
\[
\mathds{1}_{[0,t_n]}(s) \, \| P_{t-s} - P_{t_n-s} \|_{\mathcal{L}} \, \| F(u(s)) \|_\infty
\leq \left[ \| P_{t-s} \|_{\mathcal{L}} + \| P_{t_n-s} \|_{\mathcal{L}} \right] \| F(u(s)) \|_\infty
\leq 2 M_\tau^p.
\]
By dominated convergence, the right-hand side of \eqref{bound_DCT_2} then tends to zero as $n \to +\infty$. In particular, from \eqref{G_t_n_G_t_2} we easily infer that $G(t_n) \to G(t)$ in $\ell^\infty(X)$, as \( n \to +\infty \), so that the function \( G \) is continuous at \( t \) from the left.

\medskip
\noindent In conclusion, the map \( G \) is continuous at \( t \). Now, since \( t \in [0,\tau] \) is arbitrary, we infer that \( G \in C([0,\tau];\ell^\infty(X)) \).
Since, as previously obtained, the map $t\mapsto P_t u_0$ is continuous in $\ell^\infty(X)$, from \eqref{mild_u_F} we conclude that $t\mapsto u(t)$ is continuous in $\ell^\infty(X)$ on $[0,\tau]$. Finally, the arbitrariness of \( \tau \in (0,T) \) yields that \( u \in C([0,T);\ell^\infty(X)) \), which is the thesis. \hfill \( \square \)

\medskip
\medskip
\noindent \emph{Proof of \emph{\cref{proposition_mild_implies_classical_global_general_graph}}.} Let us consider an arbitrary \( u_0 \) satisfying \eqref{eq:HP_ID_G}, and let \( u \) be a global mild solution to problem \eqref{problem_HR_G}, associated with the initial datum \( u_0 \). Our goal is to verify that all the conditions in \cref{definition_sub_super_sols_G} are satisfied by \( u \), with \( T = +\infty \).

\medskip
\noindent First, we notice that under our hypotheses \cref{lemma2} can be applied. More specifically, by arguing exactly as in the proof of such result, we infer that \( \Delta \in \mathcal{L}(\ell^\infty(X)) \) and that the semigroup \( \{P_t\}_{t\ge 0} \) generated by the operator \( \Delta \) on \( \ell^\infty(X) \) satisfies all the properties in \cref{definition_S_1_S_2_semigroup}, namely it is positivity preserving, \( \ell^\infty \)-contractive, and both strongly and uniformly continuous on \( \ell^\infty(X) \).

\medskip
\noindent Let us now consider the operator \( F: \ell^\infty(X) \to \ell^\infty(X) \) defined as in \eqref{operator_F}, namely \( F(v) := |v|^p \) for all \( v \in \ell^\infty(X) \). As already noticed in the proof of \cref{lemma2}, it can be easily seen that \( F \) actually maps \( \ell^\infty(X) \) into itself. Moreover, after fixing a constant \( M > 0 \), for all \( w,v \in \ell^\infty(X) \) such \( \|w\|_\infty \le M \) and \( \|v\|_\infty \le M \) it holds:
\[
\begin{aligned}
\| F(w) - F(v) \|_\infty & = \| |w|^p - |v|^p\|_\infty \\
& = \sup_{x\in X} \, | |w(x)|^p - |v(x)|^p | \\
& \le p \, M^{p-1} \, \sup_{x\in X} \, | w(x) - v(x) | \\
& = p \, M^{p-1} \, \|w-v\|_\infty \\
& = L_M \, \|w-v\|_\infty,
\end{aligned}
\]
hence \( F \) is locally Lipschitz continuous on \( \ell^\infty(X) \), with Lipschitz constant \( L_M := p \, M^{p-1} \in (0,+\infty) \). In the previous passages, the inequality is a direct consequence of the following estimate, holding for all \( p>1 \):
\[
| |a|^p - |b|^p | \leq p \, \left( \max \left\{ |a|, |b| \right\} \right)^{p-1} |a - b| \qquad \quad \text{for all } a,b \in \mathbb{R},
\]
which can be easily inferred by applying the Mean Value Theorem to the function \( y \mapsto |y|^p \) with \( p > 1 \), defined on \( \mathbb{R} \).

\medskip
\noindent As already noticed, \( \{P_t\}_{t\geq0} \) is a \( C_0 \)-semigroup on \( \ell^\infty(X) \). Combining this fact with the local Lipschitz continuity of the map \( F \), we may invoke {Theorem 1.4} from {Section 6.1} of \cite{book_Pazy}, yielding that the integral equation
\[
v(t) = P_t \, u_0 + \int_0^t P_{t-s} \, \left[ F \left( v(s) \right) \right] \, ds \equiv P_t \, u_0 + \int_0^t P_{t-s} \, \big[ |v(s)| \big]^p \, ds
\]
admits a unique solution \( v \in C([0,t_{max});\ell^\infty(X)) \), up to a maximal existence time \( t_{max} \in (0,+\infty] \). However, by \cref{definition_mild_sol_sub_super_sol}, the global mild solution \( u \) satisfies the equation \eqref{eq:nonlinear-mild} pointwise in \( X \) and for all \( t \in (0,+\infty) \), namely:
\[
u(\cdot,t) = P_t \, u_0 + \int_0^t \, P_{t-s}  \left[ u(\cdot,s)^p \right] \, ds = P_t \, u_0 + \int_0^t P_{t-s} \, \left[ F \left( u(\cdot,s) \right) \right] \, ds,
\]
where the second equality is due to the nonnegativity of \( u \), together with the definition of \( F \). Moreover, \cref{lemma2} ensures that $u \in C([0,+\infty); \ell^\infty(X))$. Therefore, it follows that we must have \( u = v \), and in particular it holds \( t_{max} = +\infty \).

\medskip
\noindent At this point, we want to employ Theorem 1.5 from Section 6.1 of \cite{book_Pazy}, in order to infer some regularity properties of \( u \). This result can be applied: indeed, under our hypotheses \( \Delta \) is defined, as a linear bounded operator, on the whole Banach space \( \ell^\infty(X) \); furthermore, we already know that \( \{P_t\}_{t\geq0} \) is a \( C_0 \)-semigroup on \( \ell^\infty(X) \). By applying this theorem, we infer that \( u \in C([0,+\infty);\ell^\infty(X)) \, \cap \, C^1((0,+\infty);\ell^\infty(X)) \), and that this function is actually a global classical solution of the problem
\begin{equation}\label{pb_ell_inf}
\begin{cases}
u'(t) - \Delta u(t) = F(u(t)) = \left[ u(t) \right]^p \qquad \quad \text{for all } t \in (0,+\infty) \\
u(0) = u_0,
\end{cases}
\end{equation}
where the equalities are intended between elements of \( \ell^\infty(X) \). Again, the identity \( F(u(t)) = |u(t)|^p = \left[ u(t) \right]^p \) is due to the nonnegativity of \( u \).

\medskip
\noindent
In particular, it is worth observing that the regularity properties of \(u\) are sufficiently strong to ensure pointwise regularity in time for each vertex \(x \in X\). Indeed, we already know that \( u \in C([0,+\infty); \ell^\infty(X)) \cap C^1((0,+\infty); \ell^\infty(X)) \). Therefore, both the map \(t \mapsto u(t)\) and its time derivative in \( \ell^\infty(X) \) are continuous with respect to the norm \( \| \cdot \|_\infty \). Consequently, for any fixed \(x \in X\), the scalar function \(t \mapsto u(x,t)\) inherits the same properties: it is continuous on \([0,+\infty)\) and continuously differentiable on \((0,+\infty)\). In conclusion, we infer the following:
\[
u(x,\cdot) \in C^0([0,+\infty)) \, \cap \, C^1((0,+\infty)) \hspace{1.6em} \text{for all } x \in X.
\]
In other words, \( u \) satisfies the condition \eqref{reg_sol_1_G}, with \( T = +\infty \).

\medskip
\noindent
In order to pass from the equalities in \eqref{pb_ell_inf}, understood in $\ell^\infty(X)$, to a pointwise formulation,
we exploit the definition of derivative in the Banach space $\ell^\infty(X)$.
Since
\[
u \in C^1((0,+\infty);\ell^\infty(X)),
\]
it follows that, for every $t>0$,
\[
\left\|
\frac{u(\cdot,t+h)-u(\cdot,t)}{h} - u'(t)
\right\|_\infty \longrightarrow 0
\qquad \text{as } h \to 0.
\]
Fix now an arbitrary vertex $x \in X$. From the previous convergence we obtain, for all $t>0$:
\[
\left|
\frac{u(x,t+h)-u(x,t)}{h} - \big(u'(t)\big)(x)
\right|
\le
\left\|
\frac{u(\cdot,t+h)-u(\cdot,t)}{h} - u'(t)
\right\|_\infty
\xrightarrow[h \to 0]{} \text{ } 0.
\]
Therefore, for all $(x,t) \in X \times (0,+\infty)$,
\[
\frac{\partial}{\partial t} u(x,t)
= \lim_{h \to 0} \frac{u(x,t+h)-u(x,t)}{h}
= \big(u'(t)\big)(x).
\]
Since the identity
\[
u'(t) = \Delta u(t) + [u(t)]^p
\]
holds in $\ell^\infty(X)$ for all $t>0$, we conclude that
\[
\frac{\partial}{\partial t} u(x,t)
= \big(u'(t)\big)(x)
= \Delta u(x,t) + [u(x,t)]^p
\quad \qquad \text{for all } (x,t) \in X \times (0,+\infty).
\]
Moreover, from $u(0)=u_0$ in $\ell^\infty(X)$ we immediately deduce
\[
u(x,0)=u_0(x)
\quad \qquad \text{for all } x\in X.
\]
Hence $u$ satisfies problem \eqref{problem_HR_G} pointwise in $X \times (0,+\infty)$.

\medskip
\noindent Finally, let us fix an arbitrary time instant \( T' > 0 \). By definition, the global mild solution \( u \) is nonnegative and belongs to the space \( L^\infty([0,T']; \ell^\infty(X)) \), hence it is straightforward to verify that \( u \in L^\infty(X \times [0,T']) \). By the arbitrariness of \( T' > 0 \), we conclude that \( u \) satisfies also the condition \eqref{reg_sol_2_G}, with \( T = +\infty \). All the requirements in \cref{definition_sub_super_sols_G} are then verified by \( u \), with \( T = +\infty \). This ends the proof. \hfill \( \square \)

\bigskip
\noindent \emph{Proof of \emph{\cref{proposition_existence_between_barriers_mild}}.} For practical purposes, we divide the proof into sequential steps.

\medskip
\noindent \textbf{(I)} We first introduce an iterative scheme, starting from the subsolution \( \underline{u} \). More specifically, we set
\[
u^{(0)} := \underline{u},
\]
and for each $n \in \mathbb{N}_0$ we define $u^{(n+1)}: X \times [0,\tau] \to [0,+\infty) $ as
\begin{equation}\label{eq:iteration-def}
u^{(n+1)}(\cdot,t) := P_t \, u_0 + \int_0^t \, P_{t-s} \left[ u^{(n)}(\cdot,s)^p \right] \, ds \qquad \quad \text{for all } t \in [0,\tau].
\end{equation}
By definition of mild subsolution, it holds \( u^{(0)} = \underline{u} \in L^\infty([0,\tau]; \ell^\infty(X)) \). Now, arguing by induction, we assume that \( u^{(n)} \in L^\infty([0,\tau]; \ell^\infty(X)) \). After fixing \( t \in [0,\tau] \), this fact implies, for each \( s \in [0,t] \), that
\[
\left\| P_{t-s} \left[ u^{(n)}(\cdot,s)^p \right] \right\|_\infty \leq \left\| \left[ u^{(n)}(\cdot,s)^p \right] \right\|_\infty = \left[ \| u^{(n)}(\cdot,s) \|_\infty \right]^p
\leq \left[ \| u^{(n)} \|_{L^\infty([0,\tau]; \ell^\infty(X))} \right]^p,
\]
where the first inequality is due to the positivity of \( P_{t-s} \). Now, by exploiting \eqref{eq:iteration-def}, together with the positivity of \( P_t \) and the last estimate, we obtain:
\[
\begin{aligned}
\| u^{(n+1)}(\cdot,t) \|_\infty & \leq \left\| P_t \, u_0 \right\|_\infty + \left\| \int_0^t \, P_{t-s} \left[ u^{(n)}(\cdot,s)^p \right] \, ds \right\|_\infty \\
& \leq \| u_0 \|_\infty + \int_0^t \, \left\| P_{t-s} \left[ u^{(n)}(\cdot,s)^p \right] \right\|_\infty \, ds \\
& \leq \| u_0 \|_\infty + \tau \, \left[ \| u^{(n)} \|_{L^\infty([0,\tau]; \ell^\infty(X))} \right]^p,
\end{aligned}
\]
where we notice that the right-hand side is independent of \( t \). Therefore, we infer:
\[
\| u^{(n+1)} \|_{L^\infty([0,\tau]; \ell^\infty(X))} = \sup_{t \in [0,\tau]} \, \| u^{(n+1)}(\cdot,t) \|_\infty \leq \| u_0 \|_\infty + \tau \, \left[ \| u^{(n)} \|_{L^\infty([0,\tau]; \ell^\infty(X))} \right]^p < +\infty,
\]
meaning that \( u^{(n+1)} \in L^\infty([0,\tau]; \ell^\infty(X)) \). In conclusion, we have \( u^{(n)} \in L^\infty([0,\tau]; \ell^\infty(X)) \) for all \( n \in \mathbb{N}_0 \). In particular, it is straightforward to verify that all the functions in the iterative scheme are actually well-defined on \( X \times [0,\tau] \).

\medskip
\noindent The nonnegativity of such functions is due to another induction reasoning. In particular, we exploit the fact that, by definition of mild subsolution, we have \( u^{(1)} = \underline{u} \geq 0 \). Moreover, assuming that \( u^{(n)} \) is nonnegative, the positivity of the heat semigroup and the nonnegativity of \( u_0 \) yield that \( u^{(n+1)} \geq 0 \), thanks to \eqref{eq:iteration-def}. In conclusion, for every \( n \in \mathbb{N}_0 \), the function \( u^{(n)} \) is well-defined from \( X \times [0,\tau] \) to \( [0,+\infty) \), and it belongs to the space \( L^\infty([0,\tau]; \ell^\infty(X)) \).

\medskip
\noindent \textbf{(II)} We now show by induction that
\begin{equation}\label{eq:monotone-increasing}
u^{(n)}(\cdot,t) \;\le\; u^{(n+1)}(\cdot,t) \qquad \quad \text{for all } t \in [0,\tau], \hspace{0.1em} n \in \mathbb{N}_0.
\end{equation}
For $n=0$, since $\underline{u}$ is a mild subsolution, for all \( t \in [0,\tau] \) we have:
\[
\begin{aligned}
u^{(0)}(\cdot,t) & = \underline{u}(\cdot,t) \\
& \leq P_t \, u_0 + \int_0^t \, P_{t-s} \left[ \underline{u}(\cdot,s)^p \right] \, ds
\\
& = P_t \, u_0 + \int_0^t \, P_{t-s} \left[ u^{(0)}(\cdot,s)^p \right] \, ds \\
& = u^{(1)}(\cdot,t),
\end{aligned}
\]
where the last equality exploits \eqref{eq:iteration-def}. Now, assume $u^{(n)} \le u^{(n+1)}$. Because the map $x \mapsto x^p$ is monotone increasing on $[0,\infty)$ and each function $u^{(k)}$ is nonnegative by construction, we have, for a fixed \( t \in [0,\tau] \):
\[
\left[ u^{(n)}(\cdot,s) \right]^p \;\le\; \left[ u^{(n+1)}(\cdot,s) \right]^p \qquad \quad \text{for all } s \in [0,t].
\]
Applying the operator $P_{t-s}$, which is positivity preserving, and integrating, we get
\[
P_t \, u_0 + \int_0^t P_{t-s} \left[ u^{(n)}(\cdot,s)^p \right] \, ds \;\le\; P_t \, u_0 + \int_0^t \, P_{t-s} \left[ u^{(n+1)}(\cdot,s)^p \right] \, ds.
\]
In other words, thanks to \eqref{eq:iteration-def}, it holds:
\[
u^{(n+1)}(\cdot,t) \;\le\; u^{(n+2)}(\cdot,t) \qquad \quad \text{for all } t \in [0,\tau].
\]
Thus \eqref{eq:monotone-increasing} holds, meaning that $\{u^{(n)}\}_{n\ge0}$ is a pointwise nondecreasing sequence on $X \times [0,\tau]$.

\medskip
\noindent \textbf{(III)} We now show by induction that
\begin{equation}\label{eq:dominated-above}
u^{(n)}(\cdot,t) \;\le\; \overline{u}(\cdot,t) \qquad \quad \text{for all } t \in [0,\tau], \hspace{0.1em} n \in \mathbb{N}_0.
\end{equation}
For $n=0$, this fact is true by the hypothesis $u^{(0)} = \underline{u} \le \overline{u}$. Now, let us assume $u^{(n)} \le \overline{u}$. Again, since $x \mapsto x^p$ is monotone increasing on $[0,\infty)$, and each function $u^{(k)}$ is nonnegative by construction, we have, for a fixed \( t \in [0,\tau] \):
\[
\left[ u^{(n)}(\cdot,s) \right]^p \;\le\; \left[ \overline{u}(\cdot,s) \right]^p \qquad \quad \text{for all } s \in [0,t].
\]
Applying the operator $P_{t-s}$, which is positivity preserving, and integrating, we get
\[
u^{(n+1)}(\cdot,t) = P_t \, u_0 + \int_0^t P_{t-s} \left[ u^{(n)}(\cdot,s)^p \right] \, ds \;\le\; P_t \, u_0 + \int_0^t \, P_{t-s} \left[ \overline{u}(\cdot,s)^p \right] \, ds \leq \overline{u}(\cdot,t),
\]
where the last inequality holds since $\overline{u}$ is a mild supersolution.
Therefore, we infer:
\[
u^{(n+1)}(\cdot,t) \;\le\; \overline{u}(\cdot,t) \qquad \quad \text{for all } t \in [0,\tau],
\]
so that \eqref{eq:dominated-above} holds. As a consequence of the results obtained up to now, we have:
\begin{equation}\label{ordering_inequalities}
0 \;\le\; \underline{u}(\cdot,t) \;\le\; u^{(n)}(\cdot,t) \;\le\; \overline{u}(\cdot,t) \qquad \quad \text{for all } t \in [0,\tau], \hspace{0.1em} n \in \mathbb{N}_0.
\end{equation}

\medskip
\noindent \textbf{(IV)} We know that, for each \( (x,t) \in X \times [0,\tau] \), the sequence \( \left\{ u^{(n)}(x,t) \right\} \) is nondecreasing, hence we can define the pointwise limit
\[
u(x,t) := \lim_{n\to\infty} u^{(n)}(x,t) \qquad \quad \text{for all } (x,t) \in X \times [0,\tau].
\]
Passing to the limit as \( n \to +\infty \) in \eqref{ordering_inequalities} yields the following:
\begin{equation}\label{ordering_u}
0 \;\le\; \underline{u}(\cdot,t) \;\le\; u(\cdot,t) \;\le\; \overline{u}(\cdot,t) \qquad \quad \text{for all } t \in [0,\tau],
\end{equation}
which also implies that \( u \in L^\infty([0,\tau];\ell^\infty(X)) \), since \( \overline{u} \) belongs to this space by definition.

\medskip
\noindent \textbf{(V)} We are only left with showing that the function $u$ satisfies the mild equation \eqref{eq:nonlinear-mild}. In order to do so, after fixing \( t \in [0,\tau] \), we let \( n \to +\infty \) inside the recursive relation \eqref{eq:iteration-def}, namely
\[
u^{(n+1)}(\cdot,t) = P_t \, u_0 + \int_0^t \, P_{t-s} \left[ u^{(n)}(\cdot,s)^p \right] \, ds.
\]
Let us fix \(x\in X\), \(t \in [0,\tau]\), and \(s\in[0,t]\). Consider the nondecreasing sequence \(u^{(n)}(\cdot,s) \), converging to \( u(\cdot,s)\) in a monotone way, and define, for each \(n \in \mathbb{N}_0\),
\[
f_n(y):=p(x,y,t-s) \left[u^{(n)}(y,s)\right]^{p} \qquad \quad \text{for all } y\in X,
\]
with \( p \) denoting the heat kernel. Notice that \( p \) is nonnegative, since the heat semigroup is assumed to be positivity preserving. Now, thanks to the nonnegativity of the functions \( u^{(n)} \), we infer that \( f_n \geq 0 \). Moreover, the monotone convergence of \(u^{(n)}(\cdot,s) \) to \( u(\cdot,s)\) implies that \( \left\{f_n\right\} \) is monotone nondecreasing and converges to the function \( f \), defined as
\[
f(y):=p(x,y,t-s) \left[u(y,s) \right]^{p} \qquad \quad \text{for all } y\in X.
\]
We can then apply the Monotone Convergence Theorem on the measure space \((X,\mathcal P(X),\mu)\), obtaining:
\[
\lim_{n \to +\infty} \, \sum_{y\in X} p(x,y,t-s) \left[ u^{(n)}(y,s) \right]^{p} \, \mu(y) = \sum_{y\in X} p(x,y,t-s) \left[ u(y,s) \right]^{p}\,\mu(y).
\]
By the arbitrariness of \( x \in X \), we get the following identity, holding pointwise in \( X \):
\begin{equation}\label{first_MCT_semigroup}
\lim_{n \to +\infty} \, P_{t-s} \left[ u^{(n)}(\cdot,s)^{p} \right] = P_{t-s} \left[ u(\cdot,s)^{p} \right].
\end{equation}
Now, under our assumptions, for all \( s \in [0,t] \) it holds:
\[
0 \le P_{t-s} \left[u^{(n)}(\cdot,s)^p\right] \le P_{t-s} \left[ \overline{u}(\cdot,s)^p \right]
\le \left\| P_{t-s} \left[ \overline{u}(\cdot,s)^p \right] \right\|_\infty \le \|\overline{u}(\cdot,s)^p\|_\infty = \left[ \|\overline{u}(\cdot,s)\|_\infty \right]^p.
\]
In particular, we obtain:
\[
0 \le P_{t-s} \left[u^{(n)}(\cdot,s)^p\right] \le \left[ \|\overline{u}(\cdot,s)\|_\infty \right]^p \leq \left[ \| \overline{u} \|_{L^\infty([0,\tau];\ell^\infty(X))} \right]^p,
\]
and the last bound is independent of $s$, hence integrable over the finite time interval \( [0,t] \). This fact, together with \eqref{first_MCT_semigroup}, allows us to apply the Dominated Convergence Theorem, yielding the following:
\[
\lim_{n \to +\infty} \, \int_0^t \, P_{t-s} \left[ u^{(n)}(\cdot,s)^p \right] \, ds = \int_0^t \, P_{t-s} \left[ u(\cdot,s)^p \right] \, ds.
\]
Therefore, by exploiting \eqref{eq:iteration-def}, we infer:
\[
\lim_{n \to +\infty} \, u^{(n+1)}(\cdot,t) = P_t \, u_0 + \lim_{n \to +\infty} \, \int_0^t \, P_{t-s} \left[ u^{(n)}(\cdot,s)^p \right] \, ds
= P_t \, u_0 + \int_0^t \, P_{t-s} \left[ u(\cdot,s)^p \right] \, ds,
\]
holding for all \( t \in [0,\tau] \). We now notice that the left-hand side of the obtained identity corresponds exactly to $u(\cdot,t)$, by definition of $u$. We conclude that, for all $t \in [0,\tau]$, it holds:
\[
u(\cdot,t) = P_t \, u_0 + \int_0^t \, P_{t-s} \left[ u(\cdot,s)^p \right] \, ds,
\]
which is \eqref{eq:nonlinear-mild}. In particular, since \( P_0 \) coincides with the identity operator, we also infer that \( u(\cdot,0) = u_0 \). By combining these results with the fact that \( u \in L^\infty([0,\tau];\ell^\infty(X)) \), we conclude that $u$ is a mild solution to problem \eqref{problem_HR_G}, with initial datum \( u_0 \), on \( [0,\tau] \). Finally, thanks to \eqref{ordering_u}, the proof is complete. \hfill \( \square \)

\medskip
\begin{remark}\label{remark_proposition_existence_mild_ordered}
\noindent From the statement of \emph{\cref{proposition_existence_between_barriers_mild}}, it is trivial to infer that, given an initial datum \( u_0 \) for problem \eqref{problem_HR_G} satisfying \eqref{eq:HP_ID_G}, if there exist a global mild subsolution \( \underline{u} \) and a global mild supersolution \( \overline{u} \) of problem \eqref{problem_HR_G} such that
\[
\underline{u}(\cdot,t) \leq \overline{u}(\cdot,t) \qquad \quad \text{for all } \hspace{0.05em} t \in [0, +\infty),
\]
\noindent then the existence of a global mild solution \( u \) is guaranteed. In particular, in this scenario, it holds:
\[
\underline{u}(\cdot,t) \leq u(\cdot,t) \leq \overline{u}(\cdot,t) \qquad \quad \text{for all } \hspace{0.05em} t \in [0, +\infty).
\]
\end{remark}

\bigskip
\noindent \emph{Proof of \emph{\cref{proposition_summarize_global_existence}}.} The thesis directly follows by combining \cref{remark_properties_SG_connected_locally_finite}, \cref{proposition_existence_between_barriers_mild}, \cref{remark_proposition_existence_mild_ordered}, and \cref{proposition_mild_implies_classical_global_general_graph}. \hfill \( \square \)

\section{Blow-up and global existence on $\mathbb Z^N$: proofs}\label{section_proofs}


\subsection{Proof of Theorem \ref{theorem_blow_up_lattice_sub_critical_p}}\label{subsection_proofs_1}

\begin{definition}\label{definition_theta_function_lattice}
\noindent We define the theta function \( \theta : (0,+\infty) \to (0,+\infty) \) as
\[
\theta(s) := \sum_{m \in \mathbb{Z}} e^{- \pi \hspace{0.05em} m^2 \hspace{0.05em} s} \hspace{2.5em} \text{for all } \hspace{0.1em} s > 0.
\]
\end{definition}

\noindent We highlight that the theta function is well-defined. Indeed, it is immediate to see that \( \theta \) is strictly positive. Furthermore, after fixing an arbitrary \( s \in (0,+\infty) \), it is straightforward to obtain the following:
\begin{equation}\label{technical_passages_theta_function}
\theta(s) = 1 + \sum_{m \in \mathbb{Z} \setminus \{0\}} e^{- \pi \hspace{0.05em} m^2 \hspace{0.05em} s}
= 1 + 2 \sum_{m=1}^{+\infty} e^{- \pi \hspace{0.05em} m^2 \hspace{0.05em} s}
\leq 1 + 2 \sum_{m=1}^{+\infty} \left[ e^{- \pi s} \right]^m,
\end{equation}
\noindent and the last series converges since \( e^{- \pi s} \in (0,1) \).

\medskip
\noindent We shall now state two auxiliary results, which will be exploited in order to prove \cref{theorem_blow_up_lattice_sub_critical_p}.

The estimate contained in the next lemma can be easily derived. Throughout the following, we shall employ the notation \( |\cdot| \) to denote the cardinality of a set.

\begin{lemma}\label{remark_balls_lattice}
\noindent For any fixed vertex \( x \in \mathbb{Z}^N \) and for any \( r > 0 \), the following holds:
\[
|\overline{B_r(x)}| \geq c_N \, r^N,
\]
\noindent where \( c_N := N^{-\frac{N}{2}} > 0 \).
\end{lemma}

\medskip
\noindent The following blow-up result will be expedient in the proof of \cref{theorem_blow_up_lattice_sub_critical_p}; for its proof we refer the reader to \cite{PZ1}.

\begin{lemma}\label{lemma_blow_up_lattice_1}
\noindent Let \( 1 < p \leq 1 + \frac{2}{N} \), and consider a function \( u_0 \) satisfying \eqref{eq:HP_ID_G}, with \( u_0 \not \equiv 0 \). Let \( u \) solve problem \eqref{problem_HR_G} with $X=\mathbb Z^N$. Moreover, suppose that there exists \( k > 0 \) such that:
\begin{equation}\label{condition_blow_up_G_lemma_1_new}
\sum_{x \in \mathbb{Z}^N} e^{-k \hspace{0.05em} |x|^2} \hspace{0.1em} u_0(x) \, \mu(x) >
(2N)^{\frac{p}{p-1}} \left[ \theta \left( \frac{k}{\pi} \right) \right]^N k^{\frac{1}{p-1}}.
\end{equation}
\noindent Then \( u \) blows up in finite time.
\end{lemma}

\medskip
\noindent \emph{Proof of \emph{\cref{theorem_blow_up_lattice_sub_critical_p}}}. We separately consider two distinct cases.

\medskip
\noindent \emph{a) Case} \( 1 < p < 1 + \frac{2}{N} \).

\medskip
\noindent We work under the hypothesis \( 1 < p < 1 + 2/N \), assuming that \( u_0 \not \equiv 0 \) satisfies \eqref{eq:HP_ID_G} and that \( u \not \equiv 0 \) solves problem \eqref{problem_HR_G} with initial condition \( u_0 \).

\medskip
\noindent We first notice that, since on the integer lattice it holds \( \mu \equiv 2N \), then it is easily seen that condition \eqref{condition_blow_up_G_lemma_1_new} is equivalent to
\begin{equation}\label{condition_blow_up_G_lemma_1}
\sum_{x \in \mathbb{Z}^N} e^{-k \hspace{0.05em} |x|^2} \hspace{0.1em} u_0(x) > \left[ \theta \left( \frac{k}{\pi} \right) \right]^N (2kN)^{\frac{1}{p-1}}.
\end{equation}
\noindent The goal is to prove that \( u \) is nonglobal. In particular, we want to show that inequality \eqref{condition_blow_up_G_lemma_1} holds, for some \( k > 0 \). Indeed, if this is the case, then \eqref{condition_blow_up_G_lemma_1_new} holds, hence \cref{lemma_blow_up_lattice_1} can be applied, yielding that \( u \) blows up in finite time, which is the thesis.

\medskip
\noindent We then proceed by investigating the behavior of both the left-hand side and the right-hand side of \eqref{condition_blow_up_G_lemma_1} as \( k \to 0^+ \).

\medskip
\noindent \textbf{(I)} First, consider the left-hand side of \eqref{condition_blow_up_G_lemma_1}. Thanks to the fact that, as \( k \to 0^+ \), it holds \( e^{-k |\cdot|^2} \to 1 \) pointwise in \( \mathbb{Z}^N \), we obtain:
\[
\lim_{k \to 0^+} e^{-k |x|^2} u_0(x) = u_0(x) \hspace{1.6em} \text{for all } x \text{ in } \mathbb{Z}^N.
\]
\noindent Now, for any fixed \( k > 0 \), the function \( e^{-k |\cdot|^2} u_0 \) is nonnegative, since by \eqref{eq:HP_ID_G} we have \( u_0 \geq 0 \) in \( \mathbb{Z}^N \). Moreover, for any fixed \( x \) in \( \mathbb{Z}^N \) and for all positive \( k_1, k_2 \) satisfying \( k_1 \leq k_2 \), it holds:
\[
  e^{- k_2 |x|^2} u_0(x) \leq e^{- k_1 |x|^2} u_0(x).
\]
\noindent Therefore, the Monotone Convergence Theorem, applied in the context of the measure space \( (\mathbb{Z}^N,\mathcal{P}(\mathbb{Z}^N),\mu) \), yields:
\[
\lim_{k \to 0^+} \sum_{x \in \mathbb{Z}^N} e^{-k |x|^2} u_0(x) \hspace{0.02em} \mu(x) = \sum_{x \in \mathbb{Z}^N} u_0(x) \hspace{0.02em} \mu(x).
\]
\noindent Now, since \( \mu \equiv 2N \) on \( \mathbb{Z}^N \), we can divide both sides by \( 2N \), obtaining:
\[
\lim_{k \to 0^+} \sum_{x \in \mathbb{Z}^N} e^{-k |x|^2} u_0(x) = \sum_{x \in \mathbb{Z}^N} u_0(x).
\]
\noindent In particular, as remarked above, it holds \( u_0 \not \equiv 0 \), and thanks to \eqref{eq:HP_ID_G} we also have \( u_0 \geq 0 \) pointwise in \( \mathbb{Z}^N \). Therefore, there must exist a vertex \( x^* \in \mathbb{Z}^N \) such that \( u_0(x^*) > 0 \), which implies:
\[
\sum_{x \in \mathbb{Z}^N} u_0(x) \geq u_0(x^*) > 0.
\]
\noindent In conclusion, we infer:
\[
\lim_{k \to 0^+} \sum_{x \in \mathbb{Z}^N} e^{-k |x|^2} u_0(x) > 0.
\]

\medskip
\noindent \textbf{(II)} Next, we consider the right-hand side of \eqref{condition_blow_up_G_lemma_1}, namely
\begin{equation}\label{definition_C_k_N_p}
C_{k,N,p} := \left[ \theta \left( \frac{k}{\pi} \right) \right]^N (2kN)^{\frac{1}{p-1}} \equiv (2N)^{\frac{1}{p-1}} \left[ \theta \left( \frac{k}{\pi} \right) \right]^N k^{\frac{1}{p-1}}.
\end{equation}
\noindent First, let us fix \( k > 0 \). From the beginning of this subsection, we recall the definition of the theta function \( \theta : (0,+\infty) \to (0,+\infty) \), namely
\[
\theta(s) = \sum_{m \in \mathbb{Z}} e^{- \pi \hspace{0.05em} m^2 \hspace{0.05em} s} \hspace{2.5em} \text{for all } \hspace{0.1em} s > 0.
\]
\noindent We now make use of {Theorem 3.2} in \cite{book_Stein_Shakarchi}, which states that
\[
\theta(s) = \sqrt{\frac{1}{s}} \hspace{0.4em} \theta \hspace{-0.25em} \left( \frac{1}{s} \right) \hspace{2.5em} \text{for all } \hspace{0.1em} s > 0.
\]
\noindent In particular, by choosing \( s = k / \pi > 0 \), we obtain the following identity:
\begin{equation}\label{identity_theta_pi_k_1}
\theta\left( \frac{k}{\pi} \right) = \sqrt{\frac{\pi}{k}} \hspace{0.4em} \theta \hspace{-0.25em} \left( \frac{\pi}{k} \right).
\end{equation}
\noindent In addition, in \eqref{technical_passages_theta_function} we obtained the following alternative expression for the theta function:
\[
\theta(s) = 1 + 2 \sum_{m=1}^{+\infty} e^{- \pi \hspace{0.05em} m^2 \hspace{0.05em} s} \hspace{2.5em} \text{for all } \hspace{0.1em} s > 0.
\]
\noindent Here, the choice \( s = \pi / k > 0 \) implies
\[
\theta\left( \frac{\pi}{k} \right) = 1 + 2 \sum_{m=1}^{+\infty} e^{- \frac{\pi^2 \hspace{0.05em} m^2}{k}}.
\]
\noindent Combining this relation with \eqref{identity_theta_pi_k_1}, we infer:
\begin{equation}\label{identity_theta_pi_k_2}
\theta\left( \frac{k}{\pi} \right) = \sqrt{\frac{\pi}{k}} \hspace{0.4em} \theta \hspace{-0.25em} \left( \frac{\pi}{k} \right)
= \sqrt{\frac{\pi}{k}} \left[ 1 + 2 \sum_{m=1}^{+\infty} e^{- \frac{\pi^2 \hspace{0.05em} m^2}{k}} \right].
\end{equation}
\noindent Arguing exactly as in \eqref{technical_passages_theta_function} we also get, for all \( s > 0 \):
\[
0 < \sum_{m=1}^{+\infty} e^{- \pi \hspace{0.05em} m^2 \hspace{0.05em} s} \leq \sum_{m=1}^{+\infty} \left[ e^{- \pi s} \right]^m = \frac{e^{- \pi s}}{1 - e^{- \pi s}},
\]
\noindent where the last equality exploits the well-known formula for the geometric series, which holds since in this case the common ratio satisfies
\[
0 < e^{- \pi s} < 1.
\]
\noindent In particular, by choosing \( s = \pi / k > 0 \), we deduce the following estimates:
\[
0 < \sum_{m=1}^{+\infty} e^{- \frac{\pi^2 \hspace{0.05em} m^2}{k}} \leq \frac{e^{- \frac{\pi^2}{k}}}{1 - e^{- \frac{\pi^2}{k}}}.
\]
\noindent Thanks to the arbitrariness of \( k > 0 \), we can now pass to the limit as \( k \to 0^+ \), obtaining:
\[
0 \leq \lim_{k \to 0^+} \sum_{m=1}^{+\infty} e^{- \frac{\pi^2 \hspace{0.05em} m^2}{k}} \leq \lim_{k \to 0^+} \frac{e^{- \frac{\pi^2}{k}}}{1 - e^{- \frac{\pi^2}{k}}} = 0,
\]
\noindent where the last equality exploits a trivial limit computation. Therefore, it holds:
\[
\lim_{k \to 0^+} \sum_{m=1}^{+\infty} e^{- \frac{\pi^2 \hspace{0.05em} m^2}{k}} = 0,
\]
\noindent so that
\[
\lim_{k \to 0^+} \left[ \sqrt{\frac{k}{\pi}} \hspace{0.4em} \theta \hspace{-0.25em} \left( \frac{k}{\pi} \right) \right] = \lim_{k \to 0^+} \left[ 1 + 2 \sum_{m=1}^{+\infty} e^{- \frac{\pi^2 \hspace{0.05em} m^2}{k}} \right] = 1 + 2 \lim_{k \to 0^+} \sum_{m=1}^{+\infty} e^{- \frac{\pi^2 \hspace{0.05em} m^2}{k}} = 1 + 2 \cdot 0 = 1,
\]
\noindent where, in the first equality, the identity \eqref{identity_theta_pi_k_2} was used. In particular, we infer:
\[
\theta \hspace{-0.25em} \left( \frac{k}{\pi} \right) \sim \sqrt{\frac{\pi}{k}} \hspace{2.0em} \text{as } k \to 0^+.
\]
\noindent We can now exploit \eqref{definition_C_k_N_p} in order to get:
\[
C_{k,N,p} = (2N)^{\frac{1}{p-1}} \left[ \theta \left( \frac{k}{\pi} \right) \right]^N k^{\frac{1}{p-1}} \hspace{0.2em} \sim \hspace{0.2em}
(2N)^{\frac{1}{p-1}} \left[ \sqrt{\frac{\pi}{k}} \right]^N k^{\frac{1}{p-1}} \hspace{2.0em} \text{as } k \to 0^+.
\]
\noindent Finally, rewriting the term at the right-hand side of the previous asymptotic estimate yields:
\[
(2N)^{\frac{1}{p-1}} \left[ \sqrt{\frac{\pi}{k}} \right]^N k^{\frac{1}{p-1}}
= (2N)^{\frac{1}{p-1}} \left( \frac{\pi}{k} \right)^{\frac{N}{2}} k^{\frac{1}{p-1}}
= \left[ (2N)^{\frac{1}{p-1}} \hspace{0.2em} \pi^{\frac{N}{2}} \right] k^{\frac{1}{p-1} - \frac{N}{2}}.
\]
\noindent Therefore, we get:
\[
C_{k,N,p} \hspace{0.2em} \sim \hspace{0.2em} \left[ (2N)^{\frac{1}{p-1}} \hspace{0.2em} \pi^{\frac{N}{2}} \right] k^{\frac{1}{p-1} - \frac{N}{2}} \hspace{2.0em} \text{as } k \to 0^+,
\]
\noindent so that
\[
\lim_{k \to 0^+} C_{k,N,p} = \lim_{k \to 0^+} \left\{ \left[ (2N)^{\frac{1}{p-1}} \hspace{0.2em} \pi^{\frac{N}{2}} \right] k^{\frac{1}{p-1} - \frac{N}{2}} \right\}
= \left[ (2N)^{\frac{1}{p-1}} \hspace{0.2em} \pi^{\frac{N}{2}} \right] \lim_{k \to 0^+} k^{\frac{1}{p-1} - \frac{N}{2}} = 0,
\]
\noindent where the last equality comes from the fact that we are working under the hypothesis \( 1 < p < 1 + 2/N \), and thus it holds \( \frac{1}{p-1} - \frac{N}{2} > 0 \). In conclusion, the right-hand side of \eqref{condition_blow_up_G_lemma_1} tends to zero as \( k \to 0^+ \), namely:
\[
\lim_{k \to 0^+} \left\{ \left[ \theta \left( \frac{k}{\pi} \right) \right]^N (2kN)^{\frac{1}{p-1}} \right\} \equiv \lim_{k \to 0^+} C_{k,N,p} = 0.
\]

\medskip
\noindent Combining the results obtained in \textbf{(I)} and \textbf{(II)}, we get:
\[
\lim_{k \to 0^+} \sum_{x \in \mathbb{Z}^N} e^{-k |x|^2} u_0(x) > \lim_{k \to 0^+} \left\{ \left[ \theta \left( \frac{k}{\pi} \right) \right]^N (2kN)^{\frac{1}{p-1}} \right\},
\]
\noindent which implies that there must exist a small \( \bar{k} > 0 \) for which it holds
\[
\sum_{x \in \mathbb{Z}^N} e^{- \bar{k} |x|^2} u_0(x) > \left[ \theta \left( \frac{\bar{k}}{\pi} \right) \right]^N \left( 2 \bar{k} N \right)^{\frac{1}{p-1}}.
\]
\noindent Notice that this inequality exactly corresponds to \eqref{condition_blow_up_G_lemma_1}, with \( k = \bar{k} \). As previously remarked, this yields the thesis.

\medskip
\medskip
\noindent \emph{b) Case} \( p = 1 + \frac{2}{N} \).

\medskip
\noindent In order to ease the notation, throughout the following we will not explicitly highlight the dependence of the involved constants on the spatial dimension \( N \). We work under the hypothesis \( p = 1 + 2/N \), assuming that \( u_0 \not \equiv 0 \) satisfies \eqref{eq:HP_ID_G} and that \( u \not \equiv 0 \) solves problem \eqref{problem_HR_G} with initial condition \( u_0 \). Our goal is to show that \( u \) blows up in finite time.

\medskip
\noindent By contradiction, we assume that \( u \) is a global solution, namely \( u(\cdot, t) \in \ell^\infty(\mathbb{Z}^N) \) for all \( t > 0 \). Since the integer lattice is connected, locally finite and stochastically complete, we can invoke \cref{proposition_Duhamel_formula_HR_G}, yielding that the global solution \( u \) satisfies identity \eqref{eq:Duhamel_G} pointwise in \( \mathbb{Z}^N \times (0,+\infty) \), namely:
\[
u(x,t) = (P_t \, u_0)(x) + \int_0^t \left( P_{t-s} \, \left[u(\cdot,s) \right]^p \right)(x) \, ds \hspace{2.5em} \text{for all } \hspace{0.1em} (x,t) \in \hspace{0.05em} \mathbb{Z}^N \times (0,+\infty).
\]
\noindent Thanks to \eqref{definition_convolution_lattice}, this identity can be rewritten as
\begin{equation}\label{definition_Duhamel_lattice}
u(x,t) = (K_t * u_0)(x) + \int_0^t \left( K_{t-s} * \left[u(\cdot,s) \right]^p \right)(x) \, ds,
\end{equation}
\noindent holding for all \( (x,t) \in \mathbb{Z}^N \times (0,+\infty) \). Here, the convolution operator is intended as in \eqref{definition_convolution_lattice}, and \( K \) denotes the heat kernel associated to the heat equation posed on the integer lattice; the analytical expression of such function was given in \eqref{definition_K_t_HK_lattice}. Notice that the two convolutions appearing in \eqref{definition_Duhamel_lattice} are well-defined. Indeed, by \eqref{eq:HP_ID_G}, we have \( u_0 \in \ell^\infty(\mathbb{Z}^N) \); in addition, for any \( s > 0 \) it holds \( \left[u(\cdot,s) \right]^p \in \ell^\infty(\mathbb{Z}^N) \), thanks to the fact that \( u \) is a global solution. Now, as clarified in \cref{subsection_H_E_graph_setting}, the convolution with the heat kernel \( K \) is well-defined for bounded functions.

\medskip
\noindent Now, since \( u_0 \not \equiv 0 \), then there must exist a vertex \( y_0 \in \mathbb{Z}^N \) such that \( u_0(y_0) > 0 \). In particular, for all \( (x,t) \text{ in } \mathbb{Z}^N \times (0,+\infty) \) it holds:
\begin{equation}\label{geq_1_lattice}
  u(x,t) \geq (K_t * u_0)(x) = 2N \sum_{y \in \mathbb{Z}^N} K_t(x-y) \, u_0(y) \geq C_1 \, K_t(x-y_0) > 0,
\end{equation}
\noindent where \( C_1 := 2 N u_0(y_0) > 0 \). Indeed, we recall, from property b) in \cref{proposition_properties_HK_lattice}, that \( K_{h} > 0 \) in \( \mathbb{Z}^N \), for all \( h > 0 \). Now, the fact that \( u \geq 0 \) in \( \mathbb{Z}^N \times (0,+\infty) \), together with the definition of convolution, yields that the time integral in \eqref{definition_Duhamel_lattice} is nonnegative, which justifies the first inequality in \eqref{geq_1_lattice}. Instead, the equality is due to \eqref{definition_convolution_lattice}, and the following passages are trivial.

\medskip
\noindent If we now fix an arbitrary time \( t^* > 0 \) and set \( u_0^* := u(\cdot, t^*) \) in \( \mathbb{Z}^N \), then by \cref{remark_time_translation_G} \( u \) solves the following problem:
\begin{equation}\label{translated_lattice}
\begin{cases}
u_t - \Delta u = u^p & \text{ in } \mathbb{Z}^N \times (t^*, +\infty) \\
u = u_0^* & \text{ in } \mathbb{Z}^N \times \{ t^* \} \\
u \geq 0 & \text{ in } \mathbb{Z}^N \times (t^*, +\infty),
\end{cases}
\end{equation}
\noindent with \( u_0^* \) satisfying \eqref{eq:HP_ID_G}. Notice that \eqref{translated_lattice} is a translation in time of problem \eqref{problem_HR_G} for which the initial instant is \( t^* > 0 \). Therefore, we consider a new function \( u^* \), defined as
\begin{equation}\label{relation_star_lattice}
u^*(\cdot, t) := u(\cdot, t + t^*) \hspace{2.0em} \text{for every } t \geq 0.
\end{equation}
\noindent Now, given the fact that \( u \) is a global solution to problem \eqref{problem_HR_G}, it is not difficult to verify that \( u^* \) solves pointwise \eqref{problem_HR_G}, with initial condition \( u_0^* = u(\cdot, t^*) = u^*(\cdot, 0) \), and that \( u^* \geq 0 \) in \( \mathbb{Z}^N \times (0,+\infty) \). In addition, \( u^* \) satisfies the regularity properties \eqref{reg_sol_1_G} and \eqref{reg_sol_2_G} in \cref{definition_sub_super_sols_G}, with \( T = +\infty \). In conclusion, \( u^* \) is a global solution to \eqref{problem_HR_G}. Therefore, we can now exploit the theory previously developed regarding our problem, but considering the unknown \( u^* \) and the initial datum \( u_0^* \).

\medskip
\noindent We explicitly notice that, thanks to \eqref{geq_1_lattice} and \eqref{relation_star_lattice}, we obtain both \( u^* \not \equiv 0 \) and \( u_0^* \not \equiv 0 \). More specifically, it holds:
\[
u^*(x,t) = u(x, t + t^*) \geq C_1 \, K_{t + t^*}(x-y_0) > 0 \hspace{2.0em} \text{ for all } (x,t) \in \mathbb{Z}^N \times [0,+\infty),
\]
\noindent which is a trivial consequence of \eqref{geq_1_lattice}. Therefore, we obtain:
\begin{equation}\label{geq_2_lattice}
\left[ u^*(x,t) \right]^p \geq (C_1)^p \left[ K_{t + t^*}(x-y_0) \right]^p \hspace{2.5em} \text{for all } (x,t) \text{ in } \mathbb{Z}^N \times [0,+\infty).
\vspace{0.3em}
\end{equation}
\noindent In addition, for all \( (x,t) \in \mathbb{Z}^N \times (0,+\infty) \) we can write:
\begin{equation}\label{geq_3_lattice}
\begin{aligned}
u^*(x,t) & = (K_t * u_0^*)(x) + \int_0^t (K_{t-s} * \left[ u^*(\cdot,s) \right]^p)(x) \, ds \\
& \geq \int_0^t (K_{t-s} * \left[ u^*(\cdot,s) \right]^p)(x) \, ds,
\end{aligned}
\end{equation}
\noindent where we have first exploited \eqref{definition_Duhamel_lattice}, and then we have made use of the inequality \( K_t * u_0^* \geq 0 \), holding pointwise in \( \mathbb{Z}^N \) thanks to the positivity of the heat kernel and the fact that \( u_0^* \geq 0 \) in \( \mathbb{Z}^N \). Now, by combining \eqref{geq_3_lattice} with the definition of convolution, we obtain, for all \( (x,t) \in \mathbb{Z}^N \times (0,+\infty) \):
\begin{equation}\label{geq_5_lattice}
\begin{aligned}
u^*(x,t) & \geq \int_0^t (K_{t-s} * \left[ u^*(\cdot,s) \right]^p)(x) \, ds \\
& = 2N \int_0^t \sum_{y \in \mathbb{Z}^N} K_{t-s}(x-y) \left[ u^*(y,s) \right]^p \, ds \\
& \geq C_2 \int_0^t \sum_{y \in \mathbb{Z}^N} K_{t-s}(x-y) \left[ K_{s + t^*}(y - y_0) \right]^p \, ds,
\end{aligned}
\end{equation}
\noindent where \( C_2 := 2N (C_1)^p \equiv 2N (C_1)^{1 + 2/N} > 0 \). Here, the last equality uses \eqref{geq_2_lattice} and the positivity of the heat kernel.

\medskip
\noindent By summing over \( \mathbb{Z}^N \) in \eqref{geq_5_lattice}, we then get:
\begin{equation}\label{geq_6_lattice}
\sum_{x \in \mathbb{Z}^N} u^*(x,t) \geq C_2 \int_0^t \sum_{x,y \in \mathbb{Z}^N} K_{t-s}(x-y) \left[ K_{s + t^*}(y - y_0) \right]^p \, ds \hspace{3.0em} \text{for all } t > 0,
\end{equation}
\noindent where Tonelli’s Theorem is implicitly used, allowing the interchange of the order of integration and sum.

\medskip
\noindent Now, let us fix a time instant \( t > 2 \) and \( s \in (1,t-1) \). Moreover, we define the set
\[
\begin{aligned}
R_s := & \left\{ (x,y) \in \mathbb{Z}^N \times \mathbb{Z}^N : |y - y_0| \leq \sqrt{\frac{s+t^*}{N}}, \hspace{0.3em} |x - y| \leq \sqrt{\frac{t-s}{N}} \right\} \\
= & \left\{ (x,y) \in \mathbb{Z}^N \times \mathbb{Z}^N : y \in \overline{B_{\sqrt{\frac{s+t^*}{N}}}(y_0)}, \hspace{0.3em} x \in \overline{B_{\sqrt{\frac{t-s}{N}}}(y)} \right\}.
\end{aligned}
\]
\noindent For notational convenience, for the moment we also set
\begin{equation}\label{definitions_temporary_S}
S_{y_0} := \overline{B_{\sqrt{\frac{s+t^*}{N}}}(y_0)}, \hspace{2.5em} S_y := \overline{B_{\sqrt{\frac{t-s}{N}}}(y)},
\end{equation}
\noindent so that it holds:
\begin{equation}\label{cardinality_set_R_s_1}
\begin{aligned}
|R_s| & = \sum_{(x,y) \in R_s} 1 \\
& = \sum_{y \in S_{y_0}} \left[ \sum_{x \in S_y} 1 \right] \\
& \equiv \sum_{y \in S_{y_0}} |S_y| \\
& \geq \min \limits_{z \in S_{y_0}} |S_z| \left[ \sum_{y \in S_{y_0}} 1 \right] \\
& \equiv |S_{y_0}| \left[ \min \limits_{z \in S_{y_0}} |S_z| \right].
\end{aligned}
\end{equation}
\noindent From \cref{remark_balls_lattice}, we know that, for any fixed vertex \( w \in \mathbb{Z}^N \) and for all \( r > 0 \), it holds:
\[
|\overline{B_r(w)}| \geq C_3 \, r^N,
\]
\noindent where \( C_3 := N^{-\frac{N}{2}} > 0 \). Therefore, by combining \eqref{definitions_temporary_S} and \eqref{cardinality_set_R_s_1}, we get:
\begin{equation}\label{cardinality_set_R_s_2}
\begin{aligned}
|R_s| & \geq |S_{y_0}| \left[ \min \limits_{z \in S_{y_0}} |S_z| \right] \\
& \geq (C_3)^2 \left( \sqrt{\frac{s+t^*}{N}} \right)^N \left( \sqrt{\frac{t-s}{N}} \right)^N \\
& \equiv C_4 \left( s + t^* \right)^{\frac{N}{2}} \left( t - s \right)^{\frac{N}{2}},
\end{aligned}
\end{equation}
\noindent where \( C_4 := (C_3)^2 \, N^{-N} > 0 \).

\medskip
\noindent Now, since \( 1 < s < t - 1 \) and \( t^* > 0 \), it follows that
\begin{equation}\label{times_geq_1}
s + t^* > 1, \hspace{2.5em} t - s > 1.
\end{equation}
\noindent Furthermore, for all \( (x,y) \in R_s \), it holds
\begin{equation}\label{inequalities_space_time_lattice}
|y - y_0| \leq \sqrt{\frac{s+t^*}{N}} \hspace{1.0em} \text{ and } \hspace{1.0em} |x - y| \leq \sqrt{\frac{t-s}{N}},
\end{equation}
\noindent so that
\[
s + t^* \geq N \, |y-y_0|^2 \equiv N \left[ d(y-y_0,0) \right]^2 \geq \left[ \rho(y-y_0,0) \right]^2 \geq \rho(y-y_0,0),
\]
\noindent where the second inequality comes from \eqref{equivalence_rho_d__lattice}, while the last one is due to the fact that \( \rho \) has codomain \( \mathbb{N}_0 \). By arguing in the exact same way, we also infer that
\[
t - s \geq N \, |x-y|^2 \equiv N \left[ d(x-y,0) \right]^2 \geq \left[ \rho(x-y,0) \right]^2 \geq \rho(x-y,0).
\]
\noindent Therefore, thanks to \eqref{times_geq_1}, for all \( (x,y) \in R_s \) it holds:
\[
s + t^* \geq \max \left\{ 1, \rho(y-y_0,0) \right\} \hspace{1.0em} \text{ and } \hspace{1.0em} t - s \geq \max \left\{ 1, \rho(x-y,0) \right\}.
\]
\noindent We are then allowed to apply the first bound in the statement of \cref{lemma_inequalities_HK_lattice}, which yields, for a couple of positive constants \( C_5, C_6 \) only depending on \( N \), both
\[
K_{t-s}(x-y) \geq \frac{C_5}{(t-s)^{N/2}} \, e^{- C_6 \, \frac{|x-y|^2}{t-s}} \geq \frac{C_5}{(t-s)^{N/2}} \, e^{- \frac{C_6}{N}} \equiv C_7 \, (t-s)^{- \frac{N}{2}}
\]
\noindent and
\[
K_{s + t^*}(y-y_0) \geq \frac{C_5}{(s + t^*)^{N/2}} \, e^{- C_6 \, \frac{|y-y_0|^2}{s + t^*}} \geq \frac{C_5}{(s + t^*)^{N/2}} \, e^{- \frac{C_6}{N}} \equiv C_7 \, (s + t^*)^{- \frac{N}{2}},
\]
\noindent holding for all \( (x,y) \in R_s \). Here, the last inequalities are due to \eqref{inequalities_space_time_lattice}. Moreover, in the last equalities we have set \( C_7 := C_5 \, e^{- \frac{C_6}{N}} > 0 \).

\medskip
\noindent In conclusion, we obtain, for all \( (x,y) \in R_s \):
\[
\begin{aligned}
K_{t-s}(x-y) \left[ K_{s + t^*}(y - y_0) \right]^p & \geq (C_7)^{p+1} \, (t-s)^{- \frac{N}{2}} \, (s + t^*)^{- \frac{pN}{2}} \\
& \equiv C_8 \, (t-s)^{- \frac{N}{2}} \, (s + t^*)^{- \frac{pN}{2}},
\end{aligned}
\]
\noindent where \( C_8 := (C_7)^{p+1} \equiv (C_7)^{2 + 2/N} > 0 \). Therefore, the following bound holds for all \( s \in (1, t-1) \):
\[
\begin{aligned}
\sum_{x,y \in \mathbb{Z}^N} K_{t-s}(x-y) \left[ K_{s + t^*}(y - y_0) \right]^p & \geq \sum_{(x,y) \in R_s} K_{t-s}(x-y) \left[ K_{s + t^*}(y - y_0) \right]^p \\
& \geq C_8 \, (t-s)^{- \frac{N}{2}} \, (s + t^*)^{- \frac{pN}{2}} \left[ \sum_{(x,y) \in R_s} 1 \right] \\
& \equiv C_8 \, (t-s)^{- \frac{N}{2}} \, (s + t^*)^{- \frac{pN}{2}} \, |R_s| \\
& \geq C_9 \, (t-s)^{- \frac{N}{2}} \, (s + t^*)^{- \frac{pN}{2}} \, \left( s + t^* \right)^{\frac{N}{2}} \left( t - s \right)^{\frac{N}{2}} \\
& \equiv C_9 \, (s + t^*)^{- (p-1) \, \frac{N}{2}} \\
& \equiv \frac{C_9}{s + t^*},
\end{aligned}
\]
\noindent where \( C_9 := C_4 \, C_8 > 0 \). Here, the first inequality exploits the  positivity of the heat kernel, while in the third one the bound \eqref{cardinality_set_R_s_2} is used. The last equality, instead, is due to the assumption on \( p \).

\medskip
\noindent We can now integrate over \( s \in (1, t-1) \), obtaining, for all \( t > 2 \):
\[
\begin{aligned}
\int_{1}^{t-1} \sum_{x,y \in \mathbb{Z}^N} K_{t-s}(x-y) \left[ K_{s + t^*}(y - y_0) \right]^p \, ds & \geq C_9 \, \int_{1}^{t-1} \frac{1}{s + t^*} \, ds \\
& = C_9 \, \log \left( \frac{t + t^* - 1}{t^* + 1} \right),
\end{aligned}
\]
\noindent where the last equality corresponds to an elementary integral computation. In particular, since the integrand function is nonnegative, for all \( t > 2 \) it also holds:
\[
\begin{aligned}
\int_0^t \sum_{x,y \in \mathbb{Z}^N} K_{t-s}(x-y) \left[ K_{s + t^*}(y - y_0) \right]^p \, ds
& \geq \int_{1}^{t-1} \sum_{x,y \in \mathbb{Z}^N} K_{t-s}(x-y) \left[ K_{s + t^*}(y - y_0) \right]^p \, ds \\
& \geq C_9 \, \log \left( \frac{t + t^* - 1}{t^* + 1} \right).
\end{aligned}
\]
\noindent We can finally exploit \eqref{geq_6_lattice} to obtain, for all \( t > 2 \):
\[
\begin{aligned}
\sum_{x \in \mathbb{Z}^N} u^*(x,t) & \geq C_2 \int_0^t \sum_{x,y \in \mathbb{Z}^N} K_{t-s}(x-y) \left[ K_{s + t^*}(y - y_0) \right]^p \, ds \\
& \geq C_{10} \, \log \left( \frac{t + t^* - 1}{t^* + 1} \right),
\end{aligned}
\]
\noindent where \( C_{10} := C_2 \, C_9 > 0 \). In particular, since \( t^* > 0 \) is fixed, the term at the right-hand side diverges as \( t \to +\infty \), hence:
\[
\lim_{t \to +\infty} \sum_{x \in \mathbb{Z}^N} u^*(x,t) = +\infty.
\]
\noindent By exploiting the relation \eqref{relation_star_lattice}, we infer:
\[
\lim_{t \to +\infty} \sum_{x \in \mathbb{Z}^N} u(x,t) = +\infty.
\]
\noindent This means that we can always find a positive time at which the sum under consideration exceeds any fixed positive constant value. In particular, there exists \( \tilde{t} > 0 \) such that
\[
\sum_{x \in \mathbb{Z}^N} u \hspace{-0.2em} \left( x, \tilde{t} \right) > (2 \pi N)^{\frac{N}{2}} > 0.
\]
\noindent We can now set \( \tilde{u}(\cdot, t) = u \left( \cdot, t + \tilde{t} \right) \) for every \( t \geq 0 \); arguing as before, we obtain that \( \tilde{u} \not \equiv 0 \) is a global solution to problem \eqref{problem_HR_G}, with initial datum \( \widetilde{u_0} := \tilde{u}(\cdot, 0) = u(\cdot, \tilde{t}) \not \equiv 0 \) satisfying \eqref{eq:HP_ID_G}. With this setting, the previous inequality reads as follows:
\begin{equation}\label{condition2_lattice}
\sum_{x \in \mathbb{Z}^N} \widetilde{u_0}(x) > (2 \pi N)^{\frac{N}{2}}.
\end{equation}
\noindent The fact that \( \widetilde{u_0} \) satisfies \eqref{eq:HP_ID_G} allows us to proceed as in step \textbf{(I)} of the proof of case \emph{a)}; specifically, by applying the Monotone Convergence Theorem in the context of the measure space \( (\mathbb{Z}^N,\mathcal{P}(\mathbb{Z}^N),\mu) \), we obtain:
\begin{equation}\label{condition3_lattice}
\lim_{k \to 0^+} \sum_{x \in \mathbb{Z}^N} e^{-k |x|^2} \widetilde{u_0}(x) = \sum_{x \in \mathbb{Z}^N} \widetilde{u_0}(x) > (2 \pi N)^{\frac{N}{2}},
\end{equation}
\noindent where the last inequality exploits \eqref{condition2_lattice}.

\medskip
\noindent We now recall, from step \textbf{(II)} of the proof of case \emph{a)}, that the term \( C_{k,N,p} \), defined in \eqref{definition_C_k_N_p} as
\[
C_{k,N,p} = \left[ \theta \left( \frac{k}{\pi} \right) \right]^N (2kN)^{\frac{1}{p-1}},
\]
\noindent has the following asymptotic behavior for small values of \( k \):
\[
C_{k,N,p} \hspace{0.2em} \sim \hspace{0.2em} \left[ (2N)^{\frac{1}{p-1}} \hspace{0.2em} \pi^{\frac{N}{2}} \right] k^{\frac{1}{p-1} - \frac{N}{2}} \hspace{2.0em} \text{as } k \to 0^+.
\]
\noindent This result was shown independently of the value of \( p > 1 \). In particular, within the current framework, where \( p = 1 + 2/N \), we get:
\[
C_{k,N,1 + \frac{2}{N}} \hspace{0.2em} \sim \hspace{0.2em} \left[ (2N)^{\frac{N}{2}} \hspace{0.2em} \pi^{\frac{N}{2}} \right] k^{0} \equiv (2 \pi N)^{\frac{N}{2}}
\hspace{2.0em} \text{as } k \to 0^+,
\]
\noindent so that
\[
\lim_{k \to 0^+} C_{k,N,1 + \frac{2}{N}} = (2 \pi N)^{\frac{N}{2}}.
\]
\noindent Combining this result with \eqref{condition3_lattice} yields:
\[
\lim_{k \to 0^+} \sum_{x \in \mathbb{Z}^N} e^{-k |x|^2} \widetilde{u_0}(x) > (2 \pi N)^{\frac{N}{2}} = \lim_{k \to 0^+} C_{k,N,1 + \frac{2}{N}}.
\]
\noindent This fact allows to conclude that there must exist a small \( \bar{k} > 0 \) such that
\[
\sum_{x \in \mathbb{Z}^N} e^{- \bar{k} \, |x|^2} \widetilde{u_0}(x) > C_{\bar{k},N,1 + \frac{2}{N}} = \left[ \theta \left( \frac{\bar{k}}{\pi} \right) \right]^N (2 \bar{k} N)^{\frac{N}{2}},
\]
\noindent where the last equality exploits \eqref{definition_C_k_N_p}.

\medskip
\noindent The obtained inequality corresponds exactly to \eqref{condition_blow_up_G_lemma_1}, in the specific case \( p = 1 + 2/N \) and with \( k = \bar{k} \). Therefore, as already discussed, \eqref{condition_blow_up_G_lemma_1_new} holds with \( k = \bar{k} \), hence we can apply \cref{lemma_blow_up_lattice_1}, obtaining that \( \tilde{u} \) blows up in finite time. This is clearly a contradiction, since we had already concluded that \( \tilde{u} \) is a global solution of problem \eqref{problem_HR_G}. Therefore, the original assumption that \( u \) is a global solution of problem \eqref{problem_HR_G} has been revealed false, meaning that \( u \) is nonglobal. This yields the thesis. \hfill \( \square \)

\medskip

\subsection{Proof of Theorem \ref{theorem_global_existence_lattice}}

\noindent We begin by establishing two auxiliary results in the general framework of weighted graphs; they will play a key role in the proof of \cref{theorem_global_existence_lattice} and will also be employed in \cref{subsection_proofs_2} below.

\begin{lemma}\label{lemma_mild_solution_b}
Let \( (X,\omega,\mu) \) be a locally finite weighted graph, and let the heat semigroup \( \left\{ P_t \right\}_{t \geq 0} \) be strongly continuous, positivity preserving and \( \ell^\infty \)-contractive. Consider $u_0\in\ell^\infty(X)$ and \( f : [0,+\infty) \to (0,+\infty) \), with \( f \in C^1((0,+\infty)) \hspace{0.1em} \cap \hspace{0.1em} C^0([0,+\infty)) \). In addition, set
\[
b(t) := \frac{f'(t)}{f(t)} \qquad \quad \text{for all } \hspace{0.1em} t \in (0,+\infty),
\]
together with
\[
z(\cdot,t) := P_t \hspace{0.05em} u_0 \qquad \quad \text{for all } \hspace{0.1em} t \in [0,+\infty)
\]
and
\[
u(\cdot,t) := f(t) \, z(\cdot,t) = f(t) \, P_t \hspace{0.05em} u_0 \qquad \quad \text{for all } \hspace{0.1em} t \in [0,+\infty).
\]
Then the following identities hold pointwise in \( X \):
\begin{equation}\label{eq:direct-mild}
\begin{aligned}
u(\cdot,0) &= f(0) \, u_0, \\
u(\cdot,t) &= P_t \left( u(\cdot,0) \right) + \int_0^t P_{t-s} \, \left( b(s) \, u(\cdot,s) \right) \, ds \qquad \quad \text{for all } \hspace{0.1em} t \in (0,+\infty).
\end{aligned}
\end{equation}
\end{lemma}

\noindent \emph{Proof.} First, by exploiting the definition of the function $u$ and the fact that $P_0$ coincides
with the identity operator, we immediately infer:
\[
u(\cdot,0)=f(0)P_0u_0=f(0)u_0,
\]
which is the first identity in \eqref{eq:direct-mild}.

\medskip
\noindent
Let us now prove the validity of the second identity in \eqref{eq:direct-mild}.
\textit{Throughout this proof, all time integrals are intended pointwise in $X$, namely after
fixing a vertex $x\in X$.}
Let us fix an arbitrary time instant $t>0$ and an arbitrary vertex $x\in X$, and define the
(real-valued) function $F_x:[0,t]\to\mathbb{R}$ as
\[
F_x(s):= f(s)\,\big(P_{t-s}z(\cdot,s)\big)(x)
\qquad \text{for all } s\in[0,t],
\]
where $z(\cdot,s)=P_su_0$ by assumption.

\medskip
\noindent
Next, define $G_x:[0,t]\to\mathbb{R}$ by
\[
G_x(s):=\big(P_{t-s}z(\cdot,s)\big)(x)=\big(P_{t-s}P_su_0\big)(x)
\qquad \text{for all } s\in[0,t].
\]
By the semigroup identity \eqref{identity_P_t_P_s}, we have
\[
G_x(s)=\big(P_tu_0\big)(x)\qquad \text{for all } s\in[0,t],
\]
hence $G_x$ is constant. Therefore,
\[
F_x(s)= f(s)\,\big(P_tu_0\big)(x)\qquad \text{for all } s\in[0,t].
\]
Since $f\in C^1((0,+\infty))\cap C^0([0,+\infty))$, it follows that $F_x\in C^1((0,t))$ and
\[
F_x'(s)= f'(s)\,\big(P_tu_0\big)(x)
= f'(s)\,\big(P_{t-s}z(\cdot,s)\big)(x)
\qquad \text{for all } s\in(0,t).
\]
By integrating the previous identity over $(0,t)$, we obtain
\[
F_x(t)-F_x(0)=\int_0^t f'(s)\,\big(P_{t-s}z(\cdot,s)\big)(x)\,ds.
\]
Now, using the definitions of $F_x$, $z$, $b$, and $u$, we have
\[
F_x(t)= f(t)\,\big(P_0z(\cdot,t)\big)(x)= f(t)z(x,t)=u(x,t),
\]
\[
F_x(0)= f(0)\,\big(P_tz(\cdot,0)\big)(x)= f(0)\,\big(P_tP_0u_0\big)(x)
= \big(P_t(f(0)u_0)\big)(x)=\big(P_t(u(\cdot,0))\big)(x),
\]
and, for all $s\in(0,t)$,
\[
f'(s)\,z(\cdot,s)=\frac{f'(s)}{f(s)}\,f(s)z(\cdot,s)= b(s)\,u(\cdot,s).
\]
Therefore,
\[
u(x,t)-\big(P_t(u(\cdot,0))\big)(x)
=\int_0^t \big(P_{t-s}(b(s)u(\cdot,s))\big)(x)\,ds.
\]
Since $x\in X$ and $t>0$ are arbitrary, the previous identity yields the second relation in
\eqref{eq:direct-mild}, namely
\[
u(\cdot,t)=P_t(u(\cdot,0))+\int_0^t P_{t-s}(b(s)u(\cdot,s))\,ds
\qquad \text{for all } t\in(0,+\infty),
\]
holding pointwise in $X$. This completes the proof. \hfill \( \square \)

\medskip
\begin{lemma}\label{lemma_mild_supersol_1}
Under the same notation and assumptions of \emph{\cref{lemma_mild_solution_b}}, suppose in addition that \( u_0 \geq 0 \) in \( X \) and that the function \( f \) satisfies the condition $f(0)\ge 1$, together with
\begin{equation}\label{eq:keyineq}
b(t) = \frac{f'(t)}{f(t)} \ge \left[ \|u(\cdot,t)\|_\infty \right]^{p-1} \qquad \quad \text{for all } \hspace{0.1em} t \in (0,+\infty).
\end{equation}
Then the following inequality holds pointwise in \( X \) for \( u \):
\[
u(\cdot,t) \ge P_t \, u_0 + \int_0^t \, P_{t-s} \left[ u(\cdot,s)^p \right] \, ds \qquad \quad \text{for all } \hspace{0.1em} t \in (0,+\infty).
\]
\end{lemma}

\noindent \emph{Proof.} We first notice that, since for all \( t > 0 \) the function \( u \) is given by
\[
u(\cdot,t) = f(t) \, P_t \, u_0,
\]
then the positivity of \( P_t \), together with the nonnegativity of both \( f \) and \( u_0 \), implies that \( u(\cdot,t) \geq 0 \) in \( X \), for all \( t > 0 \).

\medskip
\noindent Now, by the result \eqref{eq:direct-mild} in \cref{lemma_mild_solution_b}, we already know that \( u \) satisfies the following identity, pointwise in \( X \):
\[
u(\cdot,t) = P_t \left( f(0) \, u_0 \right) + \int_0^t P_{t-s} \, \left( b(s) \, u(\cdot,s) \right) \, ds \qquad \quad \text{for all } \hspace{0.1em} t \in (0,+\infty).
\]
By exploiting the facts that $f(0)\ge1$ and that $P_t$ is positivity preserving, together with the nonnegativity of \( u_0 \), we easily obtain the validity of the inequality
\begin{equation}\label{inequality_f_0_u_0}
P_t \left( f(0) \, u_0 \right) \ge P_t \, u_0 \qquad \quad \text{for all } \hspace{0.1em} t \in (0,+\infty),
\end{equation}
holding pointwise in \( X \).

\medskip
\noindent Let us now fix \( t \in (0,+\infty) \). Then, for all \( s \in (0,t) \), \( u(\cdot,s) \) is nonnegative, hence in \( X \) it holds:
\[
b(s) \, u(\cdot,s) \ge \left[ \|u(\cdot,s)\|_\infty \right]^{p-1} \, u(\cdot,s) \ge \left[ u(\cdot,s) \right]^p,
\]
where the first inequality exploits \eqref{eq:keyineq}. We can now apply the positivity preserving operator $P_{t-s}$ and integrate over \( s \in (0,t) \), obtaining:
\[
\int_0^t P_{t-s} \left( b(s) \, u(\cdot,s) \right) \,ds \ge \int_0^t \, P_{t-s} \left[u(\cdot,s) \right]^p \, ds,
\]
which holds pointwise in \( X \), for all \( t > 0 \). By combining this inequality with \eqref{inequality_f_0_u_0} and \eqref{eq:direct-mild}, we infer the following estimate in \( X \):
\[
\begin{aligned}
u(\cdot,t) & = P_t \left( f(0) \, u_0 \right) + \int_0^t P_{t-s} \, \left( b(s) \, u(\cdot,s) \right) \, ds \\
& \geq P_t \, u_0 + \int_0^t \, P_{t-s} \left[u(\cdot,s) \right]^p \, ds,
\end{aligned}
\]
holding for all \( t > 0 \). We have then shown the validity of the desired inequality, and the proof is complete. \hfill \( \square \)

\begin{proposition}\label{proposition_global_super_sol_lattice}
If \hspace{0.1em} \( p > 1 + \frac{2}{N} \), then equation \eqref{e30} admits a nontrivial global classical supersolution \( \overline{u} \). In addition, \( \overline{u} \) is a global mild supersolution of problem \eqref{problem_HR_G} with $X=\mathbb{Z}^N$, with initial datum \( \overline{u}(\cdot,0) \).
\end{proposition}

\noindent \emph{Proof.} We work under the hypothesis \( p > 1 + 2/N \), and we consider a function \( v_0 \in \ell^1(\mathbb{Z}^N) \), with \( v_0 > 0 \) pointwise in \( \mathbb{Z}^N \). Since \( \mu \equiv 2N \) on \( \mathbb{Z}^N \), it is trivial to verify that \( v_0 \in \ell^1(\mathbb{Z}^N,\mu) \). Hence we also have \( v_0 \in \ell^\infty(\mathbb{Z}^N) \). In particular, from \cref{theorem_well_posedness_H_G_lattice} we know that the function \( v : \mathbb{Z}^N \times [0,+\infty) \to \mathbb{R} \), defined as
\[
v(x,t) := (K_t * v_0)(x) = 2N \sum_{y \in \mathbb{Z}^N} K_t(x-y) \, v_0(y) \hspace{2.5em} \text{for all } \hspace{0.1em} (x,t) \in \mathbb{Z}^N \times [0,+\infty),
\]
\noindent is the unique bounded solution of the problem \eqref{problem_H_G} posed on the lattice, associated with the initial datum \( v_0 \). Here, \( K_t \) denotes the heat kernel of the heat equation posed on the lattice; its analytical expression was given in \eqref{definition_K_t_HK_lattice}.

\medskip
\noindent We now recall, from \cref{lemma_inequalities_HK_lattice}, the following upper bound satisfied by the heat kernel:
\[
K_t(x) \leq \frac{c}{t^{N/2}} \hspace{2.5em} \text{for all } (x,t) \in \mathbb{Z}^N \times (0,+\infty),
\]
\noindent where \( c \) is a positive constant depending only on the dimension \( N \).

\medskip
\noindent Thanks to this inequality, we obtain, for all \( (x,t) \in \mathbb{Z}^N \times (0,+\infty) \):
\[
0 < 2N \sum_{y \in \mathbb{Z}^N} K_t(x-y) \, v_0(y) \leq 2N \, \frac{c}{t^{N/2}} \sum_{y \in \mathbb{Z}^N} v_0(y) \equiv \frac{c}{t^{N/2}} \, \| v_0 \|_1 < +\infty,
\]
\noindent where the first inequality exploits the strict positivity of \( v_0 \) and \( N \), together with property b) in \cref{proposition_properties_HK_lattice}, while the last one holds since \( v_0 \in \ell^1(\mathbb{Z}^N,\mu) \). In other words, we have:
\begin{equation}\label{inequalities_v_existence_lattice}
0 < v(x,t) \leq \frac{c}{t^{N/2}} \, \| v_0 \|_1 \hspace{2.5em} \text{for all } (x,t) \in \mathbb{Z}^N \times (0,+\infty).
\end{equation}
\noindent Now, we take a scalar function \( h \), defined at least in a subset of \( [0,+\infty) \), and a positive time instant \( t_0 > 0 \). Moreover, we set
\begin{equation}\label{definition_super_solution_existence_lattice}
\overline{u}(x,t) := h(t) \, v(x,t+t_0).
\end{equation}
\noindent Our aim is to find the conditions to be imposed on \( t_0 \) and \( h \) which ensure that this function is a global supersolution of the equation \( u_t - \Delta u = u^p \) on the lattice.

\medskip
\noindent In particular, we claim that, under the following conditions:
\begin{equation}\label{conditions_h_global_super_sol_lattice}
\begin{gathered}
h : [0,+\infty) \to (0,C) \hspace{2.5em} \text{for some constant } C > 0, \\[7pt]
h \in C^1((0,+\infty)) \cap C^0([0,+\infty)), \\[7pt]
\frac{h'(t)}{\left[ h(t) \right]^p} = \left( \frac{c \, \|v_0\|_1}{(t+t_0)^{N/2}} \right)^{p-1} \hspace{2.5em} \text{for all } t > 0,
\end{gathered}
\end{equation}
\noindent the function \( \overline{u} \), given by \eqref{definition_super_solution_existence_lattice}, is a nontrivial global supersolution of the equation \( u_t - \Delta u = u^p \).

\medskip
\noindent We start by noticing that, if the first condition in \eqref{conditions_h_global_super_sol_lattice} holds, then the relation \eqref{definition_super_solution_existence_lattice} is valid for all \(  (x,t) \in \mathbb{Z}^N \times [0,+\infty) \), so that \( \overline{u} : \mathbb{Z}^N \times [0,+\infty) \to \mathbb{R} \) and
\begin{equation}\label{positivity_super_solution_existence_lattice}
\overline{u}(x,t) = h(t) \, v(x,t+t_0) > 0 \hspace{2.5em} \text{for all } (x,t) \in \mathbb{Z}^N \times [0,+\infty),
\end{equation}
\noindent where we have also exploited the first inequality in \eqref{inequalities_v_existence_lattice}. In particular, it follows from \eqref{positivity_super_solution_existence_lattice} that \( u \) is nontrivial.

\medskip
\noindent Now, we verify the validity of requirements \eqref{reg_sol_1_G} and \eqref{reg_sol_2_G} in \cref{definition_sub_super_sols_G}, with \( T = +\infty \). First, after fixing an arbitrary vertex \( x \in \mathbb{Z}^N \), we have:
\[
\overline{u}(x,\cdot) = h(\cdot) \, v(x, t_0 + (\cdot)) \hspace{0.3em} \in \hspace{0.3em} C^1((0,+\infty)) \cap C^0([0,+\infty)),
\]
\noindent which is due to both \eqref{reg_sol__H_G} and the second condition in \eqref{conditions_h_global_super_sol_lattice}. Furthermore, for any fixed \( T' > 0 \) it holds:
\[
\begin{aligned}
\| u \|_{L^\infty(\mathbb{Z}^N \times [0,T'])} & = \sup_{(x,t) \in \mathbb{Z}^N \times [0,T']} h(t) \, v(x,t+t_0) \\
& = \sup_{(x,\tau) \in \mathbb{Z}^N \times [t_0,t_0 + T']} h(t) \, v(x,\tau) \\
& \leq C \, \sup_{(x,\tau) \in \mathbb{Z}^N \times [t_0,t_0 + T']} v(x,\tau) \\
& \leq C \, \| v \|_{L^\infty(\mathbb{Z}^N \times (0,+\infty))} \\
& < +\infty.
\end{aligned}
\]
\noindent Here, we have first exploited \eqref{definition_super_solution_existence_lattice}, then the first condition in \eqref{conditions_h_global_super_sol_lattice}, and finally the fact that \( v \) is a bounded solution to problem \eqref{problem_H_G}, so that \( v  \in L^\infty(\mathbb{Z}^N \times (0,+\infty)) \). In conclusion, both \eqref{reg_sol_1_G} and \eqref{reg_sol_2_G} hold for \( \overline{u} \), with \( T = +\infty \).

\medskip
\noindent We are only left to verify that \eqref{conditions_h_global_super_sol_lattice} implies that
\begin{equation}\label{inequality_overline_u_existence_lattice}
\overline{u}_t - \Delta \overline{u} \geq \overline{u}^p \hspace{2.5em} \text{in} \hspace{0.5em} \mathbb{Z}^N \times (0,+\infty).
\end{equation}
\noindent In order to do so, we exploit \eqref{definition_super_solution_existence_lattice} in order to compute formally the partial time derivative and the discrete Laplacian of the function \( \overline{u} \). More specifically, we get both
\[
\overline{u}_t(x,t) = \frac{\partial}{\partial t} \left[ h(t) \, v(x,t+t_0) \right] = h'(t) \, v(x,t+t_0) + h(t) \, v_t(x,t+t_0)
\]
\noindent and
\[
\begin{aligned}
\Delta \overline{u} (x,t) & = \frac{1}{2N} \sum_{y \in \mathbb{Z}^N} \left[ u(y,t) - u(x,t) \right] \omega_0(x,y) \\
& = \frac{1}{2N} \sum_{y \in \mathbb{Z}^N} \left[ h(t) \, v(y,t+t_0) - h(t) \, v(x,t+t_0) \right] \omega_0(x,y) \\
& = \frac{h(t)}{2N} \sum_{y \in \mathbb{Z}^N} \left[ v(y,t+t_0) - v(x,t+t_0) \right] \omega_0(x,y) \\
& = h(t) \, \Delta v (x,t+t_0).
\end{aligned}
\]
\noindent Therefore, we obtain:
\[
\begin{aligned}
\overline{u}_t(x,t) - \Delta \overline{u} (x,t) & = h'(t) \, v(x,t+t_0) + h(t) \, v_t(x,t+t_0) - h(t) \, \Delta v (x,t+t_0) \\
& = h'(t) \, v(x,t+t_0) + h(t) \left[ v_t(x,t+t_0) - \Delta v (x,t+t_0) \right] \\
& = h'(t) \, v(x,t+t_0),
\end{aligned}
\]
\noindent where the last equality is due to the fact that \( v \) solves problem \eqref{problem_H_G} on the lattice, hence in particular it holds
\[
v_t(x,t) - \Delta v (x,t) = 0 \hspace{2.5em} \text{for all } (x,t) \in \mathbb{Z}^N \times (0,+\infty).
\]
\noindent The inequality \eqref{inequality_overline_u_existence_lattice} is then equivalent to the following one:
\begin{equation}\label{equivalence_super_sol_existence_lattice}
h'(t) \, v(x,t+t_0) \geq \left[ h(t) \, v(x,t+t_0) \right]^p \hspace{2.5em} \text{for all } (x,t) \in \mathbb{Z}^N \times (0,+\infty).
\end{equation}
\noindent Now, the third requirement in \eqref{conditions_h_global_super_sol_lattice} yields, for all \( (x,t) \in \mathbb{Z}^N \times (0,+\infty) \):
\[
\frac{h'(t)}{\left[ h(t) \right]^p} = \left( \frac{c \, \|v_0\|_1}{(t+t_0)^{N/2}} \right)^{p-1} \geq \left[ v(x,t+t_0) \right]^{p-1},
\]
\noindent where the last inequality holds thanks to the upper bound in \eqref{inequalities_v_existence_lattice} and to the fact that \( p > 1 + 2/N > 1 \). Moreover, the first condition in \eqref{conditions_h_global_super_sol_lattice} implies that \( h > 0 \) in \( [0,+\infty) \), and from \eqref{inequalities_v_existence_lattice} we also know that \( v > 0 \) in \( \mathbb{Z}^N \times (0,+\infty) \). Therefore, we can multiply both sides of the obtained inequality by \( [h(t)]^p \, v(x,t+t_0) > 0 \), obtaining:
\[
h'(t) \, v(x,t+t_0) \geq \left[ h(t) \, v(x,t+t_0) \right]^p \hspace{2.5em} \text{for all } (x,t) \in \mathbb{Z}^N \times (0,+\infty),
\]
\noindent which exactly corresponds to \eqref{equivalence_super_sol_existence_lattice}. As discussed above, \eqref{inequality_overline_u_existence_lattice} is then obtained.

\medskip
\noindent In conclusion, the claim has been proved: under \eqref{conditions_h_global_super_sol_lattice}, the function \( \overline{u} \) defined in \eqref{definition_super_solution_existence_lattice} verifies all the requirements of a nontrivial global supersolution of the equation \( u_t - \Delta u = u^p \). In other words, the thesis follows if we exhibit a function \( h \) satisfying all the conditions in \eqref{conditions_h_global_super_sol_lattice}, for some \( t_0 > 0 \).

\medskip
\noindent We shall do so by coupling the third requirement in \eqref{conditions_h_global_super_sol_lattice} with a positive initial condition for \( h \); for the sake of simplicity, we choose \( h(0) = 1 \). Therefore, we look for a function \( h \) solving the following Cauchy problem:
\begin{equation}\label{ODE_h_existence_lattice}
\begin{cases}
\frac{h'(t)}{\left[ h(t) \right]^p} = \left( \frac{c \, \|v_0\|_1}{(t+t_0)^{N/2}} \right)^{p-1} \hspace{1.5em} \text{for all } t > 0 \\
h(0) = 1.
\end{cases}
\end{equation}
\noindent We proceed by finding the analytical expression of the solution to this problem; we will then investigate the values of \( t_0 \) for which such solution exists globally in time. More specifically, after fixing \( t > 0 \), we integrate both sides in \( (0,t) \) and exploit the initial condition, obtaining:
\[
\int_{0}^{t} \frac{h'(s)}{\left[ h(s) \right]^p} \, ds = \int_{1}^{h(t)} \frac{1}{h^p} \, dh = \left( c \, \| v_0 \|_{1} \right)^{p-1} \int_{0}^{t} (s+t_0)^{- \frac{N}{2} (p-1)} \, ds,
\]
\noindent where the first equality simply corresponds to a change of variable. Performing direct computations on the last equality yields the following:
\[
\frac{\left[ h(t) \right]^{1-p} - 1}{1-p} = \left( c \, \| v_0 \|_{1} \right)^{p-1} \, \frac{(t+t_0)^{1 - \frac{N}{2} (p-1)} - t_0^{1 - \frac{N}{2} (p-1)}}{1 - \frac{N}{2} (p-1)}.
\]
\noindent Rearranging the terms, we get
\[
\begin{aligned}
\left[ h(t) \right]^{1-p} & = 1 - \left[ \frac{(p-1) \left( c \, \| v_0 \|_{1} \right)^{p-1}}{\frac{N}{2}(p-1)-1} \, t_0^{1 - \frac{N}{2}(p-1)} \right]
\left[ 1 - \left( 1 + \frac{t}{t_0} \right)^{1 - \frac{N}{2}(p-1)} \right] \\
& \equiv 1 - D \left[ 1 - \left( 1 + \frac{t}{t_0} \right)^{1 - \frac{N}{2}(p-1)} \right],
\end{aligned}
\]
\noindent where we have set
\[
D := \frac{(p-1) \left( c \, \| v_0 \|_{1} \right)^{p-1}}{\frac{N}{2}(p-1)-1} \, t_0^{1 - \frac{N}{2}(p-1)}.
\]
\noindent Therefore, the analytical expression of the solution is
\begin{equation}\label{analytical_expression_h_existence_lattice}
h(t) = \left\{ 1 - D \left[ 1 - \left( 1 + \frac{t}{t_0} \right)^{1 - \frac{N}{2}(p-1)} \right] \right\}^{- \frac{1}{p-1}}.
\end{equation}
\noindent Now, notice that the assumptions on \( v_0 \) and \( p \) ensure that \( D \in (0,+\infty) \). In particular, the condition \( D \in (0,1) \) is equivalent to the following bound for \( t_0 \):
\begin{equation}\label{bound_t_0_existence_lattice}
t_0 > \left[ \frac{(p-1) \left( c \|v_0\|_{1} \right)^{p-1}}{\frac{N}{2}(p-1)-1} \right]^{\frac{2}{N(p-1)-2}} > 0,
\end{equation}
\noindent where in the last inequality we have used the fact that \( \frac{N}{2}(p-1) - 1 > 0 \), since by assumption it holds \( p > 1 + 2/N \). In conclusion, if \( t_0 \) satisfies \eqref{bound_t_0_existence_lattice} then \( D \in (0,1) \), and therefore the function \( h \) in \eqref{analytical_expression_h_existence_lattice} is well-defined for all nonnegative times. More specifically, if \( D \in (0,1) \), then it is not difficult to verify that \( h(t) > 0 \) for all \( t > 0 \); furthermore, since \( p > 1 + 2/N \), \( h \) is monotone increasing in \( [0,+\infty) \), and it exhibits an asymptote as \( t \to +\infty \), namely:
\[
\lim_{t \to +\infty} h(t) = (1 - D)^{- \frac{1}{p-1}} \hspace{0.3em} \in \hspace{0.3em} (1,+\infty).
\]
\noindent Therefore, we have:
\[
h : [0,+\infty) \to [1, (1 - D)^{- \frac{1}{p-1}}),
\]
\noindent where we have used the fact that \( h(0) = 1 \). The first condition in \eqref{conditions_h_global_super_sol_lattice} is then satisfied with \( C := (1 - D)^{- \frac{1}{p-1}} \). In addition, \( h \) is a classical global solution to the Cauchy problem \eqref{ODE_h_existence_lattice}, hence the second and the third requirements in \eqref{conditions_h_global_super_sol_lattice} are trivially verified.

\medskip
\noindent In conclusion, if we choose \( t_0 \) as in \eqref{bound_t_0_existence_lattice}, then the function \( h \) defined in \eqref{analytical_expression_h_existence_lattice} verifies \eqref{conditions_h_global_super_sol_lattice}. As already discussed, this finishes the proof of the first part of the statement.

\medskip
\medskip
\noindent Now, we explicitly remark that the function \( \overline{u}(\cdot,0) \in C(\mathbb{Z}^N) \) satisfies the assumptions in \eqref{eq:HP_ID_G}. Indeed, from \eqref{positivity_super_solution_existence_lattice} it follows that \( \overline{u}(\cdot,0) > 0 \) in \( \mathbb{Z}^N \). Furthermore, the fact that \( \overline{u} \in L^\infty(\mathbb{Z}^N \times [0,T']) \), for any \( T' > 0 \), trivially implies that \( \overline{u}(\cdot,0) \in \ell^\infty(\mathbb{Z}^N) \).

\medskip
\noindent We can then consider problem \eqref{problem_HR_G} with initial datum \( \overline{u}(\cdot,0) \). In order to conclude the proof of the last part of the statement, our aim is to apply \cref{lemma_mild_supersol_1}; we then proceed by verifying that all the hypotheses of such result are verified, within the present framework.

\medskip
\noindent We first notice that from \eqref{definition_super_solution_existence_lattice} it follows that
\begin{equation}\label{overline_u_0_general_graphs_1_lattice}
\overline{u}(\cdot,0) = h(0) \, v(\cdot,t_0) = v(\cdot,t_0) = P_{t_0} \, v_0,
\end{equation}
where we have used that \( h(0) = 1 \) and that \( v(\cdot,\tau) = P_{\tau} \, v_0 \) for all \( \tau \geq 0 \). Let us now define the function \( z : \mathbb{Z}^N \times [0,+\infty) \to (0,+\infty) \) as
\begin{equation}\label{definition_z_general_graphs_lattice}
z(x,t) := v(x,t+t_0) \qquad \quad \text{for all } (x,t) \in \mathbb{Z}^N \times [0,+\infty).
\end{equation}
The fact that \( z \) is strictly positive follows directly from \eqref{inequalities_v_existence_lattice}. Now, by combining \eqref{definition_z_general_graphs_lattice} with \eqref{identity_P_t_P_s} and \eqref{overline_u_0_general_graphs_1_lattice}, for all \( t \geq 0 \) we obtain:
\[
z(\cdot,t) = v(\cdot,t+t_0) = P_{t+t_0} \, v_0 = P_t \left( P_{t_0} \, v_0 \right) = P_t \, \left( \overline{u}(\cdot,0) \right).
\]
By exploiting \eqref{definition_super_solution_existence_lattice} we then have, for all \( t \geq 0 \):
\begin{equation}\label{condition_for_application_result_1_lattice}
\overline{u}(\cdot,t) = h(t) \, z(\cdot,t) \qquad \text{and} \qquad z(\cdot,t) = P_t \, \left( \overline{u}(\cdot,0) \right).
\end{equation}
Now, from \eqref{inequalities_v_existence_lattice} we easily infer:
\[
0 < \| v(\cdot,t) \|_\infty \leq \frac{c}{t^{N/2}} \, \| v_0 \|_1 \hspace{2.5em} \text{for all } t \in (0,+\infty),
\]
so that, by exploiting \eqref{definition_z_general_graphs_lattice}, we obtain, for all \( t > 0 \):
\[
\left[ \| z(\cdot,t) \|_\infty \right]^{p-1} = \left[ \| v(\cdot,t+t_0) \|_\infty \right]^{p-1} \leq \left( \frac{c \, \|v_0\|_1}{(t+t_0)^{N/2}} \right)^{p-1}
= \frac{h'(t)}{\left[ h(t) \right]^p},
\]
where the last equality exploits \eqref{ODE_h_existence_lattice}. We can now multiply both sides of the obtained estimate by the positive value \( \left[ h(t) \right]^{p-1} \), inferring that, for all \( t > 0 \), it holds:
\begin{equation}\label{condition_for_application_result_2_lattice}
\frac{h'(t)}{h(t)} \geq \left[ h(t) \right]^{p-1} \left[ \| z(\cdot,t) \|_\infty \right]^{p-1} \equiv \left[ h(t) \, \| z(\cdot,t) \|_\infty \right]^{p-1}
= \left[ \| \overline{u}(\cdot,t) \|_\infty \right]^{p-1},
\end{equation}
where in the last equality we have made use of \eqref{condition_for_application_result_1_lattice}.

\medskip
\noindent Now, we know that the integer lattice is connected and locally finite; furthermore, the weighted degree is bounded, since it satisfies \( \operatorname{Deg}(x) = 1 \) for every \( x \in \mathbb{Z}^N \). Since the integer lattice is both connected and locally finite, \cref{remark_properties_SG_connected_locally_finite} ensures that the corresponding heat semigroup \( \{P_t\}_{t\ge 0} \) is both positivity preserving and \( \ell^\infty \)-contractive on the Banach space \( \ell^\infty(\mathbb{Z}^N) \). In addition, it is uniformly continuous, hence strongly continuous, on \( \ell^\infty(\mathbb{Z}^N) \).

\medskip
\noindent Finally, as already shown, we have \( h : [0,+\infty) \to (0,+\infty) \), with \( h \in C^1((0,+\infty)) \hspace{0.1em} \cap \hspace{0.1em} C^0([0,+\infty)) \) and \( h(0) = 1 \). These facts, together with \eqref{condition_for_application_result_1_lattice} and \eqref{condition_for_application_result_2_lattice}, allow us to apply \cref{lemma_mild_supersol_1}, which yields that \( \overline{u} \) satisfies the following inequality, pointwise in \( \mathbb{Z}^N \):
\begin{equation}\label{inequality_overline_u_1_lattice}
\overline{u}(\cdot,t) \ge P_t \left( \overline{u}(\cdot,0) \right) + \int_0^t \, P_{t-s} \left[ \overline{u}(\cdot,s)^p \right] \, ds \qquad \quad \text{for all } t \in (0,+\infty).
\end{equation}
Now, from before, we know that \( \overline{u} \in L^\infty(\mathbb{Z}^N \times [0,T']) \) for all \( T' > 0 \). By arguing exactly as in the last part of the proof of \cref{proposition_Duhamel_formula_HR_G}, we conclude that \( \overline{u} \in L^\infty([0,T'];\ell^\infty(\mathbb{Z}^N)) \), for all \( T' > 0 \). This fact, together with \eqref{inequality_overline_u_1_lattice}, allows us to conclude that, according to \cref{definition_mild_sol_sub_super_sol}, the function \( \overline{u} \) is a global mild supersolution to problem \eqref{problem_HR_G}, with initial datum \( \overline{u}(\cdot,0) \). This completes the proof. \hfill \( \square \)

\bigskip
\medskip

\noindent \emph{Proof of \emph{\cref{theorem_global_existence_lattice}}}. Working under the hypothesis \( p > 1 + 2/N \), let \( \underline{u} \equiv 0 \) in \( \mathbb{Z}^N \times [0,+\infty) \), and consider the function \( \overline{u} \) constructed in the proof of \cref{proposition_global_super_sol_lattice}. In particular, as shown there, \( \overline{u}(\cdot,0) \) is strictly positive in \( \mathbb{Z}^N \) and satisfies condition \eqref{eq:HP_ID_G}, so that it holds \( \overline{u}(\cdot,0) \in \ell^\infty(\mathbb{Z}^N) \). Therefore, by making use of \eqref{bound_ID_global_existence_lattice} and by exploiting the assumption \( u_0 \in C(\mathbb{Z}^N) \), we infer that also \( u_0 \in \ell^\infty(\mathbb{Z}^N) \). Now, thanks to the lower bound in \eqref{bound_ID_global_existence_lattice}, we conclude that \( u_0 \) respects the requirements in \eqref{eq:HP_ID_G}, and therefore is a valid initial datum for problem \eqref{problem_HR_G}.

\medskip
\noindent As already explained in the last part of the proof of \cref{proposition_global_super_sol_lattice}, within the framework of the integer lattice the heat semigroup \( \left\{ P_t \right\}_{t \geq 0} \) is positivity preserving on \( \ell^\infty(\mathbb{Z}^N) \), so that for all \( t > 0 \) it holds \( P_t \, u_0 \geq 0 \) in \( \mathbb{Z}^N \). This fact can be combined with \eqref{bound_ID_global_existence_lattice}, in order to conclude that \( \underline{u} \equiv 0 \) is a global mild subsolution to problem \eqref{problem_HR_G}, according to \cref{definition_mild_sol_sub_super_sol}.

\medskip
\noindent Furthermore, by exploiting both \eqref{inequality_overline_u_1_lattice} and \eqref{bound_ID_global_existence_lattice}, we deduce that for all \( t > 0 \) it holds:
\[
\overline{u}(\cdot,t) \ge P_t \left( \overline{u}(\cdot,0) \right) + \int_0^t \, P_{t-s} \left[ \overline{u}(\cdot,s)^p \right] \, ds
\geq P_t \, u_0 + \int_0^t \, P_{t-s} \left[ \overline{u}(\cdot,s)^p \right] \, ds.
\]
Finally, the fact that \( \overline{u}(\cdot,0) \geq u_0 \) in \( \mathbb{Z}^N \) allows us to conclude that \( \overline{u} \) is a global mild supersolution of problem \eqref{problem_HR_G}, according to \cref{definition_mild_sol_sub_super_sol}.

\medskip
\noindent Lastly, since \( \overline{u} \) is strictly positive in \( \mathbb{Z}^N \times [0,+\infty) \), we have:
\[
0 \equiv \underline{u}(\cdot,t) < \overline{u}(\cdot,t) \qquad \quad \text{for all } t \in [0,+\infty).
\]
We can then apply \cref{proposition_summarize_global_existence}, concluding that there exists a global classical solution \( u \) to problem \eqref{problem_HR_G}.

\medskip
\noindent Finally, under the additional assumption that \( u_0 \not \equiv 0 \), it is easy to infer that \( u \not \equiv 0 \), hence \( u \) is a nontrivial global classical solution to problem \eqref{problem_HR_G}. This yields the thesis. \hfill \( \square \)


\section{Proof of \cref{theorem_global_existence_G} and \cref{theorem_homogeneous_model_trees_existence}}\label{subsection_proofs_2}

\noindent We shall first state an auxiliary result, corresponding to {Proposition 2.4} in \cite{PSa}.

\begin{lemma}\label{lemma_Punzo_Sacco_global_existence}
\noindent Let \( (X,\omega,\mu) \) be a weighted graph, and consider the corresponding heat kernel \( p = p(x,y,t) \). Assume that \( \lambda_1(X) > 0 \) and
\[
\mu_{min} := \inf_{x \in X} \mu(x) > 0.
\]
\noindent Then the heat kernel satisfies the following inequality:
\[
p(x,y,t) \leq \frac{1}{\mu_{min}} \, e^{- \lambda_1(X) \, t} \hspace{2.5em} \text{for all } \hspace{0.1em} (x,y,t) \in X \times X \times (0,+\infty).
\]
\end{lemma}

\begin{proposition}\label{proposition_global_super_sol_general_graphs}
Let \( (X,\omega,\mu) \) be a connected and locally finite weighted graph, and assume that the weighted degree is bounded, namely \eqref{bounded_weighted_degree_2} holds.
Suppose also that \( \lambda_1(X) > 0 \), and
\[
\mu_{min} := \inf_{x \in X} \mu(x) > 0.
\]
\noindent Then, for all \( p > 1 \), equation \eqref{e31} admits a nontrivial global classical supersolution \( \overline{u} \). In addition, \( \overline{u} \) is a global mild supersolution of problem \eqref{problem_HR_G}, with initial datum \( \overline{u}(\cdot,0) \).
\end{proposition}

\noindent \emph{Proof.} For the ease of notation, throughout the proof we shall denote the bottom of the spectrum \( \lambda_1(X) \) simply by \( \lambda_1 \).

\medskip
\noindent We fix an arbitrary \( p > 1 \), and we consider a function \( v_0 \in \ell^1(X,\mu) \cap \ell^\infty(X) \), with \( v_0 > 0 \) pointwise in \( X \). From \cref{theorem_well_posedness_H_G}, we know that the function \( v : X \times [0,+\infty) \to \mathbb{R} \), defined as
\[
v(x,t) := (P_t \, v_0)(x) = \sum_{y \in X} p(x,y,t) \, v_0(y) \, \mu(y) \hspace{2.0em} \text{for all } \hspace{0.1em} (x,t) \in X \times [0,+\infty),
\]
\noindent is a bounded solution to problem \eqref{problem_H_G}.

\medskip
\noindent Under our assumptions on \( \lambda_1 \) and \( \mu_{\min} \), \cref{lemma_Punzo_Sacco_global_existence} can be applied, yielding the following bound for the heat kernel:
\[
p(x,y,t) \leq \frac{1}{\mu_{min}} \, e^{- \lambda_1 t} \hspace{2.5em} \text{for all } \hspace{0.1em} (x,y,t) \in X \times X \times (0,+\infty).
\]
\noindent Thanks to this inequality, we obtain, for all \( (x,t) \in X \times (0,+\infty) \):
\[
0 < \sum_{y \in X} p(x,y,t) \, v_0(y) \, \mu(y) \leq \frac{1}{\mu_{min}} \, e^{- \lambda_1 t} \sum_{y \in X} v_0(y) \, \mu(y) = \frac{e^{- \lambda_1 t}}{\mu_{min}} \, \| v_0 \|_1
< +\infty,
\]
\noindent where the first inequality exploits the strict positivity of \( v_0 \) and \( \mu \), together with property b) in \cref{proposition_properties_HK_G}, while the last one holds since, by assumption, \( v_0 \in \ell^1(X,\mu) \). In other words, we have:
\begin{equation}\label{inequalities_v_existence_G}
0 < v(x,t) \leq \frac{1}{\mu_{min}} \, e^{- \lambda_1 t} \, \| v_0 \|_1 \hspace{2.5em} \text{for all } (x,t) \in X \times (0,+\infty).
\end{equation}
\noindent Now, we take a scalar function \( h \), defined at least in a subset of \( [0,+\infty) \), and a positive time instant \( t_0 > 0 \). Moreover, we set
\begin{equation}\label{definition_super_solution_existence_G}
\overline{u}(x,t) := h(t) \, v(x,t+t_0).
\end{equation}
\noindent Our aim is to find the conditions to be imposed on \( t_0 \) and \( h \) which ensure that this function is a global supersolution of the equation under consideration.

\medskip
\noindent In particular, we claim that, under the following conditions:
\begin{equation}\label{conditions_h_global_super_sol_G}
\begin{gathered}
h : [0,+\infty) \to (0,C) \hspace{2.5em} \text{for some constant } C > 0, \\[7pt]
h \in C^1((0,+\infty)) \cap C^0([0,+\infty)), \\[7pt]
\frac{h'(t)}{\left[ h(t) \right]^p} = \left( \frac{\| v_0 \|_1}{\mu_{min}} \right)^{p-1} \, e^{-(p-1)\lambda_1 (t+t_0)} \hspace{2.5em} \text{for all } t > 0,
\end{gathered}
\end{equation}
\noindent the function \( \overline{u} \), given by \eqref{definition_super_solution_existence_G}, is a nontrivial global supersolution of the equation \( u_t - \Delta u = u^p \).

\medskip
\noindent We start by noticing that, if the first condition in \eqref{conditions_h_global_super_sol_G} holds, then the relation \eqref{definition_super_solution_existence_G} is valid for all \(  (x,t) \in X \times [0,+\infty) \), so that \( \overline{u} : X \times [0,+\infty) \to \mathbb{R} \) and
\begin{equation}\label{positivity_super_solution_existence_G}
\overline{u}(x,t) = h(t) \, v(x,t+t_0) > 0 \hspace{2.5em} \text{for all } (x,t) \in X \times [0,+\infty),
\end{equation}
\noindent where we have also exploited the first inequality in \eqref{inequalities_v_existence_G}. In particular, it follows from \eqref{positivity_super_solution_existence_G} that \( u \) is nontrivial.

\medskip
\noindent Furthermore, by arguing exactly as in the proof of \cref{proposition_global_super_sol_lattice}, we conclude that also \eqref{reg_sol_1_G} and \eqref{reg_sol_2_G} hold for \( \overline{u} \), with \( T = +\infty \).

\medskip
\noindent We are only left to verify that \eqref{conditions_h_global_super_sol_G} implies that
\begin{equation}\label{inequality_overline_u_existence_G}
\overline{u}_t - \Delta \overline{u} \geq \overline{u}^p \hspace{2.5em} \text{in} \hspace{0.5em} X \times (0,+\infty).
\end{equation}
\noindent In order to do so, we exploit \eqref{definition_super_solution_existence_G} in order to compute formally the partial time derivative and the discrete Laplacian of the function \( \overline{u} \). More specifically, we get both
\[
\overline{u}_t(x,t) = \frac{\partial}{\partial t} \left[ h(t) \, v(x,t+t_0) \right] = h'(t) \, v(x,t+t_0) + h(t) \, v_t(x,t+t_0)
\]
\noindent and
\[
\begin{aligned}
\Delta \overline{u} (x,t) & = \frac{1}{\mu(x)} \sum_{y \in X} \left[ u(y,t) - u(x,t) \right] \omega(x,y) \\
& = \frac{1}{\mu(x)} \sum_{y \in X} \left[ h(t) \, v(y,t+t_0) - h(t) \, v(x,t+t_0) \right] \omega(x,y) \\
& = \frac{h(t)}{\mu(x)} \sum_{y \in X} \left[ v(y,t+t_0) - v(x,t+t_0) \right] \omega(x,y) \\
& = h(t) \, \Delta v (x,t+t_0).
\end{aligned}
\]
\noindent Therefore, we obtain:
\[
\begin{aligned}
\overline{u}_t(x,t) - \Delta \overline{u} (x,t) & = h'(t) \, v(x,t+t_0) + h(t) \, v_t(x,t+t_0) - h(t) \, \Delta v (x,t+t_0) \\
& = h'(t) \, v(x,t+t_0) + h(t) \left[ v_t(x,t+t_0) - \Delta v (x,t+t_0) \right] \\
& = h'(t) \, v(x,t+t_0),
\end{aligned}
\]
\noindent where the last equality is due to the fact that \( v \) solves problem \eqref{problem_H_G}, hence in particular it holds
\[
v_t(x,t) - \Delta v (x,t) = 0 \hspace{2.5em} \text{for all } (x,t) \in X \times (0,+\infty).
\]
\noindent The inequality \eqref{inequality_overline_u_existence_G} is then equivalent to the following one:
\begin{equation}\label{equivalence_super_sol_existence_G}
h'(t) \, v(x,t+t_0) \geq \left[ h(t) \, v(x,t+t_0) \right]^p \hspace{2.5em} \text{for all } (x,t) \in X \times (0,+\infty).
\end{equation}
\noindent Now, the third requirement in \eqref{conditions_h_global_super_sol_G} yields, for all \( (x,t) \in X \times (0,+\infty) \):
\[
\frac{h'(t)}{\left[ h(t) \right]^p} = \left( \frac{\| v_0 \|_1}{\mu_{min}} \right)^{p-1} \, e^{-(p-1)\lambda_1 (t+t_0)} \equiv \left[ \frac{\| v_0 \|_1}{\mu_{min}} \, e^{- \lambda_1 (t+t_0)} \right]^{p-1}
\geq \left[ v(x,t+t_0) \right]^{p-1},
\]
\noindent where the last inequality holds thanks to the upper bound in \eqref{inequalities_v_existence_G} and to the fact that \( p > 1 \). Moreover, the first condition in \eqref{conditions_h_global_super_sol_G} implies that \( h > 0 \) in \( [0,+\infty) \), and from \eqref{inequalities_v_existence_G} we also know that \( v > 0 \) in \( X \times (0,+\infty) \). Therefore, we can multiply both sides of the obtained inequality by \( [h(t)]^p \, v(x,t+t_0) > 0 \), obtaining:
\[
h'(t) \, v(x,t+t_0) \geq \left[ h(t) \, v(x,t+t_0) \right]^p \hspace{2.5em} \text{for all } (x,t) \in X \times (0,+\infty),
\]
\noindent which exactly corresponds to \eqref{equivalence_super_sol_existence_G}. As discussed above, \eqref{inequality_overline_u_existence_G} is then obtained.

\medskip
\noindent In conclusion, the claim has been proved: under \eqref{conditions_h_global_super_sol_G}, the function \( \overline{u} \) defined in \eqref{definition_super_solution_existence_G} verifies all the requirements of a nontrivial global supersolution of the equation \( u_t - \Delta u = u^p \). In other words, the thesis follows if we exhibit a function \( h \) satisfying all the conditions in \eqref{conditions_h_global_super_sol_G}, for some \( t_0 > 0 \).

\medskip
\noindent We shall do so by coupling the third requirement in \eqref{conditions_h_global_super_sol_G} with a positive initial condition for \( h \); for the sake of simplicity, we choose \( h(0) = 1 \). Therefore, we look for a function \( h \) solving the following Cauchy problem:
\begin{equation}\label{ODE_h_existence_G}
\begin{cases}
\frac{h'(t)}{\left[ h(t) \right]^p} = \left( \frac{\| v_0 \|_1}{\mu_{min}} \right)^{p-1} \, e^{-(p-1)\lambda_1 (t+t_0)} \hspace{1.5em} \text{for all } t > 0 \\
h(0) = 1.
\end{cases}
\end{equation}
\noindent We proceed by finding the analytical expression of the solution to this problem; we will then investigate the values of \( t_0 \) for which such solution exists globally in time. More specifically, after fixing \( t > 0 \), we integrate both sides in \( (0,t) \) and exploit the initial condition, obtaining:
\[
\int_{0}^{t} \frac{h'(t)}{\left[ h(t) \right]^p} \, ds = \int_{1}^{h(t)} \frac{1}{h^p} \, dh
= \left( \frac{\| v_0 \|_1}{\mu_{min}} \right)^{p-1} \, e^{-(p-1)\lambda_{1} t_0} \int_{0}^{t} e^{-(p-1)\lambda_{1} s} \, ds,
\]
\noindent where the first equality simply corresponds to a change of variable. Performing direct computations on the last equality yields the following:
\[
\frac{\left[ h(t) \right]^{1-p} - 1}{1-p}
= \left( \frac{\| v_0 \|_1}{\mu_{min}} \right)^{p-1} \, e^{-(p-1) \lambda_{1} t_0} \, \frac{1-e^{-(p-1) \lambda_{1} t}}{(p-1) \lambda_{1}}.
\]
\noindent Rearranging the terms, we get
\[
\left[ h(t) \right]^{1-p}
= 1 - \left[ \frac{1}{\lambda_1} \, \left( \frac{\| v_0 \|_1}{\mu_{min}} \right)^{p-1} \, e^{-(p-1) \lambda_{1} t_0} \right] \left( 1-e^{-(p-1) \lambda_{1} t} \right)
\equiv 1 - D \left( 1-e^{-(p-1) \lambda_{1} t} \right),
\]
\noindent where we have set
\[
D := \frac{1}{\lambda_1} \, \left( \frac{\| v_0 \|_1}{\mu_{min}} \right)^{p-1} \, e^{-(p-1) \lambda_{1} t_0}.
\]
\noindent Therefore, the analytical expression of the solution is
\begin{equation}\label{analytical_expression_h_existence_G}
h(t) = \left[ 1 - D \left( 1-e^{-(p-1) \lambda_{1} t} \right) \right]^{- \frac{1}{p-1}}.
\end{equation}
\noindent Now, notice that the assumptions on \( v_0, \lambda_1 \) and \( \mu_{min} \) ensure that \( D \in (0,+\infty) \). In particular, the condition \( D \in (0,1) \) is equivalent to the following bound for \( t_0 \):
\[
t_0 > \frac{1}{(p-1) \lambda_{1}} \, \log \left[ \frac{1}{\lambda_1} \, \left( \frac{\| v_0 \|_1}{\mu_{min}} \right)^{p-1} \right].
\]
\noindent The term at the right-hand side of this inequality could be negative; this is the case whenever the following condition holds:
\[
0 < \| v_0 \|_1 < \mu_{\min} \, \lambda_1^{\frac{1}{p-1}}.
\]
\noindent In this scenario we have \( D \in (0,1) \) for any positive value of \( t_0 \). In conclusion, if \( t_0 \) satisfies
\begin{equation}\label{bound_t_0_existence_G}
t_0 > \max \left\{ 0, \frac{1}{(p-1) \lambda_{1}} \, \log \left[ \frac{1}{\lambda_1} \, \left( \frac{\| v_0 \|_1}{\mu_{min}} \right)^{p-1} \right] \right\},
\end{equation}
\noindent then \( D \in (0,1) \), and therefore the function \( h \) in \eqref{analytical_expression_h_existence_G} is well-defined for all nonnegative times. More specifically, if \( D \in (0,1) \), then it is trivial to see that \( h(t) > 0 \) for all \( t > 0 \); furthermore, since \( p > 1 \) and \( \lambda_1 > 0 \), then \( h \) is monotone increasing in \( [0,+\infty) \), and it exhibits an asymptote as \( t \to +\infty \), namely:
\[
\lim_{t \to +\infty} h(t) = (1 - D)^{- \frac{1}{p-1}} \hspace{0.3em} \in \hspace{0.3em} (1,+\infty).
\]
\noindent Therefore, we have:
\[
h : [0,+\infty) \to [1, (1 - D)^{- \frac{1}{p-1}}),
\]
\noindent where we have used the fact that \( h(0) = 1 \). The first condition in \eqref{conditions_h_global_super_sol_G} is then satisfied with \( C := (1 - D)^{- \frac{1}{p-1}} \). In addition, \( h \) is a global classical solution to the Cauchy problem \eqref{ODE_h_existence_G}, hence the second and the third requirements in \eqref{conditions_h_global_super_sol_G} are trivially verified.

\medskip
\noindent In conclusion, if we choose \( t_0 \) as in \eqref{bound_t_0_existence_G}, then the function \( h \) defined in \eqref{analytical_expression_h_existence_G} verifies \eqref{conditions_h_global_super_sol_G}. As already discussed, this finishes the proof of the first part of the statement.

\medskip
\medskip
\noindent Now, we explicitly remark that the function \( \overline{u}(\cdot,0) \in C(X) \) satisfies the assumptions in \eqref{eq:HP_ID_G}. Indeed, from \eqref{positivity_super_solution_existence_G} it follows that \( \overline{u}(\cdot,0) > 0 \) in \( X \). Furthermore, the fact that \( \overline{u} \in L^\infty(X \times [0,T']) \), for any \( T' > 0 \), trivially implies that \( \overline{u}(\cdot,0) \in \ell^\infty(X) \).

\medskip
\noindent We can then consider problem \eqref{problem_HR_G} with initial datum \( \overline{u}(\cdot,0) \). In order to conclude the proof of the last part of the statement, our aim is to apply \cref{lemma_mild_supersol_1}; we then proceed by verifying that all the hypotheses of such result are verified, within the present framework.

\medskip
\noindent We first notice that from \eqref{definition_super_solution_existence_G} it follows that
\begin{equation}\label{overline_u_0_general_graphs_1}
\overline{u}(\cdot,0) = h(0) \, v(\cdot,t_0) = v(\cdot,t_0) = P_{t_0} \, v_0,
\end{equation}
where we have used that \( h(0) = 1 \) and that \( v(\cdot,\tau) = P_{\tau} \, v_0 \) for all \( \tau \geq 0 \). Let us now define the function \( z : X \times [0,+\infty) \to (0,+\infty) \) as
\begin{equation}\label{definition_z_general_graphs}
z(x,t) := v(x,t+t_0) \qquad \quad \text{for all } (x,t) \in X \times [0,+\infty).
\end{equation}
The fact that \( z \) is strictly positive follows directly from \eqref{inequalities_v_existence_G}. Now, by combining \eqref{definition_z_general_graphs} with \eqref{identity_P_t_P_s} and \eqref{overline_u_0_general_graphs_1}, for all \( t \geq 0 \) we obtain:
\[
z(\cdot,t) = v(\cdot,t+t_0) = P_{t+t_0} \, v_0 = P_t \left( P_{t_0} \, v_0 \right) = P_t \, \left( \overline{u}(\cdot,0) \right).
\]
By exploiting \eqref{definition_super_solution_existence_G} we then have, for all \( t \geq 0 \):
\begin{equation}\label{condition_for_application_result_1_G}
\overline{u}(\cdot,t) = h(t) \, z(\cdot,t) \qquad \text{and} \qquad z(\cdot,t) = P_t \, \left( \overline{u}(\cdot,0) \right).
\end{equation}
Now, from \eqref{inequalities_v_existence_G} we easily infer:
\[
0 < \| v(\cdot,t) \|_\infty \leq \frac{\| v_0 \|_1}{\mu_{min}} \, e^{- \lambda_1 t} \hspace{2.5em} \text{for all } t \in (0,+\infty),
\]
so that, by exploiting \eqref{definition_z_general_graphs}, we obtain, for all \( t > 0 \):
\[
\left[ \| z(\cdot,t) \|_\infty \right]^{p-1} = \left[ \| v(\cdot,t+t_0) \|_\infty \right]^{p-1} \leq \left( \frac{\| v_0 \|_1}{\mu_{min}} \right)^{p-1} \, e^{- (p-1) \lambda_1 (t+t_0)}
= \frac{h'(t)}{\left[ h(t) \right]^p},
\]
where the last equality exploits \eqref{ODE_h_existence_G}. We can now multiply both sides of the obtained estimate by the positive value \( \left[ h(t) \right]^{p-1} \), inferring that, for all \( t > 0 \), it holds:
\begin{equation}\label{condition_for_application_result_2_G}
\frac{h'(t)}{h(t)} \geq \left[ h(t) \right]^{p-1} \left[ \| z(\cdot,t) \|_\infty \right]^{p-1} \equiv \left[ h(t) \, \| z(\cdot,t) \|_\infty \right]^{p-1}
= \left[ \| \overline{u}(\cdot,t) \|_\infty \right]^{p-1},
\end{equation}
where in the last equality we have made use of \eqref{condition_for_application_result_1_G}.

\medskip
\noindent Since the graph is assumed to be connected and locally finite, \cref{remark_properties_SG_connected_locally_finite} ensures that the heat semigroup \( \{P_t\}_{t\ge 0} \) is both positivity preserving and \( \ell^\infty \)-contractive. In addition, it is uniformly continuous, hence strongly continuous, on \( \ell^\infty(X) \).

\medskip
\noindent Finally, as already shown, we have \( h : [0,+\infty) \to (0,+\infty) \), with \( h \in C^1((0,+\infty)) \hspace{0.1em} \cap \hspace{0.1em} C^0([0,+\infty)) \) and \( h(0) = 1 \). These facts, together with \eqref{condition_for_application_result_1_G} and \eqref{condition_for_application_result_2_G}, allow us to apply \cref{lemma_mild_supersol_1}, which yields that \( \overline{u} \) satisfies the following inequality, pointwise in \( X \):
\begin{equation}\label{inequality_overline_u_1_G}
\overline{u}(\cdot,t) \ge P_t \left( \overline{u}(\cdot,0) \right) + \int_0^t \, P_{t-s} \left[ \overline{u}(\cdot,s)^p \right] \, ds \qquad \quad \text{for all } t \in (0,+\infty).
\end{equation}
Now, as shown before, we know that \( \overline{u} \in L^\infty(X \times [0,T']) \) for all \( T' > 0 \). By arguing exactly as in the last part of the proof of \cref{proposition_Duhamel_formula_HR_G}, we conclude that \( \overline{u} \in L^\infty([0,T'];\ell^\infty(X)) \), for all \( T' > 0 \). This fact, together with \eqref{inequality_overline_u_1_G}, allows us to conclude that, according to \cref{definition_mild_sol_sub_super_sol}, the function \( \overline{u} \) is a global mild supersolution to problem \eqref{problem_HR_G}, with initial datum \( \overline{u}(\cdot,0) \). This completes the proof. \hfill \( \square \)

\medskip
\noindent \emph{Proof of \emph{\cref{theorem_global_existence_G}}}. Let \( \underline{u} \equiv 0 \) in \( X \times [0,+\infty) \), and consider the function \( \overline{u} \) constructed in the proof of \cref{proposition_global_super_sol_general_graphs}. In particular, as shown there, \( \overline{u}(\cdot,0) \) is strictly positive in \( X \) and satisfies condition \eqref{eq:HP_ID_G}, so that it holds \( \overline{u}(\cdot,0) \in \ell^\infty(X) \). Therefore, by making use of \eqref{bound_ID_global_existence_G} and by exploiting the assumption \( u_0 \in C(X) \), we infer that also \( u_0 \in \ell^\infty(X) \). Now, thanks to the lower bound in \eqref{bound_ID_global_existence_G}, we conclude that \( u_0 \) respects the requirements in \eqref{eq:HP_ID_G}, and therefore is a valid initial datum for problem \eqref{problem_HR_G}.

\medskip
\noindent Under our hypotheses the heat semigroup \( \left\{ P_t \right\}_{t \geq 0} \) is positivity preserving on \( \ell^\infty(X) \), so that for all \( t > 0 \) it holds \( P_t \, u_0 \geq 0 \) in \( X \). This fact can be combined with \eqref{bound_ID_global_existence_G}, in order to conclude that \( \underline{u} \equiv 0 \) is a global mild subsolution to problem \eqref{problem_HR_G}, according to \cref{definition_mild_sol_sub_super_sol}.

\medskip
\noindent Furthermore, by exploiting both \eqref{inequality_overline_u_1_G} and \eqref{bound_ID_global_existence_G}, we deduce that for all \( t > 0 \) it holds:
\[
\overline{u}(\cdot,t) \ge P_t \left( \overline{u}(\cdot,0) \right) + \int_0^t \, P_{t-s} \left[ \overline{u}(\cdot,s)^p \right]
\geq P_t \, u_0 + \int_0^t \, P_{t-s} \left[ \overline{u}(\cdot,s)^p \right] \, ds.
\]
Finally, the fact that \( \overline{u}(\cdot,0) \geq u_0 \) in \( X \) allows us to conclude that \( \overline{u} \) is a global mild supersolution of problem \eqref{problem_HR_G}, according to \cref{definition_mild_sol_sub_super_sol}.

\medskip
\noindent Lastly, since \( \overline{u} \) is strictly positive in \( X \times [0,+\infty) \), we have:
\[
0 \equiv \underline{u}(\cdot,t) < \overline{u}(\cdot,t) \qquad \quad \text{for all } t \in [0,+\infty).
\]
We can then apply \cref{proposition_summarize_global_existence}, concluding that there exists a global classical solution \( u \) to problem \eqref{problem_HR_G}. \hfill \( \square \)

%

\bigskip

\noindent \emph{Proof of \emph{\cref{theorem_homogeneous_model_trees_existence}}}. First, by definition, \( (\mathbb{T}_b,\omega_0,\mu_1) \) is connected, and thanks to \cref{remark_model_trees_locally_finite} we know that such graph is also locally finite. In addition, in view of \eqref{e20}, $\lambda_1(\mathbb{T}_b)>0.$

\medskip
\noindent Now, as already noticed in \cref{remark_model_trees_locally_finite}, the degree function has the following expression:
\[
\deg(x) =
\begin{cases}
b & \text{ if } x = x_0 \\
b + 1 & \text{ if } x \in \mathbb{T}_b \setminus \{ x_0 \},
\end{cases}
\]
\noindent and by exploiting both the definition of the weighted degree and the fact that \( \mu_1 \equiv 1 \) on \( \mathbb{T}_b \) we infer that
\[
\operatorname{Deg}(x) = \frac{\deg(x)}{\mu_1(x)} \equiv \deg(x) \hspace{2.5em} \text{for all } x \in \mathbb{T}_b.
\]
\noindent In particular, we get:
\[
\sup_{x \in \mathbb{T}_b} \operatorname{Deg}(x) = \sup_{x \in \mathbb{T}_b} \, \deg(x) = b + 1 < +\infty.
\]
\noindent Finally, since \( \mu_1 \equiv 1 \) on \( \mathbb{T}_b \), then the infimum of the node measure is obviously positive. In conclusion, all the hypotheses of both \cref{proposition_global_super_sol_general_graphs} and \cref{theorem_global_existence_G} are satisfied, and the thesis follows. \hfill \( \square \)






\addcontentsline{toc}{section}{Bibliography}

\begin{thebibliography}{99}



\bibitem{BLev}
C. Bandle, H. A. Levine, \emph{Fujita type phenomena for reaction - diffusion equations with convection like terms}, Diff. Integral Eq. \textbf{7} (1994), 1169–1193.


\bibitem{BPT}
C. Bandle, M.A. Pozio,  A. Tesei, \emph{The Fujita exponent for the Cauchy problem in the hyperbolic space}, J. Differ. Equ. \textbf{251} (2011), 2143–2163.


\bibitem{book_Barlow}
M. T. Barlow, \emph{Random Walks and Heat Kernels on Graphs}, London Mathematical Society Lecture Note Series, vol.~438, Cambridge University Press, 2017.







\bibitem{deP}
A. De Pablo, \emph{An introduction to the problem of blow-up for semilinear and quasilinear parabolic equations}, MAT Serie A \textbf{12} (2006).



\bibitem{vCr} D.E. von Criegern, \emph{Nonexistence results for a general class of parabolic problems with a potential on weighted graphs}, Nonlinear Differ. Equ. Appl. {\bf 33}, 44 (2026).






\bibitem{FR} F. Fischer, C. Rose, \emph{Optimal Poincaré–Hardy-type inequalities on manifolds and graphs}, Indagationes Mathematicae (2025).








\bibitem{Grig1}
A. Grigor’yan, \emph{Introduction to analysis on graphs}, AMS University Lecture Series \textbf{71}, 2018.












\bibitem{GMP2}
G. Grillo, G. Meglioli, F. Punzo, \emph{Blow-up and global existence for semilinear parabolic equations on infinite graphs},
Calc. Var. Part. Diff. Eq. {\bf 65} 114 (2026).

\bibitem{HSS}
T. Hasegawa, H. Saigo, S. Saito, S. Sugiyama, \emph{Lattice sums of $I$-Bessel functions, theta functions, linear codes and heat equations}, Res. Math. Sci. \textbf{11}, 62 (2024).


\bibitem{GSXX} Q. Gu, Y. Sun, J. Xiao, F. Xu, {\it Global positive solution to a semilinear parabolic equation with
potential on Riemannian manifold}, Calc. Var. Partial Diff. Eq. {\bf 59} 170 (2020).








\bibitem{KLWb}
M. Keller, D. Lenz,  R.K. Wojciechowski, \emph{Graphs and discrete Dirichlet spaces}, Springer, 2021.









\bibitem{LW1}
Y. Lin, Y. Wu, \emph{Blow-up problems for nonlinear parabolic equations on locally finite graphs}, Acta Math. Sci. \textbf{38B} (3) (2018), 843–856.


\bibitem{LW2}
Y. Lin, Y. Wu, \emph{The existence and nonexistence of global solutions for a semilinear heat equation on graphs}, Calc. Var. Part. Diff. Eq. \textbf{56}, 102 (2017).











\bibitem{MPS2}
D. D. Monticelli, F. Punzo,  J. Somaglia, \emph{Nonexistence of solutions to parabolic problems with a potential on weighted graphs}, J. Differential Equations \textbf{453} (2026), 113782.


\bibitem{Mu} D. Mugnolo, “Semigroup Methods for Evolution Equations on Networks”, Springer (2016).

\bibitem{book_Pazy}
A. Pazy, \emph{Semigroups of Linear Operators and Applications to Partial Differential Equations}, Applied Mathematical Sciences, vol.~44, Springer-Verlag, New York, 1983.






\bibitem{PSa}
F. Punzo, A. Sacco, \emph{On a semilinear parabolic equation with time-dependent source term on infinite graphs},  J. Evol. Equ. {\bf 26} 13 (2026).

\bibitem{PZ1} F. Punzo, F. Zucchero, {\it On a semilinear heat equation on infinite graphs I: blow-up for large initial data}, preprint (2026).



\bibitem{book_Stein_Shakarchi}
R. Shakarchi, E.M. Stein, \emph{Fourier Analysis: An Introduction}, Princeton Lectures in Analysis \textbf{1}, Princeton University Press, Princeton, NJ, 2003.


\bibitem{W}
L. F. Wang, \emph{Heat kernel and monotonicity inequalities on the graph}, J. Geom. Anal. \textbf{33}, 38 (2023).




\bibitem{Weiss1} F.B. Weissler, {\it $L^p$ energy and blow-up for a semilinear heat equation}, Proc. Sympos. Pure Math. {\bf 45} (1986) 545-551.

\bibitem{Weiss2}
F.B. Weissler, \emph{Existence and nonexistence of global solutions for a semilinear heat equation}, Israel J. Math. \textbf{38} (1981), 29–40.


\bibitem{Woj}
R. K. Wojciechowski, \emph{Heat kernel and essential spectrum of infinite graphs}, Indiana Univ. Math. J. \textbf{58}, no. 3 (2009), 1419–1441.


\bibitem{Wu}
Y. Wu, \emph{On nonexistence of global solutions for a semilinear heat equation on graphs}, Nonlinear Anal. \textbf{171} (2018), 73–84.






%
%
%
%
%
%
%
%
%
%
%
%
%
%
%
%
%
%
%
%
%
%
%
%
%
%
%
%
%
%
%
%
%
%
%
%
%
%
%
%
%
%
%
%
%
%
%
%
%
%
%
%
%
%
%
%
%
%
%
%
%
%
%
%
%
%
%
%
%
%
%
%
%
%
%
%
%
%
%
%
%
%
%
%
%
%
%
%
%
%
%
%
%
%
%


\end{thebibliography}
\end{document}